\documentclass[12pt, leqno]{amsart}

\usepackage{lscape}
\usepackage{graphicx}
\usepackage{amssymb,amsmath,amsthm,amscd}
\usepackage{mathrsfs}
\usepackage{enumerate}
\usepackage[usenames,dvipsnames]{color}
\usepackage[colorlinks=true, pdfstartview=FitV,  linkcolor=blue,citecolor=blue,urlcolor=blue, bookmarks=false]{hyperref}
\usepackage[all]{xy}
\allowdisplaybreaks[4]
\usepackage{oldgerm}

\newcommand{\nc}{\newcommand}
\nc{\cmt}{\fbox{comment}}
\numberwithin{equation}{section}

\newcommand{\bj}{\begin{jaune}}
\newcommand{\ej}{\end{jaune}}

\theoremstyle{plain}
\newtheorem*{theorem*}{Theorem}
\newtheorem{lemma}{Lemma}[subsection]
\newtheorem{prop}[lemma]{Proposition}
\newtheorem{theorem}[lemma]{Theorem}
\newtheorem*{maintheorem1}{Main Theorem 1}
\newtheorem*{maintheorem2}{Main Theorem 2}
\newcommand{\Prop}{\begin{prop}}
\newcommand{\enprop}{\end{prop}}
\newcommand{\Lemma}{\begin{lemma}}
\newcommand{\enlemma}{\end{lemma}}
\newcommand{\Th}{\begin{theorem}}
\newcommand{\enth}{\end{theorem}}
\newtheorem{corollary}[lemma]{Corollary}
\newcommand{\Cor}{\begin{corollary}}
\newcommand{\encor}{\end{corollary}}
\newtheorem{definition}[lemma]{Definition}
\newtheorem*{conjecture}{Conjecture}
\newcommand{\Def}{\begin{definition}}
\newcommand{\edf}{\end{definition}}
\newtheorem{sublemma}[lemma]{Sublemma}
\newcommand{\Sub}{\begin{sublemma}}
\newcommand{\ensub}{\end{sublemma}}

\theoremstyle{definition}
\newtheorem{remark}[lemma]{Remark}

\newtheorem{Convention}[lemma]{Convention}
\newcommand{\Conv}{\begin{Convention}}
\newcommand{\enconv}{\end{Convention}}
\nc{\Con}{\begin{conjecture}}
\nc{\encon}{\end{conjecture}}
\nc{\Rem}{\begin{remark}}
\nc{\enrem}{\end{remark}}

\newcommand{\C}{{\mathbb C}}
\newcommand{\Q}{\mathbb {Q}}

\newcommand{\Z}{{\mathbb Z}}
\newcommand{\B}{{\mathbf{B}}}
\newcommand{\A}{{\mathbb Z [q^{\pm1}]}}

\newcommand{\Db}[1][G]{{\operatorname{D^b_{#1}}}}

\newcommand{\R}{{\rm R}}

\newcommand{\one}{{\bf{1}}}
\newcommand{\seteq}{\mathbin{:=}}

\newcommand{\hd}{{\operatorname{hd}}}
\newcommand{\soc}{{\operatorname{soc}}}

\newcommand{\g}{\mathfrak{g}}
\newcommand{\h}{\mathfrak{h}}
\newcommand{\n}{\mathfrak{n}}

\newcommand{\Uq}[1][{\mathfrak{g}}]{{U_q(#1)}}

\newcommand{\Hom}{\operatorname{Hom}}
\newcommand{\End}{\operatorname{End}}

\newcommand{\isoto}[1][]{\mathop{\xrightarrow%
[{\raisebox{.3ex}[0ex][.3ex]{$\scriptstyle{#1}$}}]%
{{\raisebox{-.6ex}[0ex][-.6ex]{$\mspace{2mu}\sim\mspace{2mu}$}}}}}
\newcommand{\tensor}{\otimes}

\newcommand{\M}{\mathsf{M}}
\newcommand{\D}{\mathrm{D}}
\nc{\N}{\Z_{\ge0}}
\newcommand{\eq}{\begin{eqnarray}}
\newcommand{\eneq}{\end{eqnarray}}

\newcommand{\hs}{\hspace*}

\newcommand{\To}[1][{\hs{2ex}}]{\xrightarrow{\,#1\,}}

\newcommand{\eqn}{\begin{eqnarray*}}
\newcommand{\eneqn}{\end{eqnarray*}}
\newcommand{\on}{\operatorname}
\newcommand{\Ker}{\on{Ker}}

\newcommand{\bni}{\be[{\rm(i)}]}
\newcommand{\bna}{\be[{\rm(a)}]}

\newcommand{\QED}{\end{proof}}
\newcommand{\Proof}{\begin{proof}}

\newcommand{\soplus}{\mathop{\mbox{\normalsize$\bigoplus$}}\limits}
\newcommand{\sodot}{\mathop{\mbox{\normalsize$\bigodot$}}\limits}
\newcommand{\sotimes}{\mathop{\mbox{\normalsize$\bigotimes$}}\limits}

\newcommand{\id}{\on{id}}
\newcommand{\ba}{\begin{array}}
\newcommand{\ea}{\end{array}}

\newcommand{\monoto}{\rightarrowtail}

\newcommand{\set}[2]{\left\{#1 \mid #2 \right\}}
\newcommand{\supp}{\operatorname{supp}}

\newcommand{\eqsub}{\begin{subequations}\begin{eqnarray}}
\newcommand{\eneqsub}{\end{eqnarray}\end{subequations}}

\newcommand{\ol}{\overline}

\newcommand{\ko}{{{\mathbf{k}}}}

\nc{\la}{\lambda}
\nc{\lam}{\lambda}
\nc{\U}[1][\g]{U_q(#1)}
\nc{\te}{\tilde{e}}
\nc{\tei}{\tilde{e}_i}
\nc{\tf}{\tilde{f}}
\nc{\tfi}{\tilde{f}_i}
\nc{\tU}{\widetilde U_q(\g)}
\nc{\tE}{\tilde{E}}
\nc{\tF}{\widetilde{F}}
\nc{\tK}{\widetilde{K}}

\nc{\tk}{\tilde{k}}
\nc{\tkone}{\tk_{\ol{1}}}
\nc{\teone}{\tilde{e}_{\ol{1}}}
\nc{\tfone}{\tilde{f}_{\ol{1}}}

\nc{\teibar}{\tilde{e}_{\ol{i}}} \nc{\tfibar}{\tilde{f}_{\ol{i}}}
\nc{\tki}{{\tk}_{\ol {i}}}

\nc{\BZ}{{\mathbb{Z}}}
\nc{\al}{\alpha}
\nc{\qs}{{q}}
\nc{\lan}{\langle}
\nc{\ran}{\rangle}
\nc{\re}{{\mathrm{re}}}
\nc{\wt}{\operatorname{wt}}
\nc{\ch}{\operatorname{ch}}
\nc{\Um}[1][\g]{U^-_q(#1)}
\nc{\Ue}{U^+_q(\g)}
\nc{\eps}{\varepsilon}
\nc{\vphi}{\varphi}
\nc{\sphi}{\varphi^*}
\nc{\seps}{\varepsilon^*}

\nc{\nn}{\nonumber}
\def\max{{\mathop{\mathrm{max}}}}
\nc{\vph}{\varphi}
\nc{\cls}{{\operatorname{cl}}}
\nc{\Wt}{{\operatorname{Wt}}}
\nc{\Us}{U'_q(\g)}
\nc{\La}{\Lambda}
\nc{\tLa}{\widetilde\Lambda}
\nc{\ro}{{\rm(}}
\nc{\rf}{{\rm)}}
\nc{\norm}{{\mathrm{norm}}}
\nc{\qbox}{\quad\mbox}
\nc{\braid}{{\mathfrak{B}}}
\nc{\Ad}{\operatorname{Ad}}
\nc{\Aut}{\operatorname{Aut}}
\nc{\dt}[1]{\tilde{\tilde #1}}
\nc{\Sn}{S^{{\mathrm{norm}}}}
\nc{\aff}{{\rm{aff}}}
\nc{\rk}{{\mathrm{rk}}}
\nc{\tP}{\widetilde{P}}
\nc{\tW}{\widetilde{W}}
\nc{\Dyn}{\mathrm{Dyn}}
\nc{\tD}{\widetilde{\Delta}}
\nc{\height}[1]{{\operatorname{ht}}(#1)}
\nc{\bl}{\bigl(}
\nc{\br}{\bigr)}
\nc{\Hecke}{\mathrm{H}}
\nc{\HA}{\Hecke^{\mathrm{A}}}
\nc{\HB}{\Hecke^{\mathrm{B}}}
\newcommand{\scbul}{{\,\raise1pt\hbox{$\scriptscriptstyle\bullet$}\,}}
\nc{\vac}{{\phi}}
\nc{\Bt}{\B_\theta(\g)}
\nc{\be}{\begin{enumerate}}
\nc{\ee}{\end{enumerate}}
\nc{\low}{{\mathrm{low}}}
\nc{\upper}{{\mathrm{up}}}
\nc{\Zodd}{\Z_{\mathrm{odd}}}
\nc{\Ft}[1][n]{\mathbb{P}\mathrm{ol}_{#1}}
\nc{\Ftf}[1][n]{\widetilde{\mathbb{P}\mathrm{ol}}_{#1}}
\nc{\KA}{\on{K}^{\mathrm{A}}}
\nc{\KB}{\on{K}^{\mathrm{B}}}
\nc{\Res}{\on{Res}}
\nc{\Fc}[1][{n,m}]{\mathbf{F}_{#1}}
\nc{\tphi}{\tilde{\varphi}}
\nc{\CO}{\mathscr{O}}
\nc{\inte}{\mathrm{int}}
\newcommand{\Oint}{\mathcal{O}_{{\rm int}}}
\nc{\vs}{\vspace*}
\nc{\tLt}{\widetilde{L}}
\nc{\tL}{\widetilde{\Lambda}}
\nc{\tu}{\tilde{u}}
\nc{\noi}{\noindent}
\nc{\heigh}{\mathfrak{t}}
\nc{\lowest}{\mathfrak{l}}
\nc{\rootl}{\mathsf{Q}}
\nc{\cl}{\colon}
\nc{\uqpg}{U'_q(\mathfrak g)}
\nc{\uq}{\uqpg}
\nc{\Oh}{\widehat{\mathcal{O}}}

\nc{\KLR}{KLR algebra}
\nc{\KLRs}{KLR algebras}
\nc{\cor}{\mathbf{k}}
\nc{\cora}{{\cor(A)}}
\nc{\haut}{\mathrm{ht}}
\nc{\tens}{\mathop\otimes}
\nc{\gmod}{\mbox{-$\mathrm{gmod}$}}
\nc{\gMod}{\mbox{-$\mathrm{gMod}$}}
\nc{\proj}{\mbox{-$\mathrm{proj}$}}
\nc{\gproj}{\mbox{-$\mathrm{gproj}$}}
\nc{\smod}{\mbox{-$\mathrm{mod}$}}
\nc{\Mod}{\mbox{-$\mathrm{Mod}$}}
\nc{\Rnorm}{R^{\rm{norm}}}

\nc{\Vhat}{\widehat{V}}
\nc{\F}{\mathcal{F}}

\def\T{{\mathcal T}}

\nc{\fd}[1][A]{\on{\mathrm{flat.dim}_{#1}}}
\nc{\bP}{{\mathbb{P}}}
\nc{\bPh}{\widehat{\mathbb{P}}}
\nc{\bK}[1][{n}]{\widehat{\mathbb{K}}_{#1}}
\nc{\bV}[1][{n}]{\widehat{V}^{\otimes{#1}}}
\nc{\bVK}[1][{n}]{\widehat{V}^{\otimes{#1}}_{\widehat{\mathbb{K}}}}
\nc{\hV}{\widehat{V}}
\nc{\opp}{\mathrm{opp}}
\nc{\col}{\colon}
\nc{\bnum}{\be[{\rm(i)}]}
\nc{\bnam}{\be[{\rm(a)}]}
\nc{\oep}{\epsilon}
\nc{\qtext}[1][{and}]{\quad\text{#1}\quad}

\nc{\qtextq}[1]{\quad\text{#1}\quad}
\nc{\longtwoheadrightarrow}[1][]{\xymatrix{\ar@{->>}[r]^-{{#1}}&}}
\nc{\epiTo}[1][]{\longtwoheadrightarrow[{#1}]}
\nc{\epito}{\twoheadrightarrow}
\nc{\monoTo}[1][]{\xymatrix{\ar@{>->}[r]^-{{#1}}&}}
\nc{\sym}{\mathfrak{S}}
\nc{\inp}[1]{{({#1})_{\mathrm{n}}}}
\nc{\rtl}{\rootl}
\nc{\wtd}{\widetilde}
\nc{\etens}{\boxtimes}
\nc{\ds}[1]{\mathrm{d}(#1)}
\nc{\rmat}[1]{{\mathbf{r}}_%
{\mspace{-2mu}\raisebox{-.6ex}{${\scriptstyle{#1}}$}}}
\nc{\rmats}[1]{{\mathbf{r}}_%
{\mspace{-2mu}\raisebox{-.6ex}{${\scriptscriptstyle{#1}}$}}}
\nc{\shc}{\mathcal{C}}
\nc{\shs}{\mathcal{S}}
\nc{\Fct}{{\on{Fct}}}
\nc{\tC}{\widetilde{\shc}}
\nc{\Zp}{\Z_{\ge0}}
\nc{\tPhi}{\widetilde{\Phi}}
\nc{\tT}{{\widetilde{\T}}}
\nc{\Ob}{\on{Ob}}
\nc{\bwr}{\mbox{\large$\wr$}}
\nc{\Img}{\on{Im}}
\nc{\Ab}{\mathcal{A}^{\mathrm{big}}}
\nc{\Sb}{\mathcal{S}^{\mathrm{big}}}
\nc{\As}{\mathcal{A}}
\nc{\Ss}{\mathcal{S}}
\nc{\ntens}{\widetilde{\otimes}}
\nc{\hR}{\widehat{R}}
\nc{\nconv}{\mathop{\mbox{\large $\odot$}}}
\nc{\snconv}{\mbox{\scriptsize$\odot$}}
\nc{\ts}{\tilde{s}}
\nc{\sho}{\mathcal{O}}
\nc{\bc}{\begin{cases}}
\nc{\ec}{\end{cases}}
\nc{\slnh}{{\widehat{\mathfrak{sl}}_N}}
\nc{\UA}{U_q'(\slnh)}
\nc{\cQ}{\mathcal{Q}}
\nc{\Irr}{\mathcal{I}rr}
\nc{\tQ}{\widetilde{\cQ}}
\nc{\bs}{\mathbf{s}}
\nc{\bL}{\mathbb{L}}
\nc{\tg}{\tilde{g}}

\nc{\conv}{\mathbin{\mbox{\large $\circ$}}}
\nc{\Rm}{R^{\mathrm{ren}}}

\nc{\bQ}{\ol{Q}}
\renewcommand{\Im}{\on{Im}}

\nc{\de}{\on{\textfrak{d}}}

\nc{\xmono}{\ar@{>->}}
\nc{\xepi}{\ar@{->>}}
\nc{\db}[1]{\raisebox{-.5ex}[2ex][1.8ex]{$#1$}}
\nc{\wb}[1]{\mbox{$\rule[-1.1ex]{0ex}{2ex}#1$}}
\nc{\univ}{\mathrm{univ}}
\nc{\rM}{{}^*\mspace{-2mu}M}
\nc{\lM}{M^*}
\nc{\uqm}{\uq\smod}
\nc{\tR}{\widetilde{R}_{\gamma,\beta}}
\nc{\tx}{\tilde{x}}
\nc{\bi}{\mathbf{i}}
\nc{\ttau}{\widetilde{\tau}}

\nc{\tEnd}{\on{\widetilde{E}nd}}
\nc{\tHom}{\on{\widetilde{H}om}}

\nc{\K}{{J}}
\nc{\Kex}{{\K}_{\mathrm{ex}}}
\nc{\Kfr}{{\K}_{\mathrm{f\mspace{.01mu}r}}}
\nc{\coro}{\cor}
\nc{\tB}{\widetilde{B}}
\nc{\seed}{\mathscr{S}}
\nc{\up}{\mathrm{up}}
\nc{\bfa}{\mathbf{a}}
\newcommand{\wB}{\widetilde{B}}

\newcommand{\Uqg}{U_q(\g)}
\newcommand{\tUqg}{\widetilde{U}_q(\g)}

\newcommand{\im}{\mathrm{Im}}
\newcommand{\Seed}{\mathscr{S}}

\newcommand{\rl}{\mathsf{Q}}   % root lattice
\newcommand{\wl}{\mathsf{P}}   % weight lattice
\newcommand{\cmA}{\mathsf{A}}  % Cartan matrix
\newcommand{\comp}{\Delta_+}
\newcommand{\comm}{\Delta_-}

\newcommand{\hconv}{\mathbin{\mbox{$\nabla$}}}
\newcommand{\shconv}{\mathbin{\large\nabla}}
\newcommand{\sconv}{\mathbin{\mbox{$\Delta$}}}

\nc{\tensp}{\otimes_{_+}\mspace{-1mu}}
\nc{\tensm}{\otimes_{_-}\mspace{-1mu}}

\newcommand{\ve}{\varepsilon}
\newcommand{\ra}{\rangle}
\newcommand{\laa}{\langle}
\newcommand{\raa}{\rangle}%_{ \hspace{-0.5ex} _\times}}

\newcommand{\vpi}{{\varpi_i}}

\newcommand{\ex}{\mathrm{ex}}
\newcommand{\fr}{\mathrm{fr}}

\newcommand{\ri}{{\mspace{1mu}\rm r}}
\newcommand{\li}{{\rm l}}
\newcommand{\Lto}{\longrightarrow}

\newcommand{\Dv}{\mathbf{D}_\varphi}

\newcommand{\oi}{\overline{\iota}}
\nc{\bg}{{\oi_\g}}

\nc{\An}{A_q(\n)}
\nc{\tEs}{\widetilde{E}^*}
\def\max{{\mathop{\mathrm{max}}}}
\nc{\pn}{p_\n}
\nc{\dP}{\mathrm{E}^*}
\nc{\Up}{U_q^+(\g)}
\nc{\Ag}{A_q(\g)}
\nc{\QA}{\mathbf{A}}
\nc{\Pd}{\wl^+}
\nc{\Po}{\wl}
\nc{\De}[1]{\Delta(#1)}
\nc{\rt}{\ri}
\nc{\prtl}{\rtl_+}
\nc{\nrtl}{\rtl_-}
\nc{\lt}{\mathrm{l}}
\nc{\wtl}{\wt_\lt}
\nc{\wtr}{\wt_\rt}
\nc{\Cmp}{\comp}
\nc{\Cmm}{\comm}
\nc{\Cm}{\Delta}
\newlength{\mylength}
\title
{Monoidal categorification of cluster algebras}

\author[S.-J. Kang, M. Kashiwara, M. Kim, S.-j. Oh]{Seok-Jin Kang,
Masaki Kashiwara$^{1}$,  Myungho Kim$^{2}$ and Se-jin Oh$^{3}$}

\address{
Research Institute of Computers, Information and Communication \\
Pusan National University \\
2, Busandaehak-ro
Pusan 46241, Korea}
         \email{soccerkang@hotmail.com}

\address{Research Institute for Mathematical Sciences\\
          Kyoto University\\ Kyoto 606-8502, Japan}

         \email{masaki@kurims.kyoto-u.ac.jp}

\address{Department of Mathematics, Kyung Hee University \\ Seoul 02447, Korea}
         \email{mkim@khu.ac.kr}

\address{Department of Mathematics Ewha Womans University \\Seoul 03760, Korea}
\email{sejin092@gmail.com}
\thanks{$^1$ This work was supported by Grant-in-Aid for
Scientific Research (B) 22340005, Japan Society for the Promotion of
Science.}
\thanks{$^2$ This work was supported by the National Research Foundation of
Korea(NRF) grant funded by the Korea government(MSIP) (No. NRF-2017R1C1B2007824).}

\thanks{$^3$ This work was supported by  NRF Grant \# 2016R1C1B2013135.}


\date{\today}
\begin{document}

\begin{abstract}
We prove that the quantum cluster algebra structure of a unipotent quantum coordinate ring $A_q(\n(w))$,
associated with a symmetric Kac-Moody algebra and
its Weyl group element $w$,
admits a monoidal categorification via the representations of symmetric Khovanov-Lauda- Rouquier algebras.
In order to  achieve  this goal,
we  give a formulation of monoidal categorifications of quantum cluster algebras and
 provide a criterion for a monoidal category of
finite-dimensional graded $R$-modules to become a monoidal
categorification,
where $R$ is a symmetric Khovanov-Lauda-Rouquier algebra.
Roughly speaking, this criterion asserts that
a quantum monoidal seed can be mutated successively in all the directions,
once the first-step mutations are possible.
Then,
we show the
existence of a quantum monoidal seed of $A_q(\n(w))$
which admits the first-step mutations in all the directions.
As a consequence, we prove  the conjecture that
any cluster monomial is a member of the
upper global basis up to a power of $q^{1/2}$.
In the course of our investigation, we also give a proof of
a conjecture of Leclerc on
the product of upper global basis elements.
\end{abstract}
\maketitle
\tableofcontents

\section*{Introduction}

The purpose of this paper is to provide a monoidal categorification of
the quantum cluster algebra structure on the unipotent quantum coordinate ring $A_q(\n(w))$,
which is associated with a symmetric Kac-Moody algebra $\g$ and a Weyl group element $w$.

The notion of cluster algebras was introduced by Fomin and Zelevinsky in
\cite{FZ02} for studying total positivity and upper global bases.
Since their introduction,  a lot of  connections and applications have been
discovered in various fields of mathematics including
representation theory, Teichm\"uller theory, tropical geometry, integrable
systems, and Poisson geometry.

A cluster algebra is a $\Z$-subalgebra of a rational function field
given by a set of generators,
called the {\it cluster variables}. These generators are grouped
into overlapping subsets, called  the {\it clusters}, and the clusters are defined
inductively by a procedure called {\it mutation} from the {\it
initial cluster} $\{ X_i\}_{1 \le i \le r}$, which is controlled by
an exchange matrix $\wB$. We call a monomial of cluster
variables in   each cluster {\it a cluster monomial}.

Fomin and Zelevinsky proved that every cluster variable is a Laurent
polynomial of the initial cluster $\{ X_i\}_{1 \le i \le r}$ and
they conjectured that this Laurent polynomial has positive
coefficients (\cite{FZ02}). This {\it positivity conjecture} was
proved by Lee and Schiffler in the {\it skew-symmetric} cluster algebra case
in \cite{LS13}.
The {\it linearly independence
conjecture} on cluster monomials was proved
in the skew-symmetric cluster algebra case
in \cite{CKLP12}.

The notion of quantum cluster algebras, introduced by Berenstein and Zelevinsky in
\cite{BZ05}, can be considered as a $q$-analogue of cluster algebras.
The commutation relation among the cluster variables is determined by
a skew-symmetric matrix $L$.
As in the cluster algebra case, every cluster variable belongs to
$\Z[q^{\pm 1/2}][X_i^{\pm 1}]_{1 \le i \le r}$ (\cite{BZ05}), and is
expected to be an element of $\Z_{\ge0}[q^{\pm 1/2}][X_i^{\pm 1}]_{1 \le i \le r}$,
which is referred to as the {\it quantum
positivity conjecture} (cf.\ \cite[Conjecture 4.7]{DMSS}).
In \cite{KQ14}, Kimura and Qin proved the quantum positivity conjecture for quantum cluster algebras containing {\it acyclic} seed
and specific coefficients.

The \emph{unipotent quantum coordinate rings} $\An$ and $A_q(\mathfrak n(w))$
are  examples of quantum cluster algebras arising from Lie theory.
The algebra $\An$ is a $q$-deformation of the coordinate ring $\C[N]$ of the unipotent subgroup,
and is  isomorphic to the negative half $\Um$ of the quantum group as $\Q(q)$-algebras.
The algebra $A_q(\mathfrak n(w))$ is  a $\Q(q)$-subalgebra of $\An$ generated by
a set of the \emph{dual PBW basis elements} associated with a Weyl group element $w$.
The unipotent quantum coordinate ring $\An$  has a very interesting basis so called   the
{\em upper global basis} (dual canonical basis)
$\mathbf B^\up$ , which is  dual to the lower global basis (canonical basis) (\cite{Kash91, Lusz90}).
The upper global basis has been studied
 emphasizing on its multiplicative structure.
For example, Berenstein and Zelevinsky (\cite{BZ93})
conjectured that,
 in the case $\g$ is of type $A_n$,
the product $b_1 b_2$ of two elements $b_1$ and
$b_2$ in $\mathbf B^\up$ is again an element of $\mathbf B^\up$ up to a
multiple of a power of $q$
 if and only if they are $q$-commuting;
 i.e., \ $b_1b_2=q^m b_2b_1$ for some $m\in\Z$.
This conjecture turned out to be not true in general,
because Leclerc (\cite{L03}) found examples of
an {\em imaginary}  element $b \in \mathbf B^\up$
such that $b^2$ does not belong to $\mathbf B^\up$.
Nevertheless, the idea of considering  subsets of $\mathbf B^\up$
whose elements are $q$-commuting with each other and
 studying the relations between those subsets has survived
 and became one of the motivations  of the study of  (quantum) cluster algebras.

In a series of papers \cite{GLS11,GLS07,GLS}, Gei\ss, Leclerc and
Schr{\"o}er showed that the unipotent  quantum coordinate ring
$A_q(\mathfrak{n}(w))$
has a skew-symmetric quantum cluster algebra structure whose initial cluster
consists of so called the {\it unipotent quantum minors}. In \cite{Kimu12}, Kimura
proved that $A_q(\mathfrak{n}(w))$ is {\it compatible} with the
upper global basis $\B^{\upper}$ of $A_q(\mathfrak{n})$; i.e., the
set $\B^{\upper}(w) \seteq A_q(\mathfrak{n}(w)) \cap \B^{\upper}$
is a basis of $A_q(\mathfrak{n}(w))$. Thus, with a
result of \cite{CKLP12}, one can expect that every cluster monomial
of $A_q(\mathfrak{n}(w))$ is contained in the upper global basis
$\B^{\upper}(w)$, which is named the \emph{quantization conjecture} by Kimura
(\cite{Kimu12}):

\begin{conjecture} [{\cite[Conjecture 12.9]{GLS}, \cite[Conjecture 1.1(2)]{Kimu12}}] \label{conj:intro}
When $\g$ is a symmetric Kac-Moody algebra,  every quantum cluster monomial in
$A_{q^{1/2}}(\n(w))\seteq\Q(q^{1/2})\otimes_{\Q(q)}A_q(\n(w))$ belongs to the upper global basis $\mathbf B^\up$ up to a power of $q^{1/2}$.
\end{conjecture}

It can be regarded as a reformulation of Berenstein-Zelevinsky's ideas on the multiplicative properties of $\B^\up$.
There are some partial results of this conjecture. It is proved
 for $\g= A_2$, $A_3$, $A_4$ and $ A_q(\n(w))=\An$ in \cite{BZ93} and
 \cite[\S\,12]{GLS05}.
 When $\g=A^{(1)}_1$,   $A_n$ and $w$ is a square of a Coxeter element,  it is shown in \cite{Lampe11} and \cite{Lampe14} that the cluster variables belong to the upper global basis.
   When $\g$ is symmetric and $w$ is a square of  a Coxeter element, the conjecture  is proved
 in \cite{KQ14}.
Notably,  Qin provided recently  a proof of the conjecture for a large class with a condition on the Weyl group element $w$ (\cite{Qin15}).
Note that Nakajima proposed a geometric approach of this conjecture via  quiver varieties (\cite{Nak13}).

In this paper, we  prove the above conjecture completely by showing that there exists a
{\em  monoidal categorification of $A_{q^{1/2}}(\n(w))$}.

In \cite{HL10}, Hernandez and Leclerc introduced the notion of
{\it monoidal categorification} of   cluster algebras.
A simple object $S$ of a monoidal category
$\shc$ is {\it real} if $S \tens S$ is simple, and
 is {\it prime} if there exists no non-trivial
factorization $S \simeq S_1 \tens S_2$.
They say that $\shc$ is
a monoidal categorification of a cluster algebra $A$ if
the Grothendieck ring of $\shc$ is isomorphic to
$A$ and if

\vs{1.5ex}
\hs{0ex}\parbox{80ex}{

\begin{enumerate}
\setlength{\itemsep}{3pt}
\item[{\rm (M1)}] the  cluster monomials of $A$ are the classes of real simple objects of $\shc$,
\item[{\rm (M2)}] the cluster variables of $A$ are the classes of real simple prime objects of $\shc$.
\end{enumerate}}

\vs{1.5ex}
\noi
(Note that the above version is
weaker than the original definition of the monoidal categorification in
\cite{HL10}.) They proved that certain categories
of modules over symmetric quantum affine algebras
$U_q'(\g)$ give monoidal categorifications of  some cluster algebras.
Nakajima extended  this result  to the cases of the cluster algebras of type $A,D,E$
(\cite{Nak11}) (see also \cite{HL13}).
It is worth to remark that once a cluster algebra $A$ has a monoidal categorification,
the positivity of cluster
variables of $A$ and the linear independence of cluster monomials of $A$ follow
(see \cite[Proposition 2.2]{HL10}).

In this paper, we  refine Hernandez-Leclerc's notion of
monoidal categorifications including the quantum cluster algebra case.
Let us briefly explain it.
Let  $\mathcal C$  be an
abelian monoidal category equipped with
 an auto-equivalence  $q$ and
a tensor product which is compatible with a decomposition
$\shc=\soplus\nolimits_{\beta \in \rtl}\shc_\beta$.
Fix a finite index set $\K=\K_\ex \sqcup \K_\fr$ with a decomposition into
the exchangeable part and the frozen part.
Let $\Seed$ be a quadruple $(\{M_i\}_{i \in\K}, L,\wB, D)$
of a family of simple objects $\{M_i\}_{i \in\K}$
in $\mathscr C$, an  integer-valued skew-symmetric $\K \times \K$-matrix $L=(\lambda_{i,j})$,
an integer-valued $\K \times \K_\ex$-matrix $\wB = (b_{i,j})$
with skew-symmetric principal part, and a family of elements  $D=\{d_i\}_{i \in\K}$ in $\rl$.
If this datum satisfies the conditions in
Definition \ref{def:quantum monoidal seed} below, then it is called a {\em quantum monoidal seed} in $\mathcal C$.
For each $k\in\K_\ex$, we have mutations $\mu_k(L),\mu_k(\widetilde B)$ and $\mu_k(D)$ of $L,\widetilde{B}$ and $D$, respectively.
We say that a quantum monoidal seed
$\mathscr S =(\{M_i\}_{i\in\K}, L,\widetilde B, D)$
 \emph{admits a mutation in direction $k\in\K_\ex$} if
there exists  a simple object  $M_k' \in \shc_{\mu_k(D)_k}$
which fits into two short exact sequences \eqref{eq:intro} below
in $\mathcal C$
{\em reflecting} the mutation rule in quantum cluster algebras, and thus obtained
 quadruple $\mu_k(\Seed)\seteq(\{M_i\}_{i\neq k}\cup\{M_k'\},\mu_k(L), \mu_k(\widetilde B), \mu_k(D))$
is again a quantum monoidal seed in $\shc$.
We call $\mu_k(\Seed)$ the mutation of $\Seed$ in direction $k\in\K_\ex$.

Now the category $\shc$ is called a {\em monoidal categorification of a quantum cluster algebra $A$ over $\Z[q^{\pm1/2}]$}
if
\begin{eqnarray} \label{cond:monoidal}
\parbox{72ex}{ \bnum
\item the Grothendieck ring $\Z[q^{\pm1/2}]\tens_{\Z[q^{\pm1}]} K(\shc)$ is isomorphic to $A$,
\item there exists a quantum monoidal seed
$\mathscr S =(\{M_i\}_{i\in\K}, L,\widetilde B, D)$ in $\shc$ such that
$[\mathscr S]\seteq(\{q^{m_i}[M_i]\}_{i\in\K}, L, \widetilde B)$
 is a quantum seed of $A$ for some $m_i \in \frac{1}{2}\Z$,
\item $\mathscr S$ admits successive mutations in all directions in $\K_\ex$.
\ee }
\end{eqnarray}
The existence of monoidal category $\shc$ which provides a
monoidal categorification of quantum cluster algebra $A$ implies the
 following: \eqn&&\parbox{80ex}{ \begin{enumerate}
\item[{\rm (QM1)}] Every quantum cluster monomial corresponds to the isomorphism class of a real simple object of $\shc$. In particular,
the set of quantum cluster monomials is $\Z[q^{\pm 1/2}]$-linearly independent.
\item[{\rm (QM2)}] The quantum positivity conjecture holds for $A$.
\end{enumerate}
} \eneqn

\medskip
In  the case of unipotent quantum coordinate ring $A_q(\n)$, there is a natural candidate for monoidal categorification,
 the category of finite-dimensional graded modules over a {\em Khovanov-Lauda-Rouquier algebras} (\cite{KL09,KL11}, \cite{R08}).
The Khovanov-Lauda-Rouquier algebras (abbreviated by KLR algebras), introduced by
Khovanov-Lauda \cite{KL09,KL11} and Rouquier \cite{R08}
independently, are a family of $\Z$-graded algebras which
categorifies the negative half $U_q^-(\g)$ of a {\it symmetrizable}
quantum group $U_q(\g)$. More precisely, there exists a family of
algebras $\{ R(-\beta) \}_{\beta \in \rtl^-}$ such that the Grothendieck
ring of $R \gmod \seteq \bigoplus_{\beta \in \rtl^-}R(-\beta)\gmod$, the direct sum
of the categories of finite-dimensional graded $R(-\beta)$-modules, is
isomorphic to the integral form $A_q(\mathfrak{n})_{\Z[q^{\pm1}]}$ of
$A_q(\mathfrak{n}) \simeq U_q^-(\g)$. Here the tensor functor $\tens$
of the monoidal category $R \gmod$ is given by
the convolution product $\conv$, and the action of $q$ is given by
the grading shift functor. In \cite{VV09, R11},
Varagnolo-Vasserot and Rouquier
proved that the upper global basis $\B^\upper$ of $A_q(\mathfrak{n})$
corresponds to the
set of the isomorphism classes of all {\it self-dual} simple modules of $R
\gmod$ under the assumption that $R$ is associated with a {\it
symmetric} quantum group $U_q(\g)$ and the base field is of characteristic $0$.

Combining works of \cite{GLS,Kimu12,VV09}, the  unipotent quantum
coordinate ring $A_q(\mathfrak{n}(w))$ associated with a
symmetric quantum group $U_q(\g)$ and a Weyl group element $w$
is isomorphic to the Grothendieck group of
 a  monoidal abelian full subcategory $\shc_w$ of $R \gmod$  whose  base field
$\cor$ is of characteristic $0$,
satisfying the following properties : {\rm (i)} $\shc_w$ is
stable under extensions and grading shift functor, {\rm (ii)} the
composition factors of $M \in \shc_w$ are contained
in $\B^{\upper}(w)$   (see Definition \ref{def:cw}).
In particular, the first condition in \eqref{cond:monoidal} holds.
However, it is not evident that  the second and the third condition in \eqref{cond:monoidal}
on  quantum monoidal seeds are satisfied. The purpose of this paper is to
ensure that those conditions hold in $\shc_w$.

\medskip

In order to establish it, in the first part of the paper, we start with a continuation of  the work of \cite{KKKO14} about the
convolution products, heads and socles of graded modules over symmetric KLR algebras.
One of the main results in \cite{KKKO14} is that the convolution
product $M \conv N$ of a real simple $R(\beta)$-module $M$ and a
simple $R(\gamma)$-modules $N$ has a unique simple quotient and a unique
simple submodule. Moreover, if $M \conv N \simeq N \conv M$ up to a
grading shift, then $M \conv N$ is simple.
In such a case we say that $M$ and $N$ {\it
commute}. The main tool of \cite{KKKO14} was the R-matrix
$\rmat{M,N}$, constructed in \cite{K^3}, which is a homogeneous
homomorphism from $M \conv N$ to $N \conv M$ of degree
$\La(M,N)$.
In this work, we define some integers encoding necessary information on $M\conv N$,
\begin{align*}
\tLa(M,N) \seteq \frac{1}{2}\bl\La(M,N)+(\beta,\gamma)\br, \quad
\de(M,N)\seteq \frac{1}{2}\bl\La(M,N)+\La(N,M)\br
\end{align*}
and study the representation theoretic meaning of the integers $\La(M,N)$, $\tLa(M,N)$ and $\de(M,N)$.

We then prove Leclerc's first conjecture (\cite{L03})
on the multiplicative
structure of elements in $\B^\upper$, when the generalized Cartan matrix is
 symmetric (Theorem \ref{th:leclerc} and
Theorem \ref{th:head}). Theorem \ref{th:head} is due to McNamara
(\cite[Lemma 7.5]{Mc14}) and the authors thank him for informing us  of his result.

 We say that $b \in \B^\upper$ is {\it real} if
$b^2 \in q^\Z\,\B^\upper\seteq\bigsqcup_{n\in\Z}q^n\B^\upper$.
\begin{theorem*}[{\cite[Conjecture 1]{L03}}]  Let $b_1$ and $b_2$ be elements in $\B^\upper$ such that one of them is real and $b_1b_2 \not\in q^\Z\B^\upper$.
Then the expansion of $b_1b_2$ with respect to $\B^\upper$ is of the
form
$$ b_1b_2= q^m b' + q^s b'' + \sum_{c \ne b',b''}\gamma^c_{b_1,b_2}(q)c,$$
where $b' \ne b''$, $m,s \in \Z$, $m < s$,
and
$$  \gamma^c_{b_1,b_2}(q) \in q^{m+1}\Z[q] \cap q^{s-1}\Z[q^{-1}].$$
\end{theorem*}

More precisely, we prove that $q^mb'$ and $q^sb''$ correspond to the
simple head
and the simple socle of $M\conv N$, respectively,
when $b_1$ corresponds to a %real
simple module $M$ and $b_2$ corresponds to a simple module $N$.

\medskip
Next, we move to provide an algebraic
framework for monoidal categorification of
quantum cluster algebras.
In order to simplify the conditions of quantum monoidal seeds and their mutations, we introduce
the notion of {\em admissible pairs} in $\mathcal C_w$.
 A pair $(\{M_i\}_{i \in\K}, \widetilde B)$ is called admissible in $\mathcal C_w$ if
(i)  $\{M_i\}_{i \in\K}$ is a commuting family
of self-dual real simple  objects of $\shc_w$,
(ii) $\widetilde B$ is an integer-valued $\K \times \K_\ex$-matrix with skew-symmetric principal part,
and (iii)
 for each $k \in\K$, there exists a  self-dual  simple object $M'_k$ in $\shc_w$
such that $M_k'$ commutes with $M_i$  for all  $i\in\K\setminus\{k\}$
and there are exact sequences in $\shc_w$
\eq&&\ba{l}  \label{eq:intro}
 0 \to q \sodot_{b_{i,k} >0} M_i^{\snconv b_{i,k}} \to q^{\tLa(M_k,M_k')} M_k \conv M_k' \to
 \sodot_{b_{i,k} <0} M_i^{\snconv (-b_{i,k})} \to 0\\
 0 \to q \sodot_{b_{i,k} <0} M_i^{\snconv(- b_{i,k})} \to q^{\tLa(M_k',M_k)} M'_k \conv M_k
 \to \sodot_{b_{i,k} >0} M_i^{\snconv b_{i,k}} \to 0
\ea
\eneq
 where $\tLa(M_k,M_k')$ and $\tLa(M_k',M_k)$ are prescribed integers
and $\sodot$ is a convolution product up to a power of $q$.

For an admissible pair  $(\{M_i\}_{i \in\K}, \widetilde B)$, let
$\La=(\La_{i,j})_{i,j \in\K}$
be the skew-symmetric matrix
where $\La_{i,j}$ is the homogeneous degree of $\rmat{{M_i},{M_j}}$, the R-matrix between $M_i$ and $M_j$,
and let $D=\{d_i\}_{ i \in\K}$ be the family of elements in $\rl$ given by $M_i \in R(-d_i) \gmod$.

Then, together with the result of \cite{GLS},
 our main theorem in the first part of the paper reads as follows:
\begin{maintheorem1}[Theorem \ref{th:main} and Corollary \ref{cor:main}]
 If there exists an admissible pair  $(\{M_i\}_{i\in K},\widetilde B)$ in $\mathcal C_w$ such that
 $[\Seed]\seteq\bl\{q^{-(\wt(M_i), \wt(M_i))/4}[M_i]\}_{i\in\K},
-\La,\widetilde B, D\br$
  is an initial seed of $A_{q^{1/2}}(\n(w))$,
then $\mathcal C_w$ is a monoidal categorification of $A_{q^{1/2}}(\n(w))$.
\end{maintheorem1}

The second part of this paper (\S\;8--11) is mainly devoted to showing that there exists an admissible pair in $\mathcal C_w$ for every symmetric Kac-Moody algebra $\g$ and its Weyl group element $w$.
In \cite{GLS},  Gei\ss, Leclerc and
Schr{\"o}er provided an initial quantum seed in $A_q(\n(w))$
whose quantum cluster variables are
unipotent quantum minors.
The unipotent quantum minors are elements in $A_q(\n)$, which are
regarded as a $q$-analogue of a generalization of the minors of upper triangular matrices.
In particular, they are elements in $\mathbf B^\up$.
We define the \emph{determinantial module} $\M(\mu,\zeta)$ to be the simple module in
$R \gmod$ corresponding to the unipotent quantum minor $\D(\mu,\zeta)$
 under the isomorphism $A_q( \n)_{\Z[q^{\pm1}]} \simeq K(R \gmod)$.
Here $(\mu,\zeta)$ is a pair of elements in the weight lattice of $\g$ satisfying certain conditions.

Our main theorem of the second part is as follows.
\begin{maintheorem2} [Theorem \ref{thm: main}]
Let $( \{ D(k,0) \}_{1 \le k \le r}, \wB, L)$ be the initial quantum seed of $A_q(\n(w))$ in \cite{GLS} with respect to a reduced expression $\widetilde{w}=s_{i_r}\cdots s_{i_1}$ of $w$.
Let $\M(k,0):=\M(s_{i_1}\cdots s_{i_k}\varpi_{i_k},\varpi_{i_k})$  be the determinantial module corresponding to the unipotent quantum minor $D(k,0)$.
Then the pair
$$( \{ \M(k,0) \}_{1 \le k \le r}, \wB)$$ is admissible in $\mathcal C_w$.
\end{maintheorem2}

Combining these theorems, the category $\shc_w$ gives  a monoidal categorification of  the quantum cluster algebra  $A_q(\n(w))$.
 If we take the base field of the symmetric KLR algebra to be of characteristic $0$,
these theorems,  along with
Theorem~\ref{thm:categorification 2} due to \cite{VV09, R11}, imply the
quantization conjecture. %: \eqn&&\parbox{80ex}{
%\begin{enumerate}
%\item[{\rm (QM3)}] The set of cluster monomials of $A_q(\n(w))$ is contained in the upper global basis $\B^\upper(w)$.
%\end{enumerate}
%} \eneqn

\medskip
The most essential condition for an admissible pair is that there exists the \emph{first mutation} $\M(k,0)'$ in the exact sequences \eqref{eq:intro} for each $k \in \K_\ex$.
To  establish this,
 we investigate the properties of determinantial modules and those of their convolution products.
Note that a unipotent quantum minor is the image of a global basis element of  the \emph{quantum coordinate ring} $A_q(\g)$ under a natural projection $A_q(\g) \to A_q(\n)$.
Since there exists a  bicrystal embedding from the crystal basis $B(A_q(\g))$ of $A_q(\g)$ to
the crystal basis $B(\tUqg)$ of the \emph{modified quantum groups} $\tUqg$,
 this investigation amounts to the  study  of the  interplay among the
 crystal and global bases of
 $A_q(\g)$, $\tUqg$ and $A_q(\n)$.
Hence we start the second part of the paper with the studies  of those algebras and their  crystal / global  bases along the line of the works in \cite{Kas93, Kas93a, Kash94}.

Next, we recall the (unipotent) quantum minors and the  \emph{T-system},
an equation consisting of three terms in products of unipotent quantum minors studied in \cite{BZ05, GLS}.
A detailed study  of the relation between $A_q(\g)$, $\tUqg$ and $A_q(\n)$
and their global bases enables us to establish several equations
involving unipotent quantum minors  in the algebra $\An$.
The upshot is that  those equations can be translated into exact sequences
in the category $R\gmod$ involving convolution products of  determinantial modules via
the categorification of $U_q^-(\g)$.
It enables us to show that the pair $( \{ \M(k,0) \}_{1 \le k \le r}, \wB)$ is admissible.

\medskip

The paper is organized as follows.   In Section 1,  we briefly review
basic materials on quantum group $\U$ and KLR algebra $R$.
In Section 2, we continue the study in
\cite{KKKO14}  of the R-matrices between $R$-modules.
In Section 3, we derive  certain properties of
$\tLa(M,N)$ and $\de(M,N)$.
In Section 4, we prove the first conjecture of Leclerc in \cite{L03}.
In Section 5, we recall the definition of quantum
cluster algebras. In Section 6, we give the definitions of a monoidal
seed, a quantum monoidal seed, a monoidal categorification of a cluster
algebra and a monoidal categorification of a quantum cluster algebra.
In Section 7, we prove  Main Theorem 1.
 In Section 8,   we review the algebras $A_q(\g)$, $\tUqg$ and $A_q(\n)$,
 and study  the relations among them.
 In Section 9, we study the properties of quantum minors including  T-systems and generalized T-systems.
 In Section 10, we study the determinantial modules over \KLRs.
Finally, in Section 11, we establish  Main theorem 2.

\bigskip

\noindent
{\bf Acknowledgements.} The authors would like to express their gratitude to
Peter McNamara who informed us  of his result.
They would also like to express their gratitude to
Bernard Leclerc and Yoshiyuki Kimura for many fruitful discussions.
The last two authors gratefully acknowledge the hospitality of
Research Institute for Mathematical Sciences, Kyoto University
during their visits in 2014.

%%%%%%%%%%%%%%%%%%%%%%%%%%%%%%%%%%%%%%%%%%%%%%%%%%
%%Part 1
%\part{Monoidal categorification of cluster algebras and  \KLRs}

\section{Quantum groups and global bases}
In this section, we briefly recall the quantum groups and the crystal and global bases theory for
$U_q(\g)$.
We refer to \cite{Kash91,Kas93,Kas95} for materials in this subsection.

\subsection{Quantum groups} \label{subsec:qgroups}
Let $I$
be an index set. A \emph{Cartan datum} is a quintuple $(A,\wl,
\Pi,\wl^{\vee},\Pi^{\vee})$ consisting of
\begin{enumerate}[(i)]
\item an integer-valued matrix $A=(a_{ij})_{i,j \in I}$,
called the \emph{symmetrizable generalized Cartan matrix},
 which satisfies
\be[{\rm(a)}]
\item $a_{ii} = 2$ $(i \in I)$,
\item $a_{ij} \le 0 $ $(i \neq j)$,
\item there exists a diagonal matrix
$D=\text{diag} (\mathsf s_i \mid i \in I)$ such that $DA$ is
symmetric, and $\mathsf s_i$ are relatively prime positive integers,
\end{enumerate}

\item a free abelian group $\wl$, called the \emph{weight lattice},
\item $\Pi= \{ \alpha_i \in \wl \mid \ i \in I \}$, called
the set of \emph{simple roots},
\item $\wl^{\vee}\seteq\Hom_\Z(\wl, \Z)$, called the \emph{co-weight lattice},
\item $\Pi^{\vee}= \{ h_i \ | \ i \in I \}\subset \wl^{\vee}$, called
the set of \emph{simple coroots},
satisfying the following properties:
\be[{\rm(1)}]
\item $\langle h_i,\alpha_j \rangle = a_{ij}$ for all $i,j \in I$,
\item $\Pi$ is linearly independent  over $\Q$,
\item for each $i \in I$, there exists $ \varpi_i \in \wl$ such that %$\Lambda_i \in \wl$
           $\langle h_j,  \varpi_i \rangle =\delta_{ij}$  for all $j \in I$.

\end{enumerate}
We call $ \varpi_i$ the \emph{fundamental weights}.
\end{enumerate}

\medskip
\noi
The free abelian group $\rootl\seteq\soplus_{i \in I} \Z \alpha_i$ is called the
\emph{root lattice}. Set $\rootl^{+}= \sum_{i \in I} \Z_{\ge 0}
\alpha_i\subset\rootl$ and $\rootl^{-}= \sum_{i \in I} \Z_{\le0}
\alpha_i\subset\rootl$. For $\beta=\sum_{i\in I}m_i\al_i\in\rootl$,
we set
$|\beta|=\sum_{i\in I}|m_i|$.

Set $\mathfrak{h}=\Q \otimes_\Z \wl^{\vee}$.
Then there exists a symmetric bilinear form $(\quad , \quad)$ on
$\mathfrak{h}^*$ satisfying
$$ (\alpha_i , \alpha_j) =\mathsf s_i a_{ij} \quad (i,j \in I)
\quad\text{and $\lan h_i,\lambda\ran=
\dfrac{2(\alpha_i,\lambda)}{(\alpha_i,\alpha_i)}$ for any $\lambda\in\mathfrak{h}^*$ and $i \in I$}.$$

The \emph{Weyl group} of $\g$ is the group of linear transformations on $\h^*$ generated by $s_i$ $(i\in I)$, where
\eqn
s_i(\la) := \la - \lan h_i, \la \ran \al_i \quad \text{for } \ \la \in \h^*, \ i \in  I.
\eneqn

Let $q$ be an indeterminate. For each $i \in I$, set $q_i = q^{\,\mathsf s_i}$.

\begin{definition} \label{def:qgroup}
The {\em quantum group}
associated with a Cartan datum
$(A,\wl,\Pi,\wl^{\vee}, \Pi^{\vee})$ is the  algebra $\U$ over
$\mathbb Q(q)$ generated by $e_i,f_i$ $(i \in I)$ and
$q^{h}$ $(h \in \wl^\vee)$ satisfying the following relations:
\begin{equation*}
\begin{aligned}
& q^0=1,\ q^{h} q^{h'}=q^{h+h'} \ \ \text{for} \ h,h' \in \wl^\vee,\\
& q^{h}e_i q^{-h}= q^{\lan h, \alpha_i\ran} e_i, \ \
          \ q^{h}f_i q^{-h} = q^{-\lan h, \alpha_i\ran} f_i \ \ \text{for} \ h \in \wl^\vee, i \in
          I, \\
& e_if_j - f_je_i = \delta_{ij} \dfrac{t_i -t^{-1}_i}{q_i- q^{-1}_i
}, \ \ \mbox{ where } t_i=q^{\mathsf s_i h_i}, \\
& \sum^{1-a_{ij}}_{r=0} (-1)^r \left[\begin{matrix}1-a_{ij}
\\ r\\ \end{matrix} \right]_i e^{1-a_{ij}-r}_i
         e_j e^{r}_i =0 \quad \text{ if } i \ne j, \\
& \sum^{1-a_{ij}}_{r=0} (-1)^r \left[\begin{matrix}1-a_{ij}
\\ r\\ \end{matrix} \right]_i f^{1-a_{ij}-r}_if_j
        f^{r}_i=0 \quad \text{ if } i \ne j.
\end{aligned}
\end{equation*}
\end{definition}

Here, we set $[n]_i =\dfrac{ q^n_{i} - q^{-n}_{i} }{ q_{i} - q^{-1}_{i} },\quad
  [n]_i! = \prod^{n}_{k=1} [k]_i$ and
  $\left[\begin{matrix}m \\ n\\ \end{matrix} \right]_i= \dfrac{ [m]_i! }{[m-n]_i! [n]_i! }\;$
  for $i \in I$ and $m,n \in \Z_{\ge 0}$ such that $m\ge n$.

Let $U_q^{+}(\g)$ (resp.\ $U_q^{-}(\g)$) be the subalgebra of
$U_q(\g)$ generated by $e_i$'s (resp.\ $f_i$'s), and let $U^0_q(\g)$
be the subalgebra of $U_q(\g)$ generated by $q^{h}$ $(h \in
\wl^{\vee})$. Then we have the \emph{triangular decomposition}
$$ U_q(\g) \simeq U^{-}_q(\g) \otimes U^{0}_q(\g) \otimes U^{+}_q(\g),$$
and the {\em weight space decomposition}
$$U_q(\g) = \bigoplus_{\beta \in \rootl} U_q(\g)_{\beta},$$
where $U_q(\g)_{\beta}\seteq\set{ x \in U_q(\g)}{\text{$q^{h}x q^{-h}
=q^{\lan h, \beta \ran}x$ for any $h \in \wl$}}$.

There are $\Q(q)$-algebra antiautomorphisms $\varphi$ and $^*$ of $U_q(\g)$ given as follows:
\begin{align*}
&\varphi(e_i)=f_i, \quad \varphi(f_i)=e_i, \quad \varphi(q^h)=q^h, \\ %\label{eq: antiauto v} \\
&e_i^*=e_i, \quad\quad \ f_i^*=f_i, \quad\quad \ (q^h)^*=q^{-h}. %\label{eq: antiauto s}
\end{align*}
There is  also a $\Q$-algebra automorphism $\ol{\phantom{a}}$ of $U_q(\g)$ given by
\begin{align*}
& \overline{e}_i=e_i, \quad \overline{f}_i=f_i, \quad \overline{q^h}=q^{-h}, \quad \overline{q}=q^{-1}. %\label{eq: auto bar}
\end{align*}

We define the divided powers by
$$e_i^{(n)} = e_i^n / [n]_i!, \quad f_i^{(n)} =
f_i^n / [n]_i! \ \ (n \in \Z_{\ge 0}).$$
Let us denote by $U_q(\g)_\A$ the $\A$-subalgebra of $U_q(\g)$ generated by
$e_i^{(n)}$, $f_i^{(n)}$, $q^h$, and $ \displaystyle\prod_{k=1}^n \dfrac{\{q^{1-k}q^h\}}{[k]}$ ($i\in I, \ n \in \Z_{\ge 0},\  h \in \wl^\vee$), where $\{x\}:=(x-x^{-1})/(q-q^{-1})$.
Let us also denote by $U_q^{-}(\g)_\A$ the
$\A$-subalgebra of $U^-_q(\g)$ generated by $f_i^{(n)}$
($i\in I$, $n \in \Z_{\ge 0}$), and
by $U_q^{+}(\g)_\A$ the
$\A$-subalgebra of $U^+_q(\g)$ generated by $e_i^{(n)}$
($i\in I$, $n \in \Z_{\ge 0}$).

%%%%%%%%%%%%%%%%%%%%%%%%%%
%%%End of the definition of quantum groups

%%%%%%%

\subsection{Integrable representations}

A $U_q(\g)$-module $M$ is called {\em integrable}
if  $M= \bigoplus_{\eta \in \wl} M_\eta$
where $M_\eta \seteq \{  m \in M \ | \ q^hm=q^{\lan h,\eta \ra} m \}$,  $\dim M_\eta < \infty$, and
 the actions of $e_i$ and $f_i$  on $M$ are locally nilpotent for all $i \in I$.
We denote by $\Oint(\g)$ the category of integrable left $U_q(\g)$-modules $M$ satisfying
%the following conditions:  {\rm (i)} {\rm (ii)} and {\rm (iii)}
that
there exist finitely many weights $\la_1$, \ldots, $\la_m$ such that
$\wt(M)\subset\cup_{j}(\la_j+\rl^-)$. The category $\Oint(\g)$ is semisimple with its simple objects being isomorphic to the
highest weight modules $V(\lambda)$ with highest weight vector $u_\lambda$ of
highest weight $\lambda \in \wl^+ :=\set{\mu \in \wl}{\lan h_i, \mu \ran \ge 0 \ \text{for all} \ i \in I}$, the set of dominant integral weights.

For $M \in \Oint(\g)$, let us denote by $\Dv M$ the left $U_q(\g)$-module $\bigoplus_{\eta \in \wl} \Hom_{\Q(q)}(M_\eta,\Q(q))$ with the action of $U_q(\g)$ given by:
$$  (a\psi)(m)=\psi(\vph(a)m)\quad\text{for $\psi\in\Dv M$,
$m\in M$ and $a\in U_q(\g)$.}  $$
Then $\Dv M$ belongs to $\Oint(\g)$.

For a left $U_q(\g)$-module $M$, we denote by $M^\ri$ the right $U_q(\g)$-module
$\{m^\ri \ | \ m\in M \}$ with the right action of $U_q(\g)$ given by
$$\text{$(m^\ri)\, x=(\vph(x)m)^\ri$ for $m\in M$ and $x\in U_q(\g)$.}$$
We denote by $\Oint^{\ri} (\g)$ the category of right integrable $U_q(\g)$-modules $M^\ri$ such that
$M \in \Oint(\g)$.

There are two comultiplications $\comp$ and $\comm$ on $U_q(\g)$ defined as follows:
\begin{align}
&\comp(e_i)=e_i \tens 1 + t_i \tens e_i,\quad \comp(f_i)=f_i \tens t_i^{-1} + 1 \tens f_i,\quad \comp(q^h)=q^h \tens q^h, \label{eq: comp}  \\
&\comm(e_i)=e_i \tens t_i^{-1} + 1 \tens e_i,\quad \comm(f_i)=f_i \tens 1 + t_i \tens f_i,\quad \comm(q^h)=q^h \tens q^h. \label{eq: comm}
\end{align}

For two $U_q(\g)$-modules $M_1$ and $M_2$, let us denote by $M_1
\tensp M_2$ and $M_1 \tensm M_2$ the vector space $M_1 \tens_{\Q(q)}M_2$ endowed with $U_q(\g)$-module
structure induced by the comultiplications $\comp$ and $\comm$, respectively. Then we have
$$ \Dv(M_1 \otimes_{\pm} M_2 ) \simeq (\Dv M_1) \otimes_{\mp} (\Dv M_2).$$

For any $i \in I$,
there exists a unique $\Q(q)$-linear endomorphism $e'_i$
of $U^-_q(\g)$  such that
\eqn
e'_i(f_j)=\delta_{i,j} \ (j \in I), &
 e'_i(xy) = (e'_i x) y + q_i^{\langle h_i, \beta \rangle} x (e'_iy) \ (x \in U^{-}_q(\g)_{\beta}, y \in U^-_q(\g)).
\eneqn
The \emph{quantum boson algebra} $B_q(\g)$ is defined as the subalgebra of $\End_{\Q(q)}(U_q(\g))$
generated by $f_i$ and $e'_i$ $(i \in I)$.
Then $B_q(\g)$ has  a $\Q(q)$-algebra anti-automorphism $\vph$ which sends $e_i'$ to $f_i$ and
$f_i$ to $e_i'$.  As a $B_q(\g)$-module, $U^-_q(\g)$ is simple.

The  simple $U_q(\g)$-module $V(\lambda)$ and the simple
$B_q(\g)$-module $U_q^-(\g)$
have a unique non-degenerate  symmetric bilinear form
$( \ , \ )$ such that
\begin{align*}
& (u_\lambda,u_\lambda) = 1 \text{ and } (xu,v)=(u,\vph(x)v) \text{ for } u,v \in V(\lambda) \text{ and } x \in U_q(\g), \\
& \ ( \one ,\one) = 1  \text{ and } (xu,v)=(u,\vph(x)v) \text{ for } u,v \in U^-_q(\g) \text{ and } x \in B_q(\g).
\end{align*}

Note that $( \ , \ )$ induces the non-degenerate bilinear form
$$  \lan \cdot , \cdot \ra \cl  V(\lambda)^\ri \times V(\lambda) \to \Q(q)$$
given by $ \lan u^\ri, v \ra  = (u,v)$,
by which $\Dv V(\la)$ is canonically isomorphic to $V(\la)$.

%%%%%%%%%%%

\vskip 2em

\subsection{Crystal bases and global bases}

For a  subring $A$ of $\Q(q)$, we say that $L$ is an $A$-{\em lattice} of
a $\Q(q)$-vector space $V$ if
$L$ is a free $A$-submodule of $V$ such that $V=\Q(q)\tens_AL$.

Let us denote by $\QA_0$ (resp.\ $\QA_\infty$) the ring of rational
functions in $\Q(q)$ which are regular at $q=0$ (resp.\ $q=\infty$).
 Set $\QA\seteq
\Q[q^{\pm1}]$ .

\medskip

Let $M$ be a $\U$-module in $\Oint(\g)$.
Then, for each $i \in I$, any $u\in M$ can be uniquely written as
$$u=\sum_{n \ge 0}f_i^{(n)}u_n\quad\text{with $e_iu_n=0$.}$$
We define the {\em lower Kashiwara operators} by
\eqn&&\te^\low_i(u)=\sum_{n\ge1}f_i^{(n-1)}u_n
\quad\text{and}\quad
\tf^\low_i(u)=\sum_{n\ge0}f_i^{(n+1)}u_n,
\eneqn
and the {\em upper Kashiwara operators}
by
\eqn&&
\te^\up_i(u)=\te_i^{\low}q_i^{-1} t_i^{-1} u\quad\text{and}\quad
\tf^\up_i(u)=\tf_i^{\low}q_i^{-1} t_i u.
\eneqn

Similarly,
for each $ i \in  I$, any element $x \in U_q^-(\g)$ can be written uniquely as
\eqn
x = \sum_{n \ge 0} f_i^{(n)} x_n\quad\text{ with $e'_i x_n =0$.}
\eneqn
We define the {\em Kashiwara operators} $\tilde e_i, \tilde f_i$ on
$U_q^-(\g)$ by
\eqn
\tilde e_i x = \sum_{n \ge 1} f_i^{(n-1)}x_n, & \tilde f_i x = \displaystyle\sum_{n \ge 0} f_i^{(n+1)}x_n.
\eneqn

We say that an $\QA_0$-lattice $L$ of $M$
 is a lower (resp.\ upper) crystal lattice of $M$
if $L=\soplus_{\eta \in \wl}L_\eta$, where $L_\eta = L \cap M_\eta$
and it is invariant by the lower (resp.\ upper) Kashiwara operators.

\Lemma\label{lem:lowup} Let $L$ be a lower crystal lattice of $M\in \Oint(\g)$.
Then we have
\bnum
\item $\soplus\nolimits_{\la\in\wl}q^{-(\la,\la)/2}L_\la$ is an upper crystal lattice of $M$.
\item $L^\vee\seteq\set{\psi\in\Dv M}{\lan\psi,L\ran\in\QA_0}$ is an upper crystal lattice of
$\Dv M$.
\ee
\enlemma
\Proof
(i) Let $\phi_M$ be the endomorphism of $M$ given by $\phi_M\vert_{M_\la}=q^{-(\la,\la)/2}\id_{M_\la}$. Then we have
$\te_i^\up=\phi_M\circ\te_i^\low \circ \phi_M^{-1}$ and
$\tf_i^\up=\phi_M\circ\tf_i^\low \circ \phi_M^{-1}$.

\smallskip
\noi

(ii) follows from $(3.2.1)$, $(3.2.2)$  in \cite{Kas93}.
 Note that the definition of upper Kashiwara operators are slightly different from the ones in \cite{Kas93}, but  similar properties hold.
%the fact that $\te_i^\up$ and $\tf_i^\up$ are the adjoint operators of $\tf_i^\low$ and $\te_i^\low$, respectively.
\QED

\begin{definition} A {\em lower} \ro resp.\ {\em upper}\rf\
{\em crystal basis} of $M$ consists of a pair $(L,B)$
satisfying the following conditions:
\bnum
\item $L$ is a lower \ro resp. upper\rf\ crystal lattice of $M$,
\item $B= \sqcup_{\eta \in \wl} B_\eta$ is a basis of the $\Q$-vector space $L/qL$, where $B_\eta=B \cap (L_\eta/qL_\eta)$,
\item the induced maps $\tilde{e}_i$ and $\tilde{f}_i$ on $L/qL$ satisfy
$$ \tilde{e}_iB, \tilde{f}_iB \subset B \sqcup \{0\}, \text{ and }  \tilde{f}_ib=b' \text{ if and only if } b=\tilde{e}_ib' \text{ for } b,b' \in B.$$
Here $\te_i$ and $\tf_i$ denote the lower (resp.\ upper) Kashiwara operators.
\end{enumerate}
\end{definition}

For $\la \in \wl^+$, let $u_\la$ be a highest weight  vector of $V(\la)$. Let $L^\low(\la)$ be the $\QA_0$-submodule of $V(\la)$ generated by
$\set{\tf_{i_1} \cdots \tf_{i_l} u_\la}{l \in \Z_{\ge 0}, \ i_1 , \ldots, i_l \in I}$  and let $B(\la)$ be the subset of $L^\low(\la) / qL^\low(\la)$ given by
\eqn
B^\low(\la) = \set{\tilde f_{i_1} \cdots \tilde f_{i_l}  u_\la \mod q L(\la)}{l \in \Z_{\ge0}, \ i_1, \ldots, i_l \in I} \backslash \{0\}.
\eneqn

It is shown in \cite{Kash91} that  $(L^\low(\lambda),B^\low(\lambda))$ is a lower crystal basis of $V(\lambda)$. Using
the non-degenerate  symmetric bilinear form $( \ , \ )$, $V(\lambda)$ has
the upper crystal basis
$(L^\up(\lambda),B^\up(\lambda))$ where
$$ L^\up(\lambda) \seteq \{ u \in V(\lambda) \
| \ (u,L^{\low} (\lambda)) \subset \QA_0 \},$$
and $B^\up(\lambda)\subset L^\up(\lambda)/qL^\up(\lambda)$ is the dual basis of $B^\low(\lambda)$ with respect to
the induced non-degenerate pairing between
$L^\up(\lambda)/qL^\up(\lambda)$ and $L^\low(\lambda)/qL^\low(\lambda)$.

An (abstract) {\em crystal} is a set $B$ together with maps
$$ \wt\cl B \to \wl, \ \ \ve_i,\vph_i\cl  B \to \Z \sqcup \{ \infty \} \text{ and } \te_i,\tf_i\cl B \to B \sqcup \{ 0 \} \text{ for } i \in I,$$
such that
\begin{itemize}
\item[{\rm (C1)}] $\vph_i(b)=\ve_i(b)+\lan h_i,\wt(b) \ra$ for any $i$,
\item[{\rm (C2)}] if $b \in B$ satisfies $\te_i(b) \ne 0$, then
$$ \ve_i(\te_ib)=\ve_i(b)-1, \ \vph_i(\te_ib)=\vph_i(b)+1, \ \wt(\te_i b)=\wt(b)+\alpha_i, $$
\item[{\rm (C3)}] if $b \in B$ satisfies $\tf_i(b) \ne 0$, then
$$ \ve_i(\tf_ib)=\ve_i(b)+1, \ \vph_i(\tf_ib)=\vph_i(b)-1, \ \wt(\tf_i b)=\wt(b)-\alpha_i, $$
\item[{\rm (C4)}] for $b,b' \in B$, $b'=\tf_i b$ if and only if $b=\te_i b'$,
\item[{\rm (C5)}] if $\vph_i(b)=-\infty$, then $\te_ib=\tf_ib=0$.
\end{itemize}

Recall that, with the notions of {\em morphism} and {\em tensor product rule} of crystals, the category of  crystals becomes a monoidal category
(\cite{Kash94}). If $(L,B)$ is a crystal basis of
$M$, then $B$ is an abstract crystal.
 Since $B^\low(\lambda) \simeq B^\up(\lambda)$,
we drop the superscripts for simplicity.

Let $V$ be a $\Q(q)$-vector space, and let $L_0$ be an $\QA_0$-lattice of $V$,
$L_\infty$ an $\QA_\infty$-lattice of $V$ and $V_{\QA}$ an
$\QA$-lattice of $V$.
We say that the triple $(V_{\QA},L_0,L_\infty)$ is {\em balanced} if the following canonical map is a $\Q$-linear isomorphism:
$$ E\seteq V_{\QA} \cap L_0 \cap L_\infty  \isoto L_0/qL_0.$$
The inverse of the above isomorphism
$G\cl  L_0/qL_0\isoto E$ is called the
{\em globalizing map}.
If $(V_{\QA},L_0,L_\infty)$ is balanced, then we have
\eqn &&\text{$\Q(q)\tens_\Q E\isoto V$,
$\QA\tens_\Q E\isoto V_{\QA}$, $\QA_0\tens_\Q E\isoto L_0$
and  $\QA_\infty\tens_\Q E\isoto L_\infty$.}
\eneqn
Hence, if $B$ is a basis of $L_0/qL_0$, then $G(B)$ is a basis of $V$,
$V_{\QA}$, $L_0$ and $L_\infty$. We call $G(B)$ a {\em global basis}.

\medskip
We define the two $\QA$-lattices
of $V(\la)$ by
\eqn
&&V^\low(\lambda)_\QA\seteq \bl \Q\tens U_q^-(\g)_{\Z[q^{\pm1}]}\br u_\lambda\quad\text{and}\\
&&V^\up(\lambda)_\QA\seteq
\set{ u \in V(\lambda)}{\bl u, V^\low(\lambda)_\QA\br \subset \QA }.\eneqn
Recall that there is a $\Q$-linear automorphism  $-$  on $V(\la)$ defined by
$$\ol{P u_{\la}}=\overline{P} u_\la,\quad \text{for } \ P \in \Uq.$$
Then $\bl V^\low(\lambda)_\QA,\; L^\low(\lambda),\; \ol{  L^\low(\lambda)}\br$ and
$\bl V^\up(\lambda)_\QA,\; L^\up(\lambda),\; \ol{  L^\up(\lambda)}\br$
are balanced.
Let us denote by $G^\low_\la$ and $G^\up_\la$ the associated globalizing maps, respectively.
(If there is no  danger of confusion, we simply denote them $G^\low$ and $G^\up$, respectively.)
Then the sets
$$\B^\low(\lambda) \seteq  \{G^\low_\lambda(b) \ | \ b \in B^\low(\lambda) \} \ \  \text{ and } \ \ \B^\up(\lambda) \seteq  \{G^\up_\lambda(b) \ | \ b \in B^\up(\lambda) \}$$
form $\A$-bases of
\eqn
&& V^\low(\lambda)_\A:=U_q(\g)_\A u_\la \quad \text{and} \\
&&V^\up(\lambda)_\A \seteq \set{ u \in V(\lambda)}{\bl u, V^\low(\lambda)_\A\br \subset \A },
\eneqn
 respectively.
They are called the {\em lower global basis} and the {\em upper global basis} of $V(\lambda)$.

Set
\eqn
&&L(\infty) := \sum_{l \in \Z_{\ge0}, \, i_1, \ldots, i_l \in I}
\QA_0 \tilde f_{i_1} \cdots \tilde f_{i_l} \cdot \one \subset U_q^-(\g)  \quad \text{and} \\
&&B(\infty) := \set{\tilde f_{i_1} \cdots \tilde f_{i_l} \cdot \one \mod q L(\infty)}{l \in \Z_{\ge0}, i_1, \ldots, i_l \in I}  \subset L(\infty) / q L(\infty).
\eneqn
 Then  $(L(\infty),B(\infty))$ is a lower crystal basis of the simple $B_q(\g)$-module $U_q^-(\g)$
and the triple $( \Q\tens U_q^-(\g)_{\Z[q^{\pm1}]}, L(\infty),\overline{L(\infty)})$ is balanced.
Let us denote the globalizing map by $G^{\low}$.
Then the set
$$\B^\low(\Um)\seteq  \{ G^\low(b )\ | \ b \in B(\infty) \}$$
forms  a $\A$-basis of $\Um_\A$ and is called the {\em lower global basis} of $U_q^-(\g)$.

Let us denote by
\eq \label{eq:Bup}
\B^\up(\Um) \seteq  \{G^\up(b) \ | \ b \in B(\infty) \}
\eneq
the dual basis of $\B^\low(\Um)$ with respect to $( \ , \ )$.
Then it is a $\A$-basis of
 $$\Um_\A^\vee \seteq \{ x \in U^-_q(\g) \ | \ (x,\Um_\A) \subset \A \}$$
and  called the {\em upper global basis} of $U_q^-(\g)$.
Note that $U^-_q(\g)_\A^\vee$ has a $\A$-algebra structure as a subalgebra of $U^-_q(\g)$
(see also \S\,\ref{subsec:unipotent}) .

\section{\KLRs \ and R-matrices}

\subsection{\KLRs}
We recall the definition of Khovanov-Lauda-Rouquier algebra or quiver Hecke algebra
(hereafter, we abbreviate it
as KLR algebra) associated with a given
Cartan datum $(A, \wl, \Pi, \wl^{\vee}, \Pi^{\vee})$.

Let $\cor$ be a base field.
For $i,j\in I$ such that $i\not=j$, set
$$S_{i,j}=\set{(p,q)\in\Z_{\ge0}^2}{(\al_i , \al_i)p+(\al_j , \al_j)q=-2(\al_i , \al_j)}.
$$
Let us take  a family of polynomials $(Q_{ij})_{i,j\in I}$ in $\cor[u,v]$
which are of the form
\eq
&&\parbox{68ex}{
$
Q_{ij}(u,v) = \begin{cases}\hs{5ex} 0 \ \ & \text{if $i=j$,} \\[1.5ex]
\sum\limits_{(p,q)\in S_{i,j}}
t_{i,j;p,q} u^p v^q\quad& \text{if $i \neq j$}
\end{cases}
$

\vs{1ex}
\hs{8ex}\parbox{65ex}{
with $t_{i,j;p,q}\in\cor$ such that
$Q_{i,j}(u,v)=Q_{j,i}(v,u)$ and
$t_{i,j:-a_{ij},0} \in \cor^{\times}$.}
}
\label{eq:Q}
\eneq

We denote by
$\sym_{n} = \langle s_1, \ldots, s_{n-1} \rangle$ the symmetric group
on $n$ letters, where $s_i\seteq (i, i+1)$ is the transposition of $i$ and $i+1$.
Then $\sym_n$ acts on $I^n$ by place permutations.

For $n \in \Z_{\ge 0}$ and $\beta \in \rootl^+$ such that $|\beta| = n$, we set
$$I^{\beta} = \set{\nu = (\nu_1, \ldots, \nu_n) \in I^{n}}%
{ \alpha_{\nu_1} + \cdots + \alpha_{\nu_n} = \beta }.$$

\begin{definition}
For $\beta \in \rootl^+$ with $|\beta|=n$, the {\em
\KLR}  $R(\beta)$  at $\beta$ associated
with a Cartan datum $(A,\wl, \Pi,\wl^{\vee},\Pi^{\vee})$ and a matrix
$(Q_{ij})_{i,j \in I}$ is the  algebra over $\cor$
generated by the elements $\{ e(\nu) \}_{\nu \in  I^{\beta}}$, $
\{x_k \}_{1 \le k \le n}$, $\{ \tau_m \}_{1 \le m \le n-1}$
satisfying the following defining relations:
\allowdisplaybreaks[4]
\begin{equation*} \label{eq:KLR}
\begin{aligned}
& e(\nu) e(\nu') = \delta_{\nu, \nu'} e(\nu), \ \
\sum_{\nu \in  I^{\beta} } e(\nu) = 1, \\
& x_{k} x_{m} = x_{m} x_{k}, \ \ x_{k} e(\nu) = e(\nu) x_{k}, \\
& \tau_{m} e(\nu) = e(s_{m}(\nu)) \tau_{m}, \ \ \tau_{k} \tau_{m} =
\tau_{m} \tau_{k} \ \ \text{if} \ |k-m|>1, \\
& \tau_{k}^2 e(\nu) = Q_{\nu_{k}, \nu_{k+1}} (x_{k}, x_{k+1})
e(\nu), \\
& (\tau_{k} x_{m} - x_{s_k(m)} \tau_{k}) e(\nu) = \begin{cases}
-e(\nu) \ \ & \text{if} \ m=k, \nu_{k} = \nu_{k+1}, \\
e(\nu) \ \ & \text{if} \ m=k+1, \nu_{k}=\nu_{k+1}, \\
0 \ \ & \text{otherwise},
\end{cases} \\
& (\tau_{k+1} \tau_{k} \tau_{k+1}-\tau_{k} \tau_{k+1} \tau_{k}) e(\nu)\\
& =\begin{cases} \dfrac{Q_{\nu_{k}, \nu_{k+1}}(x_{k},
x_{k+1}) - Q_{\nu_{k}, \nu_{k+1}}(x_{k+2}, x_{k+1})} {x_{k} -
x_{k+2}}e(\nu) \ \ & \text{if} \
\nu_{k} = \nu_{k+2}, \\
0 \ \ & \text{otherwise}.
\end{cases}
\end{aligned}
\end{equation*}
\end{definition}

The above relations are homogeneous provided that
\begin{equation*} \label{eq:Z-grading}
\deg e(\nu) =0, \quad \deg\, x_{k} e(\nu) = (\alpha_{\nu_k}
, \alpha_{\nu_k}), \quad\deg\, \tau_{l} e(\nu) = -
(\alpha_{\nu_l} , \alpha_{\nu_{l+1}}),
\end{equation*}
and hence $R( \beta )$ is a $\Z$-graded algebra.

 For a graded $R(\beta)$-module $M=\bigoplus_{k \in \Z} M_k$, we define
$qM =\bigoplus_{k \in \Z} (qM)_k$, where
 \begin{align*}
 (qM)_k = M_{k-1} & \ (k \in \Z).
 \end{align*}
We call $q$ the \emph{grading shift functor} on the category of
graded $R(\beta)$-modules.

If $M$ is an $R(\beta)$-module, then we set $\wt(M)=-\beta \in \rootl^-$
and call it the {\em weight} of $M$.

\smallskip
We denote by $R(\beta) \Mod$
the category of $R(\beta)$-modules,
and by $R(\beta) \smod$
the full subcategory of $R(\beta)\Mod$ consisting of modules $M$
such that $M$ are finite-dimensional over $\cor$,
and the actions of the $x_k$'s on $M$ are nilpotent.

Similarly, we denote by $R(\beta)\gMod$ and by $R(\beta)\gmod$
the category of graded $R(\beta)$-modules and
the category of graded $R(\beta)$-modules which are finite-dimensional over $\cor$, respectively.
We set
$$
R\gmod=\soplus_{\beta\in\rtl^+}R(\beta)\gmod\quad\text{and}\quad
R\smod=\soplus_{\beta\in\rtl^+}R(\beta)\smod.
$$

\medskip
For $\beta, \gamma \in \rootl^+$ with $|\beta|=m$, $|\gamma|= n$,
 set
$$e(\beta,\gamma)=\displaystyle\sum_{\substack{\nu \in I^{\beta+\gamma}, \\ (\nu_1, \ldots ,\nu_m) \in I^{\beta},  \\ (\nu_{m+1}, \ldots ,\nu_{m+n} ) \in I^{\gamma} }} e(\nu) \in R(\beta+\gamma). $$
Then $e(\beta,\gamma)$ is an idempotent.
Let
\eqn R( \beta)\tens R( \gamma  )\to e(\beta,\gamma)R( \beta+\gamma)e(\beta,\gamma) \label{eq:embedding}
\eneqn
be the $\cor$-algebra homomorphism given by
$e(\mu)\tens e(\nu)\mapsto e(\mu*\nu)$ ($\mu\in I^{\beta}$
and $\nu\in I^{\gamma}$)
$x_k\tens 1\mapsto x_ke(\beta,\gamma)$ ($1\le k\le m$),
$1\tens x_k\mapsto x_{m+k}e(\beta,\gamma)$ ($1\le k\le n$),
$\tau_k\tens 1\mapsto \tau_ke(\beta,\gamma)$ ($1\le k<m$),
$1\tens \tau_k\mapsto \tau_{m+k}e(\beta,\gamma)$ ($1\le k<n$).
Here $\mu*\nu$ is the concatenation of $\mu$ and $\nu$;
i.e., $\mu*\nu=(\mu_1,\ldots,\mu_m,\nu_1,\ldots,\nu_n)$.

\medskip
For an $R(\beta)$-module $M$ and an $R(\gamma)$-module $N$,
we define the \emph{convolution product}
$M\conv N$ by
$$M\conv N=R(\beta + \gamma) e(\beta,\gamma)
\tens_{R(\beta )\otimes R( \gamma)}(M\otimes N). $$

For $M \in R( \beta) \smod$, the dual space
$$M^* \seteq \Hom_{\cor}(M, \cor)$$
admits an $R(\beta)$-module structure via
\begin{align*}
(r \cdot  f)(u) \seteq f(\psi(r) u) \quad (r \in R( \beta), \ u \in M),
\end{align*}
where $\psi$ denotes the $\cor$-algebra anti-involution on $R(\beta
)$ which fixes the generators $e(\nu)$, $x_m$ and $\tau_k$ for $\nu
\in I^{\beta}, 1 \leq m \leq  |\beta|$ and $1 \leq k<|\beta|$.

It is known that (see \cite[Theorem 2.2 (2)]{LV11})
\eqn
&&(M_1 \conv M_2)^* \simeq q^{(\beta,\gamma)}
(M_2^* \conv M_1 ^*)\label{eq:dualconv}
\eneqn
for any $M_1 \in R(\beta) \gmod$ and $M_2 \in R(\gamma) \gmod$.

A simple module $M$ in $R \gmod$ is called \emph{self-dual} if $M^* \simeq M$.
Every simple module is isomorphic to a grading shift of a self-dual simple module (\cite[\S 3.2]{KL09}).
 Note also that  we have $\End_{R(\beta)}{M} \simeq \cor$  for every simple module $M$ in $R(\beta) \gmod$ (\cite[Corollary 3.19]{KL09}).

\medskip
Let us denote by $K(R\gmod)$ the  Grothendieck group of $R\gmod$.
Then,
$K(R\gmod)$
is an algebra over  $\A$ with
 the multiplication induced by the convolution product and
the $\A$-action induced by the  grading shift functor $q$.

 In \cite{KL09, R08}, it is shown that
a \KLR \ \emph{categorifies} the negative half of the corresponding quantum group. More precisely, we have the following theorem.

\begin{theorem}[{\cite{KL09, R08}}] \label{thm:categorification 1}
 For a given  Cartan datum $(A,\wl, \Pi,\wl^{\vee},\Pi^{\vee})$,
we take a parameter matrix $(Q_{ij})_{i,j \in \K}$ satisfying the conditions in \eqref{eq:Q}, and let $U_q(\mathfrak g)$ and $R(\beta) $ be
the associated quantum group   and the \KLRs, respectively.
Then there exists  a $\A$-algebra isomorphism
\begin{align}
  U^-_q(\mathfrak g)_\A^{\vee} \simeq K(R\gmod).
\label{eq:KLRU}
\end{align}
\end{theorem}

 \KLRs\ also categorify the upper global bases.

\Def We say that  a \KLR\ $R$ is {\em symmetric} if
$Q_{i,j}(u,v)$ is a polynomial in $u-v$ for all $i,j\in I$.
\edf
In particular, the corresponding generalized Cartan matrix $A$ is symmetric.
 In symmetric case,  we assume $(\al_i,\al_i)=2$ for $i \in I$.

\begin{theorem} [\cite{VV09, R11}] \label{thm:categorification 2}
Assume that  the \KLR\ $R$ is symmetric and
the base field $\cor$ is of characteristic $0$.
 Then under the isomorphism \eqref{eq:KLRU}
in {\rm Theorem \ref{thm:categorification 1}},
the upper global basis  corresponds to
the set of the isomorphism classes of self-dual simple $R$-modules.
\end{theorem}

\subsection{R-matrices for \KLRs} \hfill

  For $|\beta|=n$ and $1\le a<n$,  we define $\vphi_a\in R( \beta)$
by
\eqn&&\ba{l}
  \vphi_a e(\nu)=
\begin{cases}
  \bl\tau_ax_a-x_{a}\tau_a\br e(\nu)
  & \text{if $\nu_a=\nu_{a+1}$,} \\[2ex]
\tau_ae(\nu)& \text{otherwise.}
\end{cases}
 \ea \eneqn
They are called the {\em intertwiners}. Since $\{\vphi_a\}_{1\le a<n}$
satisfies the braid relation, $\vphi_w\seteq\vphi_{i_1}\cdots\vphi_{i_\ell}$
does not depend on the choice of reduced expression $w=s_{i_1}\cdots s_{i_\ell}$.

For $m,n\in\Z_{\ge0}$,
let us denote by $w[{m,n}]$ the element of $\sym_{m+n}$  defined by
\eqn
&&w[{m,n}](k)=\begin{cases}k+n&\text{if $1\le k\le m$,}\\
k-m&\text{if $m<k\le m+n$.}\end{cases}
\eneqn

Let $\beta,\gamma\in \rtl^+$ with $|\beta|=m$, $|\gamma|=n$,
 and let
$M$ be an $R(\beta)$-module and $N$ an $R(\gamma)$-module. Then the
map $M\tens N\to N\conv M$ given by
$u\tens v\longmapsto \vphi_{w[n,m]}(v\tens u)$
is  $R( \beta)\tens R(\gamma )$-linear, and  hence it extends
to an $R( \beta +\gamma)$-module  homomorphism
\eqn &&R_{M,N}\col
M\conv N\To N\conv M. \eneqn

\medskip
Assume that the \KLR\ $R(\beta)$ is symmetric.
Let $z$ be an indeterminate  which is homogeneous of degree $2$, and
let $\psi_z$ be the graded algebra homomorphism
\eqn
&&\psi_z\col R( \beta )\to \cor[z]\tens R( \beta )
\eneqn
given by
$$\psi_z(x_k)=x_k+z,\quad\psi_z(\tau_k)=\tau_k, \quad\psi_z(e(\nu))=e(\nu).$$

For an $R( \beta )$-module $M$, we denote by $M_z$
the $\bl\cor[z]\tens R( \beta )\br$-module
$\cor[z]\tens M$ with the action of $R( \beta )$ twisted by $\psi_z$.
Namely,
\eqn&&
\ba{l}
e(\nu)(a\tens u)=a\tens e(\nu)u,\\[1ex]
x_k(a\tens u)=(za)\tens u+a\tens (x_ku),\\[1ex]
\tau_k(a\tens u)=a\tens(\tau_k u)
\ea
\label{eq:spt1}
\eneqn
for  $\nu\in I^\beta$,  $a\in \cor[z]$ and $u\in M$.
 Note that the multiplication by $z$ on $\cor[z]$ induces  an  $R(\beta)$-module endomorphism on $M_z$.
 For $u\in M$, we sometimes denote by $u_z$ the corresponding element $1\tens u$ of
the $R( \beta )$-module  $M_z$.

For a non-zero $M\in R(\beta)\smod$
and a non-zero $N\in R(\gamma)\smod$,
\eq&&\parbox{70ex}{%
let $s$ be the order of zero of $R_{M_z,N}\col  M_z\conv
N\To N\conv M_z$;
i.e., the largest non-negative integer such that the image of
$R_{M_z,N}$ is contained in $z^s(N\conv M_z)$.}
\label{def:s} \eneq
Note that such an $s$ exists because
$R_{M_z,N}$ does not vanish (\cite[Proposition 1.4.4 (iii)]{K^3}).
We denote by $\Rm_{M_z,N}$
the morphism $z^{-s}R_{M_z,N}$.

\Def Assume that $R(\beta)$ is symmetric.
For a non-zero $M\in R(\beta)\smod$
and a non-zero $N\in R(\gamma)\smod$,
let $s$ be an integer as in \eqref{def:s}.
We define $$\rmat{M,N}\col M\conv N\to N\conv M$$
by  $$\rmat{M,N} = \Rm_{M_z,N}\vert_{z=0},$$
and call it the {\em renormalized R-matrix}.
\edf

By the definition, the renormalized R-matrix
$\rmat{M,N}$ never vanishes.

We define also
$$\rmat{N,M}\col N\conv M\to M\conv N$$
by  $$\rmat{N,M} = \bl (-z)^{-t}R_{N,M_z}\br\vert_{z=0},$$
where $t$ is the order of zero of
$R_{N,M_z}$.

If $R(\beta)$ and $R(\gamma)$ are symmetric, then
$s$ coincides with the order of zero of
$R_{M,N_z}$, and $\bl z^{-s}R_{M_z,N}\br\vert_{z=0}=\bl (-z)^{-s}R_{M,N_z}\br\vert_{z=0}$
(see, \cite[(1.11)]{KKKO14}).

By the construction, if the composition
$(N_1 \conv \rmat{M,N_2}) \circ (\rmat{M,N_1} \conv N_2)$
for $M,N_1,N_2 \in R \smod$
doesn't vanish, then it is equal to $\rmat{M,N_1 \circ N_2}$.

\Def
A simple $R(\beta)$-module $M$ is called \emph{real} if $M \conv M$ is simple.
\edf

The  following lemma was used significantly in \cite{KKKO14}.
\Lemma[{\cite[Lemma 3.1]{KKKO14}}] \label{lem:simplicity}
Let $\beta_k\in\rtl^+$ and $M_k\in R(\beta_k)\smod$
$(k=1,2,3)$.
Let $X$ be an $R(\beta_1+\beta_2)$-submodule of $M_1\conv M_2$ and
$Y$ an $R(\beta_2+\beta_3)$-submodule of $M_2\conv M_3$
such that $X\conv M_3\subset M_1\conv Y$ as submodules of
$M_1\conv M_2\conv M_3$.
Then there exists an $R(\beta_2)$-submodule $N$ of $M_2$
such that
$X\subset M_1\conv N$ and $N\conv M_3\subset Y$.
\enlemma

One of the main results in \cite{KKKO14} is
\Th[{\cite[Theorem 3.2]{KKKO14}}] \label{thm:simplicity}
Let $\beta, \gamma \in \rootl^+$ and assume that $R(\beta)$ is symmetric.
Let $M$ be a real simple module in
$R(\beta) \smod$ and
$N$ a simple module in $R(\gamma) \smod$.
Then
\bnum
  \item $M \conv N$ and $N \conv M$ have simple socles and simple heads.
  \item Moreover, $\Im(\rmat{M,N})$ is equal to the head of $M \conv N$
and socle of $N \conv M$.
Similarly,
$\Im(\rmat{N,M})$ is equal to the head of $N \conv M$
and socle of $M \conv N$.
\end{enumerate}
\enth

We will use the following convention frequently.
\Def
For simple $R$-modules $M$ and $N$, we denote by  $M \hconv N$   the head of
$M \conv N$ and  by $M \sconv N$ the socle of $M \conv N$ .
\edf

\section{Simplicity of heads and socles of convolution products}

{\em In this section,
we assume that $R(\beta)$ is symmetric} for any $\beta\in\rtl^+$,
i.e., $Q_{ij}(u,v)$ is a function in $u-v$ for any $i,j\in I$.

We also work always in the category of graded modules.
For the sake of simplicity, {\em
we simply say that $M$ is an $R$-module instead of saying
that $M$ is a graded $R(\beta)$-module for $\beta\in\rtl^+$.}
We also sometimes ignore grading shifts if there is no  danger of confusion.
Hence, for $R$-modules $M$ and $N$, {\em we sometimes say that
$f\col M\to N$ is a homomorphism if
$f\col q^aM\to N$ is a morphism in $R\gmod$ for some $a\in\Z$.}
If we want to emphasize that $f\col q^aM\to N$ is a morphism in $R\gmod$,
we say so.

\subsection{Homogeneous degrees of R-matrices}

\Def
For non-zero $M,N \in R \gmod$, we denote by $\Lambda(M,N)$ the homogeneous degree of the R-matrix
$\rmat{M,N}$.
\edf

Hence
\eqn
&&\Rm_{M_z,N}:M_z \conv N \to   q^{-\Lambda(M,N)} N \conv M_z \quad \text{and}\\
&&\rmat{M,N}:M\conv N \to q^{-\Lambda(M,N)} N \conv M
\eneqn
are morphisms in $R\gMod$ and in $R \gmod$, respectively.

\Lemma \label{lem:tLa even}
 For non-zero $R$-modules $M$ and $N$, we have
\eqn
\Lambda(M,N) \equiv \bl\wt(M),\wt(N)\br \mod 2.
\eneqn
\enlemma
\Proof
Set $\beta\seteq-\wt(M)$ and $\gamma\seteq-\wt(N)$.
By \cite[(1.3.3)]{K^3}, the homogeneous degree of  $R_{M_z,N}$ is $-(\beta,\gamma)+2\inp{\beta,\gamma}$, where
$\inp{{\scbul,\scbul}}$ is the symmetric bilinear form on $\rootl$ given by
$\inp{\al_i,\al_j}=\delta_{ij}$.
Hence  $\Rm_{M_z,N}=z^{-s}R_{M_z,N}$  has degree $-(\beta,\gamma)+2\inp{\beta,\gamma}-2s$.
\QED

\Def
For non-zero $R$-modules $M$ and $N$, we set
\eqn
\tL(M,N) \seteq \frac{1}{2}\bl\Lambda(M,N) +(\wt(M), \wt(N))\br  \in \Z .
\eneqn
\edf

\Lemma\label{lem:selfdual}
Let $M$ and $N$ be self-dual simple modules.
If one of them is real, then
\eqn
q^{\tL(M,N)} M \hconv N
\eneqn
is a self-dual simple module.
\enlemma
\Proof
Set $\beta=\wt(M)$ and $\gamma=\wt(N)$.
Set $M \hconv N = q^c L$ for some self-dual simple module $L$ and some $c \in \Z$.
Then we have
\eqn
M \conv N \epito q^c L \monoto q^{-\Lambda(M,N)} N \conv M,
\eneqn
since $M \hconv N=\Im \rmat{M,N}$.
Taking dual, we obtain
\eqn
q^{\Lambda(M,N)+(\beta,\gamma)} M \conv N \epito q^{-c} L \monoto q^{(\beta,\gamma)} N \conv M.
\eneqn
In particular, $q^{-c-\Lambda(M,N)-(\beta,\gamma)} L $ is a simple quotient of $M\conv N$.
Hence we have $c=-c-\Lambda(M,N)-(\beta,\gamma)$, which implies $c= -\tL(M,N)$.
\QED

\Lemma \label{lem:composition}
\bni
 \item  Let $M_k$ be non-zero modules $( k=1,2,3)$, and let $\vphi_1: L \to  M_1\conv M_2$ and
$\vphi_2:  M_2 \conv M_3 \to L'$ be non-zero homomorphisms.  Assume further that $M_2$ is a simple module.
Then the composition
\eqn
L \conv M_3 \To[\vphi_1 \circ M_3] M_1 \conv M_2 \conv M_3 \To[M_1 \circ \vphi_2] M_1 \conv L'
\eneqn
does not vanish.
\item Let $M$ be a simple module and let $N_1, N_2$ be non-zero modules.
Then the composition
\eqn
M \conv N_1 \conv N_2 \To[\rmat{M,N_1} \circ N_2]
N_1  \conv M \conv N_2 \To[N_1 \circ \; \rmat{M,N_2}]
N_1 \conv N_2 \conv M
\eneqn
coincides with $\rmat{M,N_1 \circ N_2}$,
 and the composition
\eqn
N_1 \conv N_2 \conv M \To[N_1\circ\; \rmat{N_2,M}]
N_1  \conv M \conv N_2 \To[\rmat{N_1,M} \circ N_2]
 M \conv N_1 \conv N_2
\eneqn
coincides with $\rmat{N_1 \circ N_2,M}$.

In particular, we have
\eqn
\Lambda(M, N_1\conv N_2) = \Lambda(M,N_1)+\Lambda(M,N_2)
\eneqn
and
\eqn
\Lambda(N_1\conv N_2,M) = \Lambda(N_1,M)+\Lambda(N_2,M).
\eneqn
\end{enumerate}
\enlemma
\Proof
 (i)  Assume  that the composition vanishes. Then we have $\im \vphi_1 \conv M_3 \subset M_1 \conv \Ker \vphi_2$.
By Lemma \ref{lem:simplicity},
 there is a submodule $N$ of $M_2$
such that $\im \vphi_1  \subset M_1 \conv N$ and $N \conv M_3 \subset \Ker \vphi_2$.
The first inclusion implies that $N\neq 0$  since $\vphi_1$ is non-zero,  and the second  implies  $N\neq M_2$ since $\vphi_2$ is non-zero.
It contradicts  the simplicity of  $M_2$.

(ii) It is  enough to show that the compositions
$(N_1 \conv \rmat{M,N_2}) \circ (\rmat{M,N_1} \conv N_2)$ and
$ (\rmat{N_1,M} \conv N_2) \circ (N_1 \conv \rmat{N_2,M})$
 do not vanish, but these immediately follow from (i).
\QED

\subsection{Properties of $\tL(M,N)$ and $\de(M,N)$}
\Lemma \label{lem:even}
Let $M$ and $N$ be simple $R$-modules.
Then we have
\bnum
   \item $\Lambda(M,N)+\Lambda(N,M) \in 2 \Z_{\ge 0}.$
   \item If $\Lambda(M,N)+\Lambda(N,M) =2m$ for some $m \in \Z_{\ge 0}$, then
  \eqn
\Rm_{M_z,N} \circ \Rm_{N,M_z}=z^m \id_{N \conv M_z} \quad
\text{and} \quad
\Rm_{N,M_z} \circ \Rm_{M_z,N}=z^m \id_{M_z \conv N}
  \eneqn
  up to constant multiples.
\end{enumerate}
\enlemma
\Proof
By \cite[Proposition 1.6.2]{K^3}, the morphism
\eqn
\Rm_{N,M_z} \circ \Rm_{M_z,N} : M_z \circ N \to
M_z \circ N
\eneqn
is equal to $f(z) \id_{M_z \circ N}$ for some $0 \neq f(z) \in \cor[z]$.
Since $\Rm_{N,M_z} \circ \Rm_{M_z,N}$ is  homogeneous of degree $\Lambda(M,N)+\Lambda(N,M)$,
we have $f(z)=c z^{\frac{1}{2}(\Lambda(M,N)+\Lambda(N,M))}$ for some $c \in \cor^\times$.
\QED

\Def
For non-zero modules $M$ and $N$, we set
\eqn
\de(M,N) = \frac{1}{2}\bl\Lambda(M,N)+\Lambda(N,M)\br.
\eneqn
\edf
Note that if $M$ and $N$ are simple modules, then  we have
$\de(M,N) \in \Z_{\ge 0}$.
Note also that
if $M,N_1,N_2$ are simple modules, then we have
$\de(M,N_1 \conv N_2) =\de(M,N_1)+\de(M,N_2)$ by Lemma \ref{lem:composition} (ii).

\Lemma[\cite{KKKO14}] \label{lem:commute_equiv}
Let $M,N$ be simple modules and assume that one of them is real. Then the following conditions are equivalent.
\bnum
  \item $\de(M,N)=0$.
  \item $\rmat{M,N}$ and $\rmat{N,M}$ are inverse to each other up to a constant multiple.
  \item $M \conv N$ and $N \conv M$ are isomorphic up to a grading shift.
  \item $M \hconv N$ and $N \hconv M$ are isomorphic up to a grading shift.
  \item $M \conv N$ is simple.
\end{enumerate}
\enlemma
\Proof
By specializing the equations in Lemma \ref{lem:even} (ii) at $z=0$, we obtain that $\de(M,N)=0$ if and only if $\rmat{M,N} \circ \rmat{N,M} = \id_{N\conv M}$ and $\rmat{N,M} \circ \rmat{M,N} = \id_{M\conv N}$ up to non-zero constant multiples. Hence the conditions (i) and (ii) are equivalent.

The conditions  (ii), (iii), (iv), and (v) are equivalent by \cite[Theorem 3.2, Proposition 3.8, and Corollary 3.9]{KKKO14}.
\QED

\Def
Let $M,N$ be simple modules.
\bnum
  \item We say that $M$ and $N$ \emph{commute} if $\de(M,N)=0$.
  \item We say that $M$ and $N$  are \emph{simply-linked} if $\de(M,N)=1$.
\end{enumerate}
\edf

\Prop  \label{prop:real simple}
Let $M_1,\ldots, M_r$ be a commuting family of real simple modules.
Then the convolution product
\eqn
M_{1} \conv \cdots \conv M_{r}
\eneqn
is a real simple module.
\enprop
\Proof We shall first show the simplicity of the convolutions.
By induction on $r$, we may assume that $ M_{2} \conv \cdots \conv M_{r}$
is a simple module.
Then we have
\eqn
\de(M_{1}, M_{2} \conv \cdots \conv M_{r})
=\sum_{s=2}^r \de(M_{1},M_{s})=0
\eneqn
so that
$M_{1} \conv \cdots \conv M_{r}$ is simple by Lemma \ref{lem:commute_equiv}.

Since $(M_{1} \conv \cdots \conv M_{r}) \conv (M_{1} \conv \cdots \conv M_{r})$ is also simple,
$M_{1} \conv \cdots \conv M_{r}$ is real.
\QED

\begin{definition}
Let $M_1,\ldots,M_m$ be real
simple modules. Assume that they commute with each other. We set
\eqn
&&M_1\sodot M_2\seteq q^{\tLa(M_1,M_2)}M_1\conv M_2,\\
&&\sodot_{1 \le k \le m}M_k
\seteq(\cdots(M_1\sodot M_2)\cdots)\sodot M_{m-1})\sodot M_m\\
&&\hs{20ex}\simeq q^{\sum_{1 \le i <j\le m} \tLa(M_i,M_j)}M_1 \conv \cdots \conv M_m.
\eneqn
\end{definition}
It is invariant  under the permutations of $M_1,\ldots,M_m$.

\Lemma\label{lem:multipro}
Let $M_1,\ldots, M_m$ be real simple modules commuting with each other.
Then for any $\sigma\in\sym_m$, we have
$$\nconv_{1\le k\le m} M_k\simeq \nconv_{1\le k\le m} M_{\sigma(k)}\quad\text{in $R\gmod$.}$$
Moreover, if the $M_k$'s are self-dual,
then so is
$\nconv_{1\le k\le m} M_k$.
\enlemma
\Proof
It follows from Lemma~\ref{lem:selfdual} and
$q^{\tL(M_i,M_j)}M_i\conv M_j\simeq
q^{\tL(M_j,M_i)}M_j\conv M_i$.
\QED

\Prop \label{prop:inequalities}
Let $f:N_1 \to N_2$ be a morphism between non-zero
$R$-modules $N_1, N_2 $
and let $M$ be a non-zero $R$-module.
\bnum
\item
If $\Lambda(M,N_1)=\Lambda(M,N_2)$, then the following diagram is commutative:
  \eqn
  \xymatrix@C=5em{
  M\conv N_1 \ar[r]^{\rmat{M,N_1}} \ar[d]_{M \circ f} &%
N_1 \conv M \ar[d]^{f \circ M} \\
  M\conv N_2 \ar[r]^{\rmat{M,N_2}}& %
N_2 \conv M.
  }\eneqn
\item If $\Lambda(M,N_1) < \Lambda(M,N_2)$, then the composition
  \eqn
  M \conv N_1 \To[M \circ f] M \conv N_2 \To[\rmat{M,N_2}] N_2 \conv M
  \eneqn
  vanishes.
 \item If $\Lambda(M,N_1) > \Lambda(M,N_2)$, then the composition
  \eqn
M \conv N_1 \To[\rmat{M,N_1}] N_1 \conv M \To[f \circ M] N_2 \conv M
  \eneqn
  vanishes.
 \item If $f$ is surjective, then we have
 \eqn
 \Lambda(M,N_1) \ge \Lambda(M,N_2) \quad \text{and} \quad
 \Lambda(N_1,M) \ge \Lambda(N_2,M)
 \eneqn
If $f$ is injective, then we have
\eqn
\Lambda(M,N_1) \le \Lambda(M,N_2) \quad \text{and}\quad
\Lambda(N_1,M) \le \Lambda(N_2,M)
\eneqn
\end{enumerate}

\enprop
\Proof

Let $s_i$  be the order of zero of $R_{M_z,N_i}$ for $i=1,2$.
Then we have  $\Lambda(M,N_1) - \Lambda(M,N_2)=2(s_2-s_1)$.

Set $m\seteq\min\{s_1,s_2\}$.
Then the following diagram is commutative:
  \eqn
  \xymatrix@C=10em{
  M_z\conv N_1 \ar[r]^{z^{-m}R_{M_z,N_1}} \ar[d]_{M_z \conv f} &  N_1 \conv M_z \ar[d]^{f \conv M_z} \\
  M_z\conv N_2 \ar[r]^{z^{-m}R_{M_z,N_2}}&  N_2 \conv M_z.
  }\eneqn

(i) If $s_1=s_2$, then by specializing $z=0$ in the above diagram,
we obtain the commutativity of the diagram in (i).

(ii) If $s_1 > s_2$, then we have
\eqn
z^{-m}R_{M_z,N_1}=z^{s_1-m} \bl z^{-s_1}R_{M_z,N_1}\br
\eneqn
so that $z^{-m}R_{M_z,N_1}|_{z=0}$ vanishes.
Hence
we have
\eqn
\rmat{M,N_2} \circ (M\conv f) = z^{-m}R_{M_z,N_2}|_{z=0} \circ (M\conv f) =0,
\eneqn
as desired. In particular, $f$ is not surjective.

(iii) Similarly, if $s_1 < s_2$, then we have
$
(f \conv M) \circ \rmat{M,N_1} =0,
$
and $f$ is not injective.

(iv) The statements for $\Lambda(M,N_1)$ and   $\Lambda(M, N_2)$ follow
from (ii) and (iii).
The other statements can be shown in a similar way.
\QED

\Prop\label{prop:rmat}
 Let $M$ and $N$ be simple modules.
We assume that one of them is real.
Then we have
$$\Hom_{R\smod}(M\conv N,N\conv M)=\cor\,\rmat{M,N}.$$
%In particular, if $M$ is a real simple module, then $\rmat{M,M}=c \id_{M\conv M}$ for some $c \in \cor^\times$.
\enprop
\Proof
 Since the other case can be proved similarly, we assume that $M$ is real.
Let $f\cl M\conv N\to N\conv M$ be a morphism.
 Note that we have $\rmat{M,\,M\circ N}=M\conv \rmat{M,N}$ and $\rmat{M,\,N\circ M}=\rmat{M,N}\conv M$   by Lemma \ref{lem:composition} (ii) and by the fact that $\rmat{M,M}=\id_{M\conv M}$ up a constant multiple.
Thus, by Proposition \ref{prop:inequalities}, we have a commutative diagram (up to a constant multiple)
$$\xymatrix@C=10ex{
M\conv M\conv N\ar[r]^{M\circ \rmat{M,N}}\ar[d]_{M\circ f}
&M\conv N\conv M\ar[d]_{f\circ M}\\
M\conv N\conv M\ar[r]^{\rmat{M,N}\circ M}
&N\conv M\conv M
}$$
Hence we have
$$M\conv {\rm Im}(\rmat{M,N})\subset f^{-1}\bl {\rm Im}(\rmat{M,N})\br\conv M.$$
Hence there exists a submodule $K$ of $N$ such that
${\rm Im}(\rmat{M,N})\subset K\conv M$ and
$M\conv K\subset f^{-1}\bl{\rm Im}(\rmat{M,N})\br$ by Lemma \ref{lem:simplicity}.
Since $K\not=0$, we have $K=N$.
Hence
$f(M\conv N)\subset {\rm Im}(\rmat{M,N})$, which means that
$f$ factors as
$M\conv N\to \soc(N\conv M)\monoto N\conv M$.
It remains to remark that
$\Hom_{R\smod}\bl M\conv N,\soc(N\conv M)\br=\cor\,\rmat{M,N}$.
\QED

\Prop \label{pro:subquotient}
Let $L$, $M$ and $N$ be simple modules.
Then we have
\eq&&\ba{l}
\La(L,S)\le\La(L,M)+\La(L,N),\
\La(S,L)\le\La(M,L)+\La(N,L)\ \text{and}\\[1ex]
\de(S,L)\le\de(M,L)+\de(N,L)\ea\label{eq:LMN<}
\eneq
for any subquotient $S$ of $M\conv N$.
Moreover, when $L$ is real, the following conditions are equivalent.
\bni
\item $L$ commutes with $M$ and $N$.
\item Any simple subquotient $S$ of $M\conv N$ commutes with
$L$ and satisfies $\La(L,S)=\La(L,M)+\La(L,N)$.
\item Any simple subquotient $S$ of $M\conv N$ commutes with $L$ and satisfies
$\La(S,L)=\La(M,L)+\La(N,L)$.
\ee
\enprop
\Proof
The inequalities \eqref{eq:LMN<} are consequences of Proposition \ref{prop:inequalities}.
Let us show the equivalence
of (i)--(iii).

Let $M\conv N=K_0\supset K_1\supset \cdots\supset  K_\ell\supset K_{\ell+1}=0$ be a Jordan-H\"older series of $M\conv N$.
Then the renormalized R-matrix $\Rm_{L_z,M\circ N}=(M\conv\Rm_{L_z,N})\circ(\Rm_{L_z,M}\conv N)
\cl L_z\conv M\conv N\to M\conv N\conv L_z$
is homogeneous of degree $\La(L,M)+\La(L,N)$ and it sends $L_z\conv K_k$ to
$K_k\conv L_z$ for any $k\in\Z$. Hence
$f\seteq\rmat{L,M\circ N}=\Rm_{L_z,M\circ N}\vert_{z=0}$ sends
$L\conv K_k$ to $K_k\conv L$.

\medskip
\noi
First assume (i).
Then $f$ is an isomorphism.
Hence $f\vert_{L\circ K_k}\cl L\conv K_k\to K_k\conv L$ is injective.
By comparing their dimension, $f\vert_{L\circ K_k}$ is an isomorphism,
Hence $f\vert_{L\circ (K_k/K_{k+1})}$ is an isomorphism of homogeneous degree
$\La(L,M)+\La(L,N)$.
Hence we obtain (ii).

\medskip
\noi
Conversely, assume (ii).
Then,
$\Rm_{L_z,M\circ N}\vert_{L_z\conv (K_k/K_{k+1})}$ and $\Rm_{L_z,K_k/K_{k+1}}$
have the same homogeneous degree, and hence they should coincide.
It implies that $f\vert_{L\circ(K_k/K_{k+1})}=\rmat{L,K_k/K_{k+1}}$ is an isomorphism
for any $k$.
Therefore $f=(M\conv\rmat{L,N})\circ(\rmat{L,M}\conv N)$ is an isomorphism,
which implies that $\rmat{L,N}$ and $\rmat{L,M}$ are isomorphisms.
Thus we obtain (i).

\medskip
\noi
Similarly, (i) and (iii) are equivalent.
\QED

\Lemma \label{lem:LMN1}
Let $L$, $M$ and  $N$ be simple modules.
 We assume that $L$ is real and commutes with $M$.
Then the diagram
$$\xymatrix@C=10ex
{L\conv (M\conv N)\ar[r]^{\rmat{L,M\circ N}}\ar[d]
&(M\conv N)\conv L\ar[d]\\
L\conv (M\hconv N)\ar[r]^{\rmat{L,M\shconv N}}
&(M \hconv N)\conv L}
$$
commutes.
\enlemma
\Proof
Otherwise
the composition

$$
\xymatrix@C=8ex{
L\conv M\conv N\ar[r]^-\sim_-{\rmat{L,M}\circ N}
&M\conv L\conv N\ar[r]_-{M\circ\rmat{L,N}}&M\conv N\conv L
\ar[r]&
(M\hconv N)\conv L}
$$
vanishes by Proposition \ref{prop:inequalities}.
Hence we have
$$M\circ \Im(\rmat{L,N})\subset \Ker(M\conv N\to M\hconv N)\conv L.$$
Hence, by Lemma \ref{lem:simplicity}, there exists a submodule $K$ of $N$ such that
$$\text{$\Im(\rmat{L,N})\subset K\conv L$ and
$M\conv K\subset\Ker(M\conv N\to M\hconv N)$.}$$
The first inclusion implies $K\not=0$ and the second implies $K\not=N$,
which contradicts the simplicity of $N$.
\QED

The following lemma can be proved similarly.
\Lemma\label{lem:LMN2}
Let $L$, $M$ and  $N$ be simple modules.
We assume that $L$ is real and commutes with $N$.
Then the diagram
$$\xymatrix@C=10ex
{(M\conv N)\conv L\ar[r]^{\rmat{M\circ N,L}}\ar[d]
&L\conv(M\conv N)\ar[d]\\
(M\hconv N)\conv L\ar[r]^{\rmat{M\shconv N,L}}
&L\conv(M\hconv N)}
$$
 commutes. %
\enlemma

The following proposition follows from Lemma~\ref{lem:LMN1} and Lemma~\ref{lem:LMN2}.
\Prop \label{prop:Lambda decomp}
Let $L$, $M$ and  $N$ be simple modules. Assume that $L$ is real.
Then we have:
\bnum
\item
If $L$ and $M$ commute, then
$$\La(L,M\hconv N)=\La(L,M)+\La(L,N).$$

\item
If $L$ and $N$ commute, then
$$\La(M\hconv N,L)=\La(M,L)+\La(N,L).$$
\ee
\enprop

%%%%%%%%%%%%%%%%%%%%%

\Prop \label{prop:socNM}
Let $M$ be a real simple module and let $N$ be a module with a simple socle.
If the following diagram
$$\xymatrix@C=12ex
{\wb{\soc(N)\conv M}\ar[r]^{\rmat{\soc(N),M}}\ar@{>->}[d]
&\wb{M\conv\soc(N)}\ar@{>->}[d]\\
N\conv M\ar[r]^{\rmat{N,M}}&M\conv N
}
$$ commutes up to a non-zero constant multiple, then
$\soc\bl M\conv\soc(N)\br$ is
equal to the socle of $M\conv N$.
In particular, $M\conv N$ has a simple socle.
\enprop
\Proof
Let $S$ be an arbitrary simple submodule of $M\conv N$.
Then we have the following commutative diagram:
$$\xymatrix@C=12ex
{\wb{S\conv M_z}\ar[r]^{ R_{S\circ M_z} }  \ar@{>->}[d]
&\wb{M_z\conv S}\ar@{>->}[d]\\
M\conv N\conv M_z\ar[r]^{R_{M\circ N,M_z}}
&M\conv M\conv N.
}$$
 By multiplying $z^{-m}$, where  $m$ be the order of zero of $R_{M\conv N, M}$,  and
specializing at $z=0$,
we have a commutative diagram  (up to a constant multiple):
$$\xymatrix@C=12ex
{\wb{S\conv M}\ar[r] \ar@{>->}[d]
&\wb{M\conv S}\ar@{>->}[d]\\
M\conv N\conv M\ar[r]^{M\circ\rmat{N,M}}&M\conv M\conv N.
}$$
% Note that the upper horizontal morphism can vanish, but it doesn't matter.

 Here, we use the fact that
$\rmat{M\conv N, M} = (\rmat{M,M} \conv N) \circ (M\conv \rmat{N,M})$ from Lemma \ref{lem:composition} and the fact that
$\rmat{M,M}$
is equal to $\id_{M\conv M}$ up to a non-zero constant multiple, because $M$ is a real simple module.

It follows that  $S\conv M\subset M\conv (\rmat{N,M})^{-1}(S)$.
Hence there exists a submodule $K$ of $N$ such that
$S\subset M\conv K$ and $K\conv M\subset (\rmat{N,M})^{-1}(S)$
by Lemma \ref{lem:simplicity}.
Hence $K\not=0$ and $\soc(N)\subset K$ by the assumption.
Hence $\rmat{N,M}\bl\soc(N)\conv M\br
\subset \rmat{N,M}\bl K\conv M\br\subset S$.
Since $\rmat{N,M}\bl\soc(N)\conv M\br$ is non-zero by the assumption, we have
$\rmat{N,M}\bl\soc(N)\conv M\br=S$.
Thus we obtain the desired result.
\QED

The following is a dual form of the preceding proposition.
\Prop\label{prop:hdMN}
Let $M$ be a real simple module.
Let $N$ be a module with a simple head.
If the following diagram
$$\xymatrix@C=12ex
{{M\conv N}\ar[r]^{\rmat{M,N}}\ar@{->>}[d]
&{N\conv M}\ar@{->>}[d]\\
M\conv \hd(N)\ar[r]^{\rmat{M,\hd(N)}}&\hd(N)\conv M
}
$$ commutes up to a non-zero constant multiple, then
$M\hconv \hd(N)$ is
equal to the simple head of $M\conv N$.
\enprop

\begin{prop} \label{prop:3simple}
Let $L$, $M$ and  $N$ be simple modules.
We assume that $L$ is real and
one of $M$ and $N$ is real.
\begin{enumerate}
\item[{\rm (i)}]
If $\La(L,M\hconv N)=\La(L,M)+\La(L,N)$,
then
$L\conv M\conv N$ has a simple head and
$N\conv M\conv L$ has a simple socle.
\item[{\rm (ii)}]
If $\La(M\hconv N,L)=\La(M,L)+\La(N,L)$,
then
$M\conv N\conv L$ has a simple head and
$L\conv N\conv M$ has a simple socle.
\item[{\rm (iii)}]
If $\de(L,M\hconv N)=\de(L,M)+\de(L,N)$,
then
$L\conv M\conv N$ and $M\conv N\conv L$  have  simple heads, and
$N\conv M\conv L$ and $L\conv N\conv M$ have  simple socles.
\end{enumerate}
\end{prop}

\begin{proof}
(i)\ Denote $k=\La(L,M\hconv N)=\La(L,M)+\La(M,N)$ and $m=\La( M,N)$. Then the diagram
$$\xymatrix@C=12ex@R=3ex
{{L\conv M\conv N}\ar[r]^{\rmat{L,M\circ N}}\ar@{->>}[d]
& q^{-k}{M\conv N\conv L}\ar@{->>}[d]
\\
\wb{L\conv (M\hconv N)}\ar[r]^{\rmat{L,M\shconv N}}\ar@{>->}[d]
&\wb{q^{-k}(M\hconv N)\conv L}\ar@{>->}[d]\\
q^{-m}L\conv N\conv M\ar[r]^{\rmat{L,N\circ M}}&q^{-k-m} N\conv M\conv L
}
$$
commutes. Hence
 Proposition \ref{prop:socNM} and Proposition \ref{prop:hdMN}  imply that $L\conv M\conv N$ has a simple head and $N\conv M\conv L$ has a simple socle.
(ii) are proved similarly.

\medskip\noindent
(iii)\ If $\de(L,M\hconv N)=\de(L,M)+\de(L,N)$,
then we have $\La(L,M\hconv N)=\La(L,M)+\La(L,N)$
and $\La(M\hconv N,L)=\La(M,L)+\La(N,L)$ by Proposition \ref{prop:inequalities}. Thus
 the statements in (iii) follow from (i) and (ii).
\end{proof}

%%Simply linked impies length 2
\Prop\label{Prop: l2}
Let $M$ and $N$ be simple modules.
Assume that one of them is real and $\de(M,N)=1$.
Then we have an exact sequence
\eqn
&&0\to  M\sconv N \to M\conv N\to M\hconv N\to 0.
\eneqn
In particular, $M\conv N$ has length $2$.
\enprop
\Proof In the course of the proof, we ignore the grading.

Set $X=M_z\circ N$
and $Y=N\circ M_z$.
By $\Rm_{N,M_z}\cl Y\monoto X$ let us regard $Y$ as a submodule of $X$.
By the condition, we have
$\Rm_{N,M_z}\circ \Rm_{M_z,N}=z\id_X$ up to a constant multiple (see Lemma \ref{lem:even} (ii)),
and hence we have
$$zX\subset Y\subset X.$$

We have an exact sequence
$$ 0\To \dfrac{Y}{zX}\To\dfrac{X}{zX}\To\dfrac{X}{Y}\To0.$$
Since
$$M\conv N\simeq\dfrac{X}{zX}\epito \dfrac{X}{Y}\monoto \dfrac{z^{-1}Y}{Y}
\simeq N\conv M,$$
we have $\dfrac{X}{Y}\simeq M\hconv N$ by Proposition \ref{prop:rmat}.
Similarly,
$$N\conv M\simeq\dfrac{Y}{zY}\epito \dfrac{Y}{zX}\monoto \dfrac{X}{zX}
\simeq M\conv N$$
implies that
$\dfrac{Y}{zX}\simeq  M\sconv N $ by Proposition \ref{prop:rmat}.
\QED

%%Lemmas on convolution products of composition length 2.
\Lemma \label{lem:crde}
Let $M$ and $N$ be simple  modules.
Assume that one of them is real.
If there is an exact sequence
$$0\to q^mX\To M\circ N\To q^n Y\To 0$$
for self-dual simple modules $X$, $Y$ and
integers $m$, $n$, then we have
$$\de(M,N)=m-n.$$
\enlemma
\Proof
We may assume that $M$ and $N$ are self-dual without loss of generality.
Then we have
$n=-\tLa(M,N)$.
Since $q^mX\simeq q^{\La(N,M)}N\hconv M
\simeq q^{\La(N,M)-\tL(N,M)}\bl q^{\tL(N,M)}N\hconv M\br$, we have
$m=\La(N,M)-\tL(N,M)$.
Thus we obtain
$$m-n=\La(N,M)-\tL(N,M)+\tLa(M,N)=\de(M,N).$$
\QED

\Lemma\label{lem:short}
Let $M$ and $N$ be simple modules.
Assume that one of them is real.
If the equation
$$[M][N]=q^m[X]+q^n[Y]$$
holds in $K(R\gmod)$ for self-dual simple modules $X$, $Y$ and
integers $m$, $n$ such that $m\ge n$,
then we have
\bnum
\item $\de(M,N)=m-n>0$,
\item there exists an exact sequence
$$0\To q^mX\To M\conv N\To q^nY\To0,$$
\item
 $ q^mX$ is  the  socle of $M\conv N$ and $q^nY$ is
 the
head of
$M\conv N$.
\ee
\enlemma
\Proof
First note that $\de(M,N)>0$ since $M\conv N$ is not simple.
By the assumption,
there exists either an exact sequence
$$0\To q^mX\To M\conv N\To q^nY\To0,$$
or
$$0\To q^nY\To M\conv N\To q^mX\To0.$$
The  second sequence cannot exist by Lemma \ref{lem:crde}
because $\de(M,N)=n-m\le0$.
Hence the first sequence exists, and the assertion  (iii) follows from
 Theorem \ref{thm:simplicity}.
\QED

%%%%%%%%%%%%%%%%%%%%%

\Prop \label{prop:real}
Let $X,Y,M$ and $N$ be simple $R$-modules.
Assume that there is  an exact sequence
\eqn
0 \to X \to M \conv N \to Y \to 0,
\eneqn
$X \conv N$ and $Y \conv N$ are simple and $X \conv N \not\simeq Y \conv N$
as ungraded modules.
Then $N$ is a real simple module.
\enprop
\Proof
Assume that $N$ is not real.
Then $N \conv N$ is reducible and we have $\rmat{N,N} \neq c \id_{N \conv N}$ for any $c \in \cor$ by \cite[Corollary 3.3]{KKKO14}.
Note that $N \conv N$
is of length $2$, because $M\conv N \conv N$ is of length $2$.

Let $S$ be a simple  submodule of $N \conv N$.
Consider an exact sequence
$$0\To X\conv N\To M\conv N\conv N\To Y\conv N\To0.$$
Then we have
\eq (X \conv N) \cap (M\conv S) =0.
\label{eq:SN}
\eneq
Indeed, if $(X\conv N) \subset (M \conv S)$, then there exists a submodule $Z$ of $N$ such that
$X \subset M \conv Z$ and $Z \conv N \subset S$ by \cite[Lemma 3.1]{KKKO14}.
It contradicts the simplicity of $N$.
Thus \eqref{eq:SN} holds.

Note that \eqref{eq:SN} implies
$$M\conv S\simeq Y\conv N$$
since $Y \conv N$ is simple.

\medskip\noi
(a)\ Assume first
that $N \conv N$ is semisimple so that $N \conv N=S \oplus S'$ for some simple submodule
 $S'$ of $N \conv N$.
 Then
$M\conv S\simeq Y\conv N\simeq M\conv S'$.
Hence $M\conv S\simeq X\conv N\simeq M\conv S'$.
Therefore we obtain $X \conv N \simeq Y \conv N,$ which is a contradiction.

\medskip\noi
(b)\   Assume that $N \conv N$ is not semisimple so that $S$ is a unique non-zero proper submodule of
  $N \conv N$ and $(N \conv N) / S$ is a unique non-zero proper quotient of $N \conv N$.
  Without loss of generality, we may assume that $\cor$ is algebraically closed (\cite[Corollay 3.19]{KL09}).
  Let $x \in \cor$ be an eigenvalue of $\rmat{N,N}$.
 Since $\rmat{N,N} \not\in\cor\id_{N \conv N}$, we have
 $0  \subsetneq \Im(\rmat{N,N}-x \id_{N\conv N}) \subsetneq N\conv N$.
It follows that
\eqn
S = \Im(\rmat{N,N}-x \id_{N\conv N}) \simeq (N \conv N) /S,
\eneqn
and  hence we have an exact sequence
\eqn
0 \To M\conv S \To M \conv N \conv N \To M \conv \bl(N \conv N) /S\br  \To 0.
\eneqn
Since $M \conv N \conv N $ is of length $2$, we have
\eqn
X \conv N \simeq M \conv S \simeq M\conv\bl(N \conv N) /S\br\simeq  Y \conv N,
\eneqn
 which is a contradiction.
 \QED

 \Cor \label{cor:real}
Let $X,Y,N$ be simple $R$-modules and let $M$ be a real simple  $R$-module.
 If we have an exact sequence
\eqn
0 \to X \to M \conv N \to Y \to 0
\eneqn
and if $X \conv N$ and $Y \conv N$ are simple,
then $N$ is a real simple module.
\encor
\Proof
Since $M$ is real and $M \conv N$ is not simple, $X$ is not isomorphic to $Y$
as an ungraded module by Lemma \ref{lem:commute_equiv} (iv).
It follows that $X \conv N$ is not isomorphic to $Y \conv N$,
because $K(R\gmod)$ is a domain so that
$[X\conv N]=q^m [Y \conv N]$ for some $m \in \Z$ implies $[X]=q^m [Y]$.
Now the assertion follows from Proposition \ref{prop:real}.
\QED

%%Lemmas on multiple (h)convolutions:
\Lemma\label{lem:MN}
Let $\{M_i\}_{1\le i\le n}$ and $\{N_i\}_{1\le i\le n}$ be a pair of
commuting families of real simple modules.
We assume that
\bnam
 \item $\{M_i\hconv N_i\}_{1\le i\le n}$ is a commuting family of real simple
modules,
\item $M_i\hconv N_i$ commutes with $N_{j}$
for any $1\le i,j\le n$.
\ee

Then we have
$$(\conv_{1\le i\le n}M_i)\hconv (\conv_{1\le j\le n}N_j)
\simeq \conv_{1\le i\le n}(M_i\hconv N_i)\quad
\text{up to a grading shift.}
$$
\enlemma
\Proof
 Since $\conv_{1\le i\le n}(M_i\hconv N_i)$ is simple, it is enough to give
an epimorphism $(\conv_{1\le i\le n}M_i)\conv (\conv_{1\le j\le n}N_j)
\epito \conv_{1\le i\le n}(M_i\hconv N_i)$.
We shall show it by induction on $n$.
For $n>0$, we have
\eqn
&&(\conv_{1\le i\le n}M_i)\conv (\conv_{1\le j\le n}N_j)\simeq
(\conv_{1\le i\le n-1}M_i)\conv M_n\conv N_n\conv (\conv_{1\le j\le n-1}N_j)\\
&&\epito (\conv_{1\le i\le n-1}M_i)\conv (M_n\hconv N_n) \conv (\conv_{1\le j\le n-1}N_j) \\
&&\simeq (\conv_{1\le i\le n-1}M_i)\conv (\conv_{1\le j\le n-1}N_j) \conv  (M_n\hconv N_n)  \\
&&\epito (\conv_{1\le i\le n-1}(M_i \hconv N_i)) \conv  (M_n\hconv N_n),
\eneqn
as desired.
\QED

\section{Leclerc's conjecture}
{\em In this section, $R$ is assumed to be a symmetric KLR algebra
over a base field $\cor$.}

\subsection{Leclerc's conjecture}

  The following theorem is a part of Leclerc's conjecture stated in the introduction.
\Th\label{th:leclerc}
Let $M$ and $N$ be simple modules.
We assume that $M$ is real.
Then we have
the equalities in the Grothendieck group $K(R\gmod)${\rm:}
\bnum
\item
$[M\conv N]=[M\hconv N]+\sum_{k}[S_k]$\\[.5ex]
with simple modules $S_k$ such that $\La(M,S_k)<\La(M,M\hconv N)=\La(M,N)$,
\item
$[M\conv N]=[ M\sconv N]+\sum_{k}[S_k]$\\[.5ex]
with simple modules $S_k$ such that $\La(S_k,M)<\La( M\sconv N,M)=\La(N,M)$,
\item
$[N\conv M]=[N\hconv M]+\sum_{k}[S_k]$\\[.5ex]
with simple modules $S_k$ such that $\La(S_k,M)<\La(N\hconv M,M)=\La(N,M)$,
\item
$[N\conv M]=[ N\sconv M]+\sum_{k}[S_k]$\\[.5ex]
with simple modules $S_k$ such that $\La(M,S_k)<\La(M, N \sconv M )=\La(M,N)$.
\ee
In particular, $M\hconv N$ as well as $ M \sconv N $ appears only once
in the Jordan-H\"older series of $M\conv N$ in $R\smod$.
\enth

The following result is an immediate consequence of this theorem.
\Cor\label{cor:compest}
Let $M$ and $N$ be simple modules.
We assume that one of them is real.
Assume that $M$ and $N$ do not commute,
Then we have
the equality in the Grothendieck group $K(R\gmod)$
$$[M\conv N]=[M\hconv N]+[ M\sconv N ]+\sum_{k}[S_k]$$
with simple modules $S_k$.
Moreover we have
\bnum
\item
If $M$ is real, then we have
$\La(M,  M\sconv N)<\La(M,N)$, $\La(M\hconv N,M)<\La(N,M)$ and
$\La(M, S_k)<\La(M,N)$, $\La(S_k,M)<\La(N,M)$.
\item
If $N$ is real, then we have
$\La(N, M\hconv N)<\La(N,M)$, $\La( M \sconv N, N)<\La( M,N)$ and
$\La(N, S_k)<\La(N,M)$, $\La(S_k,N)<\La(M,N)$.
\ee
\encor

\medskip

\Proof[Proof of Theorem~\ref{th:leclerc}]
We shall prove only (i).
The other statements are proved similarly.

$$M\conv N=K_0\supset K_1\supset \cdots\supset K_\ell\supset K_{\ell+1}=0.$$
Then we have
$K_0/K_1\simeq M\hconv N$.
Let us consider the
renormalized R-matrix
$\Rm_{M_z,M\circ N}=(M\conv\Rm_{M_z,N})\circ(\Rm_{M_z,M}\conv N)$
$$\xymatrix@C=10ex{
M_z\conv M\conv N\ar[r]^{\Rm_{M_z,M}\circ N}
&M\conv M_z\conv N\ar[r]^{M\circ\Rm_{M_z,N}}&M\conv N\conv M_z.}
$$
Then $\Rm_{M_z,M\circ N}$ sends $M_z\conv K_k$ to $K_k\conv M_z$ for any $k$.
Hence evaluating the above diagram at $z=0$, we obtain
$$\xymatrix@C=10ex{
M\conv M\conv N\ar[r]^{M\circ\rmat{M,N}}
&M\conv N\conv M\\
\wb{M\conv K_1}\ar[r]\ar@{^{(}->}[u] &\wb{K_1\conv M\,.}\ar@{^{(}->}[u]}
$$
Since $\Im(\rmat{M,N}\cl M\conv N\to N\conv M)\simeq (M\conv N)/K_1$,
we have
$\rmat{M,N}(K_1)=0$.
Hence, $\Rm_{M_z,M\circ N}$ sends $M_z\conv K_1$ to
$(K_1\conv M_z)\cap z\bl (M\conv N)\conv M_z\br
=z(K_1\conv M_z)$.
Thus $z^{-1}\Rm_{M_z,M\circ N}\vert_{M_z\circ K_1}$ is well defined.
Then it sends
$M_z\conv K_k$ to $K_k\conv M_z$ for $k\ge1$.
Thus we obtain an R-matrix
$$z^{-1}\Rm_{M_z, M\circ N}\vert_{M_z\circ(K_k/K_{k+1})}\cl
M_z\conv (K_k/K_{k+1})\to (K_k/K_{k+1})\conv M_z
\quad\text{for $1\le k\le \ell$.}$$
Hence we have
$$\Rm_{M_z, K_k/K_{k+1}}=z^{-s_k}z^{-1}\Rm_{M_z, M\circ N}\vert_{M_z\circ(K_k/K_{k+1})}$$
for some $s_k\in \Z_{\ge0}$.
Since the homogeneous degree of $\Rm_{M_z,M\circ N}$
is $\La(M,M\circ N)=\La(M,N)$, we obtain
$$\La(M,K_k/K_{k+1})=\La(M,N)-2(1+s_k)<\La(M,N).$$
\QED

Recall that the isomorphism classes of self-dual simple modules in $R \gmod$ are  parameterized by the crystal basis $B(\infty)$
(\cite{LV11}).
The following theorem is an application of the above theorem.

\Th \label{thm:divisible}
Let $\phi$ be an element of the Grothendieck group $K(R\gmod)$
given by
\eqn
\phi = \sum_{b \in B(\infty)} a_b [L_b],
\eneqn
where $L_b$ is the self-dual simple module corresponding to $b \in B(\infty)$ and $a_b \in \Z[q^{\pm1}]$.
Let $A$ be a real simple module in $R\gmod$.
Assume that we have an equality
\eqn
\phi [A] = q^l [A] \phi
\eneqn
in $K(R\gmod)$ for some $l \in \Z$.
Then $A$ commutes with $L_b$ and
\eqn
l= \La(A,L_b)
\eneqn
for  every $b \in B(\infty)$ such that $a_b \neq 0$.

\enth
\Proof
Note that we have
\eqn
\phi [A] = \sum_b a_b [L_b \conv A]
= \sum_b a_b ([L_b \hconv A] + \sum_k [S_{b,k}]) \quad \text{and} \\
q^l [A] \phi = q^l \sum_b a_b [A \conv L_b ]
= q^l \sum_b a_b (q^{\La(L_b,A)}[L_b \hconv A] + \sum_k [S^{b,k}]),
\eneqn
for some simple modules $S_{b,k}$ and $S^{b,k}$ satisfying
\eqn
\La(S_{b,k},A) < \La(L_b,A) \quad \text{and} \quad\La(S^{b,k},A) < \La(L_b,A)
\eneqn
by Theorem \ref{th:leclerc}.

We may assume that $\set{ b\in B(\infty)}{a_b \neq 0}\not=\emptyset$. Set
\eqn
t\seteq\displaystyle\max \set{ \La(L_b,A)}{ a_b \neq 0}.
\eneqn
By taking the classes of self-dual simple modules $S$ with $\La(S,A)=t$
in the expansions of $\phi [A]$ and $q^l [A] \phi$, we obtain
\eqn
\sum_{\La(L_b,A)=t} a_b [L_b \hconv A ]
=\sum_{\La(L_b,A)=t} q^l a_b q^{\La(L_b,A)} [L_b \hconv A ].
\eneqn
In particular, we have
$t=-l$.

Set
\eqn
t'\seteq\displaystyle\max \set{ \La(A,L_b)}{a_b \neq 0}.
\eneqn
Then, by a similar argument we have
$t'=l$.

It follows that

\eqn
0= t+t' \ge  \La(L_b,A) + \La(A,L_b)\ge0
\eneqn
for every $b$ such that $a_b \neq 0$. Hence $A$ and $L_b$ commute.

Since
\eqn \sum a_b ~ q^{\La(A,L_b)} [A\conv L_b]
 =\sum a_b [L_b \conv A]
=\phi [A]=q^l [A] \phi
 =q^l \sum a_b [A\conv L_b],
\eneqn
we have
\eqn
l=\La(A,L_b)
\eneqn
for any $b$ such that $a_b \neq 0$, as desired.
\QED

\Cor \label{cor:qcomm_module}
Let $M$ and $N$ be simple modules. Assume that one of them is real.
If $[M]$ and $[N]$ q-commute \ro i.e., $[M][N]=q^n[N][M]$ for some $n\in\Z$\rf, then
$M$ and $N$ commute. In particular, $M\conv N$ is simple.
\encor
The following corollary is an immediate consequence of the corollary above
and Theorem~\ref{thm:categorification 2}.
\Cor \label{cor:qcomm}
Assume that the generalized Cartan matrix $A$ is symmetric
and  that $b_1, b_2\in B(\infty)$ satisfy the conditions{\rm:}
\bni
\item one of  $G^\up(b_1)^2$ and $G^\up(b_2)^2$ is a member of the upper global basis
up to a power of $q$,
\item
$G^\up(b_1)$ and $G^\up(b_2)$ q-commute.
\ee
Then their product $G^\up(b_1)G^\up(b_2)$ is a member of the upper global basis
of $\Um$ up to a power of $q$.
\encor

\medskip
\subsection{Geometric results}\label{subsec:geometry}
The result of this subsection (Theorem~\ref{th:head})
 was explained to us
by Peter McNamara.  It will be used in the proof of the crucial result Theorem \ref{thm: canonical surjection}.
{\em In this subsection, we assume further that
the base field $\cor$ is  of characteristic $0$}.

\Th[{\cite[Lemma 7.5]{Mc14}}]\label{th:head}
Assume that the base field\/ $\cor$
is  of characteristic $0$.
Assume that $M\in R\gmod$ has
a head $q^cH$ with a self-dual simple module $H$ and $c\in\Z$. Then we have
the equality in the Grothendieck group $K(R\gmod)$
$$[M]=q^c[H]+\sum_{k}q^{c_k}[S_k]$$
with self-dual simple modules $S_k$ and $c_k>c$.
\enth

By duality, we obtain the following corollary.
\Cor Assume that the base field $\cor$ is a field  of characteristic $0$.
Assume that $M\in R\gmod$ has a socle $q^cS$
with a self-dual simple module $S$ and $c\in\Z$. Then we have
the equality in $K(R\gmod)$
$$[M]=q^c[S]+\sum_{k}q^{c_k}[S_k]$$
with self-dual simple modules $S_k$ and $c_k<c$.
\encor

Applying this theorem to convolution products,
we obtain the following corollary.

\Cor Assume that the base field $\cor$ is  of characteristic $0$.
Let $M$ and $N$ be simple modules.
We assume that one of them is real.
Then we have the equalities in  $K(R\gmod)$:
\bnum
\item
$[M\conv N]=[M\hconv N]+\sum_{k}q^{c_k}[S_k]$\\
with self-dual simple modules $S_k$ and
$$c_k>-\tL(M,N)=\bl -\La(M,N)-(\wt(M),\wt(N)\br/2.$$
\item
$[M\conv N]=[ M \sconv N ]+\sum_{k}q^{c_k}[S_k]$\\
with self-dual simple modules $S_k$ and $c_k<\bl \La(N,M)-(\wt(N),\wt(M))\br/2$.
\ee
\encor
Note that $q^{\tL(M,N)}M\hconv N$ is self-dual by Lemma~\ref{lem:selfdual}.

Theorem~\ref{th:leclerc} and Theorem~\ref{th:head} solve affirmatively Conjecture~1 of Leclerc (\cite{L03}) in the symmetric generalized Cartan matrix case, as stated in the introduction.
 More precisely, let  $R$ be  a symmetric \KLR\ over a base field $\cor$ of characteristic $0$ and let $M$ and $N$ be simple modules over $R$.  Assume further that  $M$ is real.
Then by Theorem ~\ref{th:leclerc} $M \hconv N$ and $M \sconv N$ appear exactly once in a Jordan-H\"older series of $M \conv N$.
Write $M \hconv N = q^m H$ and $M \sconv N=q^s S$ for some self-dual simple modules $H$,  $S$ and $m,s \in \Z$. By Theorem \ref{th:head}, we have
\eqn
[M \conv N] = q^m [H] + q^s [S] + \sum_k q^{c_k} [S_k],
\eneqn
where $S_k$ are self-dual simple modules, and $ m < c_k < s$ for all $k$.
Collecting the terms,
we obtain
\eqn
[M\conv N]=q^m [H] + q^s [S] + \sum_{L \not \simeq H, S} \gamma_{M,N}^L (q) [L],
\eneqn
 with
$$  \gamma^L_{M,N}(q) \in q^{m+1}\Z[q] \cap q^{s-1}\Z[q^{-1}],$$
which proves Leclerc's first conjecture via Theorem \ref{thm:categorification 2}.

We obtain the following result
which is a generalization of Lemma~\ref{lem:crde} in
the characteristic-zero case.

\Cor \label{cor:minmaxD}
Assume that the base field $\cor$ is  of characteristic $0$.
Let $M$ and $N$ be simple modules.
We assume that one of them is real. Write
$$[M\conv N]=\sum_{k=1}^nq^{c_k}[S_k]$$
with self-dual simple modules $S_k$ and $c_k\in\Z$.
Then we have
$$\max\set{c_k}{1\le k\le n}-\min\set{c_k}{1\le k\le n}=\de(M,N).$$
\encor

\subsection{Proof of Theorem~\ref{th:head}}
Recall that the graded algebra $R(\beta)$ ($\beta\in \rtl^+$)
is geometrically realized as follows (\cite{VV09}).
There exist a reductive group $G$ and a $G$-equivariant projective morphism
$f\cl X\to Y$ from a smooth algebraic $G$-variety $X$
to an affine $G$-variety $Y$ defined over the complex number field $\C$
such that
$$R(\beta)\simeq \tEnd_{\Db(\cor_Y)}(\R f_* (\cor_X [\dim X]))
\quad\text{as a graded $\cor$-algebra.}$$
Here, $\Db(\cor_Y)$ denotes the  $G$- equivariant derived category
of the $G$-variety $Y$ with coefficient $\cor$,
and $\tEnd_{\Db(\cor_Y)}(K)=\tHom_{\Db(\cor_Y)}(K, K)$  with
$$\tHom_{\Db(\cor_Y)}(K, K')\seteq\soplus_{n\in\Z} \Hom_{\Db(\cor_Y)}(K, K'[n]).$$
We denote by $\cor_X[\dim X]$ the direct sum of the constant sheaves  on each connected components of $X$, all of which are shifted by their  dimensions.
By the decomposition theorem (\cite{BBDG}), we have a decomposition
$$\R f_* (\cor_X [\dim X])\simeq\soplus_{a\in J}E_a\tens \F_a,$$
where $\{\F_a\}_{a\in J}$ is a finite family of simple perverse sheaves
on $Y$ and $E_a$ is a  non-zero  finite-dimensional graded $\cor$-vector space such that
\eq
H^k(E_a)\simeq H^{-k}(E_a)\quad\text{for any $k\in\Z$.}
\label{eq:Eselfdual}
\eneq
The last fact \eqref{eq:Eselfdual} follows from the hard Lefschetz theorem
(\cite{BBDG}).

Set $A_{a,b}=\tHom_{\Db(\cor_Y)}(\F_b, \F_a)$.
Then we have the multiplication morphisms
$$A_{a,b}\tens A_{b,c}\to A_{a,c}$$ so that
$$A\seteq\soplus_{a,b\in J}A_{a,b}$$
has a structure of $\Z$-graded algebra such that
$$A_{\le0}\seteq\soplus_{n\le0}A_n=A_0\simeq\cor^J.$$
Hence the family of the isomorphism classes of simple objects
(up to a grading shift) in $A\gmod$
is $\{\cor_a\}_{a\in J}$.
Here, $\cor_a$ is the module obtained by the algebra homomorphism
$A\to A_{\le0}\simeq \cor^J\to\cor$,
where the last arrow is the $a$-th projection.
Hence we have
$$K(A\gmod)\simeq\soplus_{a\in J}\Z[q^{\pm1}][\cor_a].$$

On the other hand, we have
$$R(\beta)\simeq \soplus_{a,b\in J}E_a\tens A_{a,b}\tens E_b^*.$$
Set
$$L\seteq\soplus_{a,b\in J}E_a\tens A_{a,b}.$$
Then, $L$ is endowed with a natural structure of
$(\soplus_{a,b\in J}E_a\tens A_{a,b}\tens E_b^*, A)$-bimodule.
It is well-known that
the functor $M\mapsto L\otimes_AM$ gives a graded Morita-equivalence
$$\Phi\cl A\gmod\isoto R(\beta)\gmod.$$
Note that
$\Phi(\cor_a)\simeq E_a$ and
$\{E_a\}_{a\in J}$ is the set of isomorphism classes of self-dual
simple graded $R(\beta)$-modules by
\eqref{eq:Eselfdual}.

By the above observation, in order to prove the theorem,
it is enough to show the corresponding statement for
the graded ring $A$, which is obvious.

\section{Quantum cluster algebras}
In this section we recall the definition of skew-symmetric quantum cluster algebras following \cite{BZ05}, \cite[\S 8]{GLS}.

\subsection{Quantum seeds} Fix a finite index set
$\K=\Kex\sqcup\Kfr$ with the decomposition into
the set $\Kex$ of exchangeable indices   and the set $\Kfr$ of frozen indices.
Let $L=(\lambda_{ij})_{i,j\in \K}$ be a skew-symmetric integer-valued $\K\times \K$-matrix.
\Def We define $\mathscr P(L)$ as the  $\Z[q^{\pm 1/2}]$-algebra generated by
a family of elements $\{X_i\}_{i\in \K}$
with the defining relations
\eqn
&&X_iX_j=q^{\lambda_{ij}}X_j X_i \quad (i,j\in \K).\label{eq:Xcom}
\eneqn
We denote by $\mathscr F(L)$ the skew field of fractions of $\mathscr P(L)$.

\edf
For ${\bf a}=(a_i)_{i\in\K}\in \Z^\K$, we define the element $X^{\bf a}$
of $\mathscr F(L)$
as
\eqn
&&X^{\bf a}\seteq q^{1/2 \sum_{i > j} a_ia_j\lambda_{ij}}
\overset{\To[{\ }]}{\prod}_{i\in\K}
 X_i^{a_i}.
\eneqn
Here we take a total order $<$ on the set $\K$ and
$\overset{\To[{\ }]}{\prod}_{i\in\K} X_i^{a_i}=X_{i_1}^{a_{i_1}}\cdots X_{i_r}^{a_{i_r}}$
where $\K=\{i_1,\ldots,i_r\}$ with $i_1<\cdots <i_r$.
Note that
$X^{\bf a}$ does not depend on the choice of a total order of $\K$.

We have
\eq
&&X^{\bf a}X^{\bf b}=q^{1/2\sum_{i,j\in\K} a_ib_j\la_{ij}}X^{\bf a+\bf b}.
\label{eq:XaXb}
\eneq
If $\mathbf{a}\in\Z_{\ge0}^\K$, then $X^{\mathbf a}$ belongs to $\mathscr P(L)$.

It is well known that $\{X^{\mathbf a}\}_{\mathbf{a}\in\Z_{\ge0}^\K}$ is a basis
of $\mathscr P(L)$ as a $\Z[q^{\pm 1/2}]$-module.

Let $A$ be a $\Z[q^{\pm1/2}]$-algebra. We say that a family
$\{x_i\}_{i\in\K}$ of elements of $A$ is {\em $L$-commuting} if
it satisfies $x_ix_j=q^{\la_{ij}}x_jx_i$ for any $i,j\in\K$.
In such a case we can define $x^\mathbf{a}$ for any $\mathbf{a}\in\Z_{\ge0}^\K$.
We say that an $L$-commuting family $\{x_i\}_{i\in\K}$ is {\em algebraically independent} if the algebra map
$\mathscr P(L)\to A$ given by $X_i\mapsto x_i$ is injective.

Let $\widetilde B = (b_{ij})_{(i,j)\in\K\times\Kex}$ be an  integer-valued
$\K \times \Kex$-matrix.
 We assume that the {\em principal part $B\seteq(b_{ij})_{i,j\in\Kex}$}  of $\widetilde B$ is skew-symmetric.

To the matrix $\wB$ we can associate the quiver $Q_{\wB}$ without loops,  $2$-cycles  and arrows between frozen vertices
such that its vertices are labeled by $J$ and
the arrows are given by
\begin{equation} \label{eq: bij}
 b_{ij} =  (\text{the number of  arrows from $i$ to $j$}) - (\text{the number of  arrows from $j$ to $i$}).
 \end{equation}
Here we extend the $\K\times\Kex$-matrix $\wB$ to
the skew-symmetric $\K\times\K$-matrix $\wB'=(b_{ij})_{i,j\in\K}$
by setting $b_{ij}=0$ for $i,j\in\Kfr$.

Conversely, whenever we have a quiver with vertices labeled by $J$ and without loops, $2$-cycles and arrows between frozen vertices, we can
associate a $\K \times \K_\ex$-matrix $\wB$ by \eqref{eq: bij}.

We say that the pair $(L, \widetilde B)$ is {\em compatible} if there exists a positive integer $d$ such that
\eq \label{eq:compatible}
&&\sum_{k\in\K} \lambda_{ik} b_{kj} =\delta_{ij} d \quad (i\in \K,\;j\in\Kex).
\eneq

Let $(L,  \widetilde B)$ be a compatible pair and $A$ a $\Z[q^{\pm1/2}]$-algebra.
We say that $\seed = (\{x_i\}_{i\in\K},L, \widetilde B)$ is
a {\em quantum seed} in $A$ if $\{x_i\}_{i\in\K}$ is
an algebraically independent $L$-commuting family of elements of
 $A$.

The set $\{x_i\}_{i\in\K}$ is called the {\em cluster} of $\seed$ and
its elements  the {\em cluster variables}.
The cluster variables $x_i$ ($i\in\Kfr$) are called  the {\em frozen variables}.
The elements $x^{\bf a}$ (${\bf a}\in \Z_{\ge0}^\K$) are called  the {\em quantum cluster monomials}.

\subsection{Mutation}
For $k\in\Kex$, we define a $\K\times \K$-matrix  $E=(e_{ij})_{i,j\in\K}$ and a
$\Kex\times \Kex$-matrix $F=(f_{ij})_{i,j\in\Kex}$ as follows:
\eqn
e_{ij}=
\begin{cases}
  \delta_{ij} & \text{if} \ j \neq k, \\
  -1 & \text{if} \ i= j = k, \\
  \max(0, -b_{ik}) & \text{if} \ i \neq  j = k,
\end{cases}
\hs{10ex}
f_{ij}=
\begin{cases}
  \delta_{ij} & \text{if} \ i \neq k, \\
  -1 & \text{if} \ i= j = k, \\
  \max(0, b_{kj}) & \text{if} \ i = k \neq j.
\end{cases}
\eneqn
The {\em mutation $\mu_k(L,\widetilde B)\seteq(\mu_k(L),\mu_k(\widetilde B))$ of a compatible pair $(L,\widetilde B)$ in direction $k$} is given by
\eqn
\mu_k(L)\seteq(E^T) \, L \, E, \quad \mu_k(\widetilde B)\seteq E \, \widetilde B \, F.
\eneqn
Then the pair $(\mu_k(L),\mu_k(\widetilde B))$
is also compatible with the same integer $d$ as in the case of $(L,\widetilde B)$ (\cite{BZ05}).

Note that for each $k\in\Kex$, we have
\eq \label{eq:mutation B}
\mu_k(\widetilde B)_{ij} =
\begin{cases}
  -b_{ij} & \text{if}  \ i=k \ \text{or} \ j=k, \\
  b_{ij} + (-1)^{\delta(b_{ik} < 0)} \max(b_{ik} b_{kj}, 0) & \text{otherwise,}
\end{cases}
\eneq
and
\eqn
\mu_k(L)_{ij} =
\begin{cases}
  0 & \text{if}  \ i=j \\
  -\la_{kj}+\displaystyle\sum _{t\in\K} \max(0, -b_{tk}) \la_{tj} & \text{if} \ i=k, \ j\neq k, \\
  -\la_{ik}+\displaystyle\sum _{t\in\K} \max(0, -b_{tk}) \la_{it} & \text{if} \ i \neq k, \ j= k, \\
  \la_{ij} & \text{otherwise.}
\end{cases}
\eneqn
Note  also that we have $$\displaystyle\sum _{t\in\K} \max(0, -b_{tk}) \la_{it}
 =\displaystyle\sum _{t\in\K} \max(0, b_{tk}) \la_{it}$$
 for $i\in\K$ with $i\neq k$, since $(L,\widetilde B)$ is compatible.

We define
\eq
&&a_i'=
\begin{cases}
  -1 & \text{if} \ i=k, \\
 \max(0,b_{ik}) & \text{if} \ i\neq k,
\end{cases} \qquad
a_i''=
\begin{cases}
  -1 & \text{if} \ i=k, \\
 \max(0,-b_{ik}) & \text{if} \ i\neq k.
\end{cases}
\label{eq:aa}
\eneq
and set ${\bf a}'\seteq(a_i')_{i\in\K}$ and ${\bf a}''\seteq(a_i'')_{i\in\K}$.

Let $A$ be a $\Z[q^{\pm1/2}]$-algebra contained in a skew field $K$.
Let $\seed=(\{x_i\}_{i\in\K},L, \widetilde B)$ be a quantum seed in $A$.
Define the elements $\mu_k(x)_i$ of $K$ by
\eq \label{eq:quantum mutation}
\mu_k(x)_i\seteq
\begin{cases}
  x^{{\bf a}'}  +   x^{{\bf a}''}, & \text{if} \ i=k, \\
 x_i & \text{if} \ i\neq k.
\end{cases}
\eneq
Then $\{\mu_k(x)_i\}$ is an algebraically independent
$\mu_k(L)$-commuting family in $K$.
We call
\eqn
\mu_k(\seed)\seteq\bl\{\mu_k(x)_i\}_{i\in\K},\mu_k(L),\mu_k(\widetilde B)\br
\eneqn
the {\em mutation
of $\seed$ in direction $k$}.
It becomes a new quantum seed in $K$.

 \begin{definition}
Let $\seed=(\{x_i\}_{i\in\K},L, \widetilde B)$ be a quantum seed in $A$.
   The {\em quantum cluster algebra $\mathscr A_{q^{1/2}}(\seed)$} associated to the quantum seed $\seed$ is  the $\Z[q^{\pm 1/2}]$-subalgebra of the skew field $K$ generated by all the quantum cluster variables in the quantum seeds obtained from $\seed$ by any sequence of mutations.
 \end{definition}
We call $\seed$ the {\em initial quantum seed}
of the quantum cluster algebra $\mathscr A_{q^{1/2}}(\seed)$.

\section{Monoidal categorification of cluster algebras}

Throughout  this section, fix $\K=\Kex\sqcup\Kfr$ and a base field
$\coro$.

Let $\shc$ be a $\coro$-linear abelian monoidal category.
For the definition of monoidal category, see,  for example, \cite[Appendix A.1]{K^3}.
Note that  in \cite{K^3}, it was called   the \emph{tensor category}.
A  $\coro$-linear abelian monoidal category is a $\coro$-linear
monoidal category such that
it is abelian and
the tensor functor $\;\tens\;$ is $\coro$-bilinear and exact.

We assume further the following conditions on $\shc$
\eq
&&\hs{-4ex}\left\{\parbox{60ex}{
\bnum
\item
Any object of $\shc$ has a finite length,
\vs{.5ex}
\item $\coro\isoto\Hom_{\shc}(S,S)$ for any simple object $S$ of $\shc$.
\ee}\right.\label{cond:monoidal category}
\eneq

A simple object $M$ in $\shc$ is called \emph{real} if $M \otimes M$ is simple.

\subsection{Ungraded cases}
\begin{definition}
%  \label{def:monoidal seed}
Let  $\seed = (\{ M_i\}_{i\in \K },\widetilde B)$ be a pair
of
a family $\{ M_i\}_{i\in\K}$ of simple objects in $\shc$ and
an  integer-valued $\K\times\Kex$-matrix
$\widetilde B = (b_{ij})_{(i,j)\in\K\times\Kex}$
whose principal part is skew-symmetric.
We call $\seed$ a \emph{monoidal seed in $\shc$} if
\bnum
  \item $M_i \otimes M_j \simeq M_j \otimes M_i$ for any $i,j\in\K$,
  \item $\sotimes_{i\in\K} M_i^{\otimes a_i}$ is simple for any $(a_i)_{i\in\K}\in \Z_{\ge 0}^{\K}$.
  \end{enumerate}
\end{definition}

\begin{definition}
  % \label{def:mutation monoidal seed}
  For $k\in\Kex$,
  we say that a  monoidal seed  $\seed = (\{ M_i\}_{i\in \K },\widetilde B)$
 \emph{admits a mutation in direction $k$} if
there exists a simple object  $M_k' \in \shc$
such that
\bni
  \item
there exist exact sequences in $\shc$
\eqn
&&0 \to  \sotimes_{b_{ik} >0} M_i^{\tensor b_{ik}} \to M_k \tensor M_k' \to
 \sotimes_{b_{ik} <0} M_i^{\tensor (-b_{ik})} \to 0, \label{eq:ses_mutation1}\\
 &&0 \to  \sotimes_{b_{ik} <0} M_i^{\tensor(-b_{ik})} \to M_k' \tensor M_k \to
  \sotimes_{b_{ik} >0} M_i^{\tensor b_{ik}} \to 0.\label{eq:ses_mutation2}
\eneqn
\item the pair $\mu_k(\seed)\seteq
(\{M_i\}_{i\neq k}\cup\{M_k'\},\mu_k(\widetilde B))$ is
a monoidal seed in $\shc$.
\end{enumerate}
\end{definition}

Recall that a cluster algebra $A$ with an initial seed
$(\{x_i\}_{i\in\K},\widetilde B)$ is the $\Z$-subalgebra of
$\Q(x_i\vert i\in\K)$ generated by all the cluster variables in the seeds obtained from  $(\{x_i\}_{i\in\K},\widetilde B)$ by any sequence of mutations.
Here, the mutation $x_k'$ of a cluster variable $x_k$ $(k\in\Kex)$ is given  by
\eqn
x_k' = \dfrac{\prod_{b_{ik} \ge 0} x_i^{b_{ik}}+\prod_{b_{ik} \le 0} x_i^{-b_{ik}}}{x_k},
\eneqn
and the mutation of $\widetilde B$ is given in \eqref{eq:mutation B}.

\begin{definition}
A $\coro$-linear abelian monoidal category $\shc$
 satisfying \eqref{cond:monoidal category} is called a \em{monoidal categorification of a cluster algebra $A$}
if
\bni
\item the Grothendieck ring $K(\shc)$ is isomorphic to $A$,
\item there exists a monoidal seed $\seed = ( \{M_i\}_{i\in\K},\widetilde B)$ in $\shc$ such that
$[\seed]\seteq( \{[M_i]\}_{i\in\K},\widetilde B)$ is the initial seed of $A$ and
$\seed$ admits  successive mutations in all directions.
\end{enumerate}
\end{definition}
Note that if $\shc$ is a monoidal categorification of $A$,
then every seed in $A$ is of the form
$( \{[M_i]\}_{i\in\K},\widetilde B)$ for some monoidal seed
$( \{M_i\}_{i\in\K},\widetilde B)$  in $\shc$.
 In particular, all the cluster monomials in $A$ are the classes of real
simple objects in $\shc$.

\subsection{Graded cases}
Let $\rootl$ be a free abelian group equipped with a symmetric bilinear form
\eqn
( \ , \ ) : \rootl \times \rootl \to \Z \quad \text{such that} \
 (\beta,\beta) \in 2\Z \ \text{for all} \ \beta \in \rootl.
\eneqn
We consider  a $\coro$-linear abelian monoidal category $\shc$ satisfying
\eqref{cond:monoidal category} and the following conditions:
\eq
&&\hs{-2ex}\left\{\parbox{67ex}{
\bnum
\item We have a direct sum decomposition
  $\shc = \soplus_{\beta \in \rootl} \shc_\beta $ such that
 the tensor product $\otimes$
   sends  $\shc_\beta \times \shc_\gamma$ to $\shc_{\beta + \gamma}$
   for every  $\beta, \gamma \in \rootl$.

 \item There exists an object $Q  \in \shc_0$ satisfying
\bna
\item there is an isomorphism
$$R_{Q}(X) : Q \tensor X \isoto X \tensor Q$$
functorial in $X \in \shc$ such that
$$\xymatrix@C=7ex{
Q\tensor X\tensor Y\ar[r]_-{R_Q(X)}\ar@/^3ex/[rr]^{R_Q(X\tensor Y)}
&X\tensor Q\tensor Y\ar[r]_-{R_Q(Y)}
&X\tensor Y\tensor Q}$$
 commutes for any $X,Y\in \shc$,
\item the functor $X \mapsto Q \tensor X$ is an  equivalence of categories.
\end{enumerate}
\item for any $M$, $N\in\shc$, we have
$\Hom_\shc(M,Q^{\tens n}\tens N)=0$ except finitely many integers $n$.
\end{enumerate}
}\right.\label{cond:quantum monoidal category}
\eneq

We denote by $q$ the auto-equivalence $Q \otimes \scbul$ of $\shc$,
and call it the {\em grading shift functor}.

In such a case the Grothendieck group $K(\shc)$ is a $\rtl$-graded
$\Z[q^{\pm1}]$-algebra: $K(\shc)=\soplus_{\beta\in\rtl}K(\shc)_\beta$ where
$K(\shc)_\beta=K(\shc_\beta)$. Moreover, we have
$$K(\shc)=\soplus_{S}\Z[q^{\pm1}][S],$$
where $S$ ranges over  equivalence classes of
simple modules. Here, two simple modules $S$ and $S'$ are equivalent
if $q^nS\simeq S'$ for some $n\in\Z$.

\medskip
For $M\in\shc_{\beta}$, we sometimes write
$\beta=\wt(M)$ and call it the {\em weight} of $M$.
Similarly, for $x\in \Q(q^{1/2})\tens_{\Z[q^{\pm1}]}K(\shc_\beta)$,
we write $\beta=\wt(x)$ and call it the {\em weight} of $x$.

\begin{definition} \label{def:quantum monoidal seed}
We call a quadruple $\seed = (\{M_i\}_{i\in\K}, L,\widetilde B, D)$
a \emph{quantum monoidal seed  in $\shc$}
 if it satisfies the following conditions{\rm:}
\bnum
\item $\widetilde B = (b_{ij})_{i\in\K,\,j\in\Kex}$ is an  integer-valued $\K\times\Kex$-matrix
whose principal part is skew-symmetric,
\item $L=(\lambda_{ij})_{i,j\in\K}$ is an integer-valued skew-symmetric $\K\times\K$-matrix,
\item $D=\{d_i\}_{i\in\K}$ is a family of elements in $\rootl$,
\item $\{M_i\}_{i\in\K}$ is a family of simple objects such that $M_i \in \shc_{d_i}$ for any $i\in\K$,
\item $M_i \otimes M_j \simeq q^{\lambda_{ij}} M_j \otimes M_i$ for all $ i, j \in\K$,
\item $M_{i_1} \otimes \cdots \otimes M_{i_t}$ is simple for any
sequence  $(i_1,\ldots,i_t)$ in $\K$,
\item The pair  $(L,\widetilde B)$  is  compatible in the sense of \eqref{eq:compatible} with $d=2$,
\item $\lambda_{ij} - (d_i,d_j) \in 2\Z$ for all $i,j \in\K$,
\item $\displaystyle\sum_{i\in\K}b_{ik}d_i =0$ for all $k\in\Kex$.
\end{enumerate}
\end{definition}

Let  $\seed =(\{M_i\}_{i\in\K}, L,\widetilde B, D)$
be a  quantum monoidal seed.
For any $X\in\shc_\beta$ and $Y\in\shc_\gamma$ such that $X\otimes Y\simeq
q^{c}Y \otimes X$ and $c+(\beta,\gamma)\in2\Z$, we set
\eqn&&\tL(X,Y)=\frac{1}{2}\bl-c+(\beta,\gamma)\br\in\Z\eneqn
and
\eqn X\nconv Y\seteq q^{\tL(X,Y)}X\tens Y\simeq
q^{\tL(Y,X)}Y\tens X.\label{eq:balpro}\eneqn
Then $X\nconv Y\simeq Y\nconv X$.
For any sequence $(i_1,\ldots, i_\ell)$ in $\K$,
we define
$$\sodot_{s=1}^{\ell} M_{i_s}\seteq
(\cdots((M_{i_1}\nconv M_{i_2})\nconv M_{i_3})\cdots)\nconv M_{i_\ell}.$$
Then we have
\eqn
\sodot_{s=1}^{\ell} M_{i_s} =
q^{\frac{1}{2}\sum_{1\le u<v \le \ell}(-\lambda_{i_u i_v}  +(d_{i_u},d_{i_v}) ) } M_{i_1} \tensor \cdots \tensor M_{i_\ell}.
\eneqn
For any $w\in\sym_\ell$, we have
\eqn
&&\sodot_{s=1}^{\ell} M_{ i_{w(s)}}\simeq
\sodot_{s=1}^{\ell} M_{i_s}\eneqn
Hence for any subset $A$ of $\K$ and any set of non-negative integers
$\{m_a\}_{a\in A}$, we can define
$\sodot_{a\in A}M_a^{\snconv m_a}$.

For $(a_i)_{i\in\K}\in\Z_{\ge0}^\K$ and $(b_i)_{i\in\K}\in\Z_{\ge0}^\K$,
we have
$$\bl\sodot_{i\in \K}M_i^{\snconv a_i}\br\nconv
\bl\sodot_{i\in\K}M_i^{\snconv b_i}\br\simeq\sodot_{i\in \K}M_i^{\snconv (a_i+b_i)}.$$

\medskip
Let $\seed=(\{M_i\}_{i\in\K}, L,\widetilde B, D)$ be  a quantum monoidal seed.
When the $L$-commuting family $\{[M_i]\}_{i\in\K}$
of elements of $\Z[q^{\pm1/2}]\tens_{\Z[q^{\pm1}]}K(\shc)$
is algebraically independent,
we shall define a quantum seed
$[\seed]$ in $\Z[q^{\pm1/2}]\tens_{\Z[q^{\pm1}]}K(\shc)$
by
\eqn&&[\seed]=(\{q^{-(d_i,d_i)/4}[M_i]\}_{i\in\K}, L,\widetilde B).
\eneqn

Set
\eqn&&X_i:=q^{-(d_i,d_i)/4}[M_i].\label{eq:XM}
\eneqn
Then for any $\mathbf a=(a_i)_{i\in\K}\in\Z_{\ge0}^\K$, we have
$$X^{\mathbf a}=q^{-(\mu,\mu)/4}[\sodot_{i\in \K}M_i^{\snconv a_i}],$$
where $\mu=\wt(\sodot_{i\in \K}M_i^{\snconv a_i})=\wt(X^{\mathbf a})=\sum_{i\in \K}a_id_i$.

\medskip
For a given $k\in\Kex$, we define the \emph{mutation $\mu_k(D) \in \rootl^\K$ of $D$ in direction $k$ with respect to $\widetilde B$} by
\eqn
\mu_k(D)_i =d_i \ (i \neq k), \quad \mu_k(D)_k=-d_k+\sum_{b_{ik} >0}   b_{ik} d_i.
\eneqn
Note that
\eqn
\mu_k(\mu_k(D))=D.
\eneqn
Note also  that
$(\mu_k(L),\mu_k(\tB),\mu_k(D))$ satisfies conditions (viii) and (ix)
in Definition~\ref{def:quantum monoidal seed} for any $k\in\Kex$.

\medskip

We have the following
\Lemma\label{lem:decat} Set $X'_k=\mu_k(X)_k$, the mutation of $X_k$
as in \eqref{eq:quantum mutation}.
Set $\zeta=\wt(X'_k)=-d_k+\sum_{b_{ik}>0}b_{ik}d_i$.
Then we have
\eqn
&&\hs{-20ex}\ba{l}q^{m_k}[M_k] q^{(\zeta,\zeta)/4}X'_k=
q [\sodot_{b_{ik} >0} M_i^{\snconv b_{ik}}]+
[\sodot_{b_{ik} <0} M_i^{\snconv (-b_{ik})}],\\[2.5ex]
q^{m'_k} q^{(\zeta,\zeta)/4}X'_k[M_k]= [\sodot_{b_{ik} >0} M_i^{\snconv b_{ik}}]+
q[\sodot_{b_{ik} <0} M_i^{\snconv (-b_{ik})}],\ea
\eneqn
where
\eq
&&\hs{-10ex}\left\{\ba{l}
m_k=\dfrac{1}{2}(d_k,\zeta) +\dfrac{1}{2} \displaystyle
\sum_{b_{ik} < 0}\la_{ki}  b_{ik},
\\[1ex]
m'_k=\dfrac{1}{2}(d_k,\zeta) +\dfrac{1}{2} \displaystyle
\sum_{b_{ik} > 0} \la_{ki}b_{ik}.
\ea\right.\label{eq:mm'}
\eneq
\enlemma
\Proof
By \eqref{eq:XaXb}, we have
\eqn
X_k X^{\bf a} = q^{\frac{1}{2} \sum_{i \in \K} a_i \la_{ki}}  X^{{{\bf e}_k} +\bf a}
\quad \text{for ${\bf a}=(a_i)_{i\in\K} \in \Z^\K$ and
$({\bf e}_k)_i = \delta_{ik}$  ($i \in \K$).}
\eneqn
Let $\mathbf{a}'$ and ${\bf a}''$ be as in \eqref{eq:aa}. Because
\eqn
\sum_{i \in \K} a_i' \la_{ki} - \sum_{i \in \K} a_i'' \la_{ki}
= \sum_{b_{ik} >0} b_{ik} \la_{ki} - \sum_{b_{ik} <0} (-b_{ik}) \la_{ki}
= \sum_{i \in \K} b_{ik} \la_{ki} = 2,
\eneqn
 we have
\eqn
X_k X'_k = X_k(X^{{\bf a}'}  +   X^{{\bf a}''})
=q^{\frac{1}{2} \sum_i a_i'' \la_{ki} } (qX^{{\bf e_k}+{\bf a}'} + X^{{\bf e_k}+{\bf a}''}).
\eneqn
Note that
$
\wt(X^{{\bf e_k}+{\bf a}'}) =\wt(X^{{\bf e_k}+{\bf a}''})= d_k +\zeta.
$
It follows that
\eqn
m_k&=&-\dfrac{1}{4}((d_k,d_k)+(\zeta,\zeta))-\dfrac{1}{2} \sum_{i \in J} a_i'' \la_{ki} + \dfrac{1}{4}(\zeta+d_k,\zeta+d_k) \\
&=&\dfrac{1}{2} (d_k, \zeta) +\dfrac{1}{2} \sum_{b_{ik} <0} b_{ik} \la_{ki}.
\eneqn

One can calculate $m_k'$ in a similar way.
\QED

\begin{definition} \label{def:monoidal mutation}
We say that a quantum monoidal seed
$\seed =(\{M_i\}_{i\in\K}, L,\widetilde B, D)$
 \emph{admits a mutation in direction $k\in\Kex$} if
there exists  a simple object  $M_k' \in \shc_{\mu_k(D)_k}$
such that
\bni
  \item
there exist exact sequences in $\shc$
\eq
&&0 \to q \sodot_{b_{ik} >0} M_i^{\snconv b_{ik}} \to q^{m_k} M_k \tensor M_k' \to
 \sodot_{b_{ik} <0} M_i^{\snconv (-b_{ik})} \to 0,
\label{eq:ses graded mutation1} \\
 &&0 \to q \sodot_{b_{ik} <0} M_i^{\snconv(-b_{ik})} \to q^{m_k'} M_k' \tensor M_k \to
  \sodot_{b_{ik} >0} M_i^{\snconv b_{ik}} \to 0, \label{eq:ses graded mutation2}
\eneq
where $m_k$ and $m'_k$ are as in \eqref{eq:mm'}.
\item %the quadruple
$\mu_k(\seed)\seteq\bl\{M_i\}_{i\neq k}\sqcup\{M_k'\},\mu_k(L),
\mu_k(\widetilde B), \mu_k(D)\br$ is
a quantum monoidal seed in $\shc$.
\end{enumerate}
We call $\mu_k(\seed)$ the {\em mutation} of $\seed$ in direction $k$.
\end{definition}

By Lemma~\ref{lem:decat}, the following lemma is obvious.
\Lemma
Let $\seed =(\{M_i\}_{i\in\K}, L,\widetilde B, D)$ be
 a quantum monoidal seed  which admits a mutation in direction $k\in\Kex$.
Then we have
$$[\mu_k(\seed)]=\mu_k([\seed]).$$
\enlemma

\begin{definition}
Assume that a $\coro$-linear abelian monoidal category
$\shc$ satisfies  \eqref{cond:monoidal category}
and \eqref{cond:quantum monoidal category}.
The category $\shc$ is called a
\em{monoidal categorification of a quantum cluster algebra $A$
over $\Z[q^{\pm1/2}]$}
if
\bnum
\item the Grothendieck ring $\Z[q^{\pm1/2}]\tens_{\Z[q^{\pm1}]} K(\shc)$ is isomorphic to $A$,
\item there exists a quantum monoidal seed
$\seed =(\{M_i\}_{i\in\K}, L,\widetilde B, D)$ in $\shc$ such that
$[\seed]\seteq(\{q^{-(d_i,d_i)/4}[M_i]\}_{i\in\K}, L, \widetilde B)$
 is a quantum seed of $A$,
\item $\seed$ admits successive mutations in all the directions.
\end{enumerate}
\end{definition}
Note that if $\shc$ is a monoidal categorification of a quantum cluster algebra
$A$, then any quantum seed in $A$ obtained by  a sequence of mutations from the initial quantum seed
is of the form
$(\{q^{-(d_i,d_i)/4}[M_i]\}_{i\in\K}, L, \widetilde B)$  for some quantum monoidal seed
$(\{M_i\}_{i\in\K}, L,\widetilde B, D)$.
 In particular, all the quantum cluster monomials in $A$ are the
classes of real simple objects in $\shc$ up to a power of $q^{1/2}$.

\section{Monoidal categorification via modules over KLR algebras}
\subsection{Admissible pair}
In this section, we assume that \emph{$R$ is a symmetric KLR algebra
over a base field $\coro$}.

From now on, we focus on the case when $\shc$ is a full subcategory of $R \gmod$
 stable under taking convolution products, subquotients, extensions
and grading shift.
In particular, we have
\eqn
\shc = \soplus_{\beta \in \rootl^-} \shc_\beta, \quad
\text{where} \ \shc_\beta \seteq\shc\cap R(-\beta) \gmod,
\eneqn
and we have the grading shift functor $q$ on $\shc$.
Hence we have
$$K(\shc_\beta)\subset\Um_\beta,$$
and $K(\shc)$ has a $\Z[q^{\pm1}]$-basis consisting of
the isomorphism classes of self-dual simple modules.

\begin{definition} \label{def:admissible}
A pair $(\{M_i\}_{i\in\K}, \widetilde B)$ is called \emph{admissible} if
\bnum
\item  $\{M_i\}_{i\in\K}$ is a family of real simple  self-dual  objects of $\shc$ which commute with each other,
\item $\widetilde B$ is an integer-valued $\K \times\Kex$-matrix with
skew-symmetric principal part,
\item
 for each $k\in\Kex$, there exists a  self-dual  simple object $M'_k$ of $\shc$
 such that
 there is an exact sequence in $\shc$
 \eq
&&0 \to q \sodot_{b_{ik} >0} M_i^{\snconv b_{ik}} \to q^{\tLa(M_k,M_k')} M_k \conv M_k' \to
 \sodot_{b_{ik} <0} M_i^{\snconv (-b_{ik})} \to 0,
\label{eq:ses graded mutation KLR}
 \eneq
and
$M_k'$ commutes with $M_i$  for any  $i \neq k$.
  \end{enumerate}
\end{definition}

Note that $M'_k$ is uniquely determined by $k$ and
$(\{M_i\}_{i\in\K}, \widetilde B)$. Indeed, it follows from
$q^{\tLa(M_k,M_k')}M_k\hconv M'_k\simeq\sodot_{b_{ik} <0} M_i^{\snconv(-b_{ik})}$ and
\cite[Corollary 3.7]{KKKO14}.
Note also that if there is an epimorphism $q^mM_k\conv M'_k\epito
\sodot_{b_{ik} <0} M_i^{\snconv(-b_{ik})}$ for some $m\in\Z$, then $m$ should coincide with $\tL(M_k,M'_k)$ by Lemma~\ref{lem:selfdual} and Lemma~\ref{lem:multipro}.

\smallskip
For an admissible pair  $(\{M_i\}_{i\in\K}, \widetilde B)$, let
$\La=(\La_{ij})_{i,j\in\K}$
be the skew-symmetric matrix
given by $\La_{ij}=\Lambda(M_i,M_j)$.
and let $D=\{d_i\}_{i\in\K}$ be the family of elements of $\rootl^-$ given by
$d_i=\wt(M_i)$.

\smallskip
Now we can simplify  the conditions in Definition \ref{def:quantum monoidal seed} and
Definition \ref{def:monoidal mutation} as follows.

\begin{prop} \label{prop:condition simplified}
Let $(\{M_i\}_{i\in\K},\widetilde B)$ be an admissible pair in $\shc$,
and let $M'_k$ $(k\in\Kex)$ be as in {\rm Definition~\ref{def:admissible}}.
    Then we have the following properties.
  \bna
      \item The quadruple $\seed\seteq(\{M_i\}_{i\in\K}, -\La,\widetilde B,D)$
    is a quantum monoidal seed in $\shc$.\label{item:1}
    \item The self-dual simple object $M_k'$ is real for every $k\in\Kex$.
\label{item:2}
    \item The quantum monoidal seed $\seed$ admits a mutation in each direction $k\in\Kex$.\label{item:3}
\item $M_k$ and $M'_k$ are simply-linked for any $k\in\Kex$
\ro i.e., $\de(M_k,M'_k)=1$\rf.\label{item:sl}
\item For any $j\in\K$ and $k\in\Kex$, we have
\eq&&
\ba{l}\La(M_j,M'_k)=-\La(M_j,M_k)-\sum_{b_{ik}<0}\La(M_j,M_i)b_{ik},\\[1ex]
\La(M'_k,M_j)=-\La(M_k,M_j)+\sum_{b_{ik}>0}\La(M_i,M_j)b_{ik}.
\ea\label{eq:Larel}
\eneq\label{item:Larel}
  \end{enumerate}
\end{prop}
\Proof
\eqref{item:sl} follows from the exact sequence
\eqref{eq:ses graded mutation KLR} and Lemma \ref{lem:crde}.

\medskip\noi
\eqref{item:2} follows from the exact sequence
\eqref{eq:ses graded mutation KLR}
by applying Corollary \ref{cor:real} to the case
$$M=M_k, \ N=M'_k, \ X= q  \sodot_{b_{ik} > 0} M_i^{\snconv b_{ik}}
\ \text{and} \  Y= \sodot_{b_{ik} < 0} M_i^{\snconv(-b_{ik})}. $$

\medskip
\noi
\eqref{item:Larel} follows from
\begin{align*}
\La(M_j,M_k)+\La(M_j,M'_k)&=\La(M_j,M_k\hconv M'_k)
=\La\bl M_j,\sodot_{b_{ik} <0} M_i^{\snconv(-b_{ik})}\br\\
&=\sum_{b_{ik}<0}\La(M_j,M_i)(-b_{ik})
\end{align*}
and
\begin{align*}
\La(M_k,M_j)+\La(M'_k,M_j)&=\La(M'_k\hconv M_k,M_j)
=\La\bl\sodot_{b_{ik} >0} M_i^{\snconv b_{ik}}, M_j\br\\
&=\sum_{b_{ik}>0}\La(M_i,M_j)b_{ik}.
\end{align*}

\medskip %\noi
Let us show \eqref{item:1}.
 The conditions (i)--(v) in Definition \ref{def:quantum monoidal seed} are satisfied by the construction.
 The condition (vi) follows from Proposition \ref{prop:real simple}
and the fact that $M_i$ is real simple
for every $i\in\K$.
 The condition (viii) is nothing but Lemma \ref{lem:tLa even}.
 The condition (ix) follows easily from the fact that
the weights of the first and the last terms in the exact sequence \eqref{eq:ses graded mutation KLR} coincide.

%\medskip %\noi
Let us show  the condition (vii)  in Definition \ref{def:quantum monoidal seed}.
By \eqref{eq:Larel} and \eqref{item:sl} of this proposition, we have
\begin{align*}
2\delta_{jk}=2\de(M_j,M'_k)
&=-2\de(M_j,M_k)-\sum_{b_{ik}<0}\La(M_j,M_i)b_{ik}
+\sum_{b_{ik}>0}\La(M_i,M_j)b_{ik}\\
&=-\sum_{b_{ik}<0}\La(M_j,M_i)b_{ik}-\sum_{b_{ik}>0}\La(M_j,M_i)b_{ik}
=-\sum_{i\in\K}\La(M_j,M_i)b_{ik}
\end{align*}
for $k\in\Kex$ and $j\in\K$.
Thus we have shown that $\seed$ is a quantum monoidal seed
in $\shc$.

\medskip %\noi
Let us show \eqref{item:3}. Let $k\in\Kex$. The exact sequence
\eqref{eq:ses graded mutation1}
follows from  \eqref{eq:ses graded mutation KLR}
and the equality
\eq
\tL(M_k,M'_k)=\dfrac{1}{2}\bl(\wt(M_k,M'_k)
-\sum_{b_{ik}<0}\La(M_k,M_i)b_{ik}\br=m_k,
\eneq
which is an immediate consequence of \eqref{eq:Larel}.

 Similarly, taking the dual of
the exact sequence \eqref{eq:ses graded mutation KLR},
we obtain an exact sequence
 \eqn
&&0 \to  \sodot_{b_{ik} <0} M_i^{\snconv (-b_{ik})}  \to q^{-\tLa(M_k,M_k')+(\wt M_k, \wt M_k')} M_k' \conv M_k  \to q^{-1}\sodot_{b_{ik} >0} M_i^{\snconv b_{ik}}
 \to 0,
 \eneqn
which gives the exact sequence \eqref{eq:ses graded mutation2}.
\smallskip

It remains to prove that $\mu_k(\seed)\seteq(\{M_i\}_{i\neq k}\cup\{M_k'\},\mu_k(-\La),
\mu_k(\widetilde B), \mu_k(D))$ is
a quantum monoidal seed in $\shc$ for any $k\in\Kex$.

We see easily that $\mu_k(\seed)$
satisfies the conditions (i)--(iv) and (vii)--(ix) in Definition~\ref{def:quantum monoidal seed}.

For  the condition (v), it is enough to show that
for $i,j\in\K$ we have
\eqn \mu_k(-\La)_{ij} = -\Lambda(\mu_k(M)_i,\mu_k(M)_j),
\eneqn
where $\mu_k(M)_i = M_i$ for $i\neq k$ and $\mu_k(M)_k=M_k'$.
In the case $i \neq k$ and $j \neq k$, we have
\eqn \mu_{k}(-\La)_{ij}=-\Lambda(M_i,M_j)=-\Lambda(\mu_k(M_i),\mu_k(M_j)).
\eneqn
The other cases follow from \eqref{eq:Larel}.

 The condition (vi) in Definition~\ref{def:quantum monoidal seed}
for $\mu_k(\seed)$ follows from Proposition \ref{prop:real simple} and the fact that $\{\mu_k(M)_i\}_{i\in\K}$ is a commuting family of real simple modules.
\QED

Now we are ready to give  one of our main theorems.
\Th\label{th:main}
Let  $(\{M_i\}_{i\in \K},\widetilde B)$ be an admissible pair in $\shc$
and set $$\seed=(\{M_i\}_{i\in\K}, -\La,\widetilde B,D)$$
as in {\rm Proposition~\ref{prop:condition simplified}}.
We set
$[\seed]\seteq(\{q^{-\frac{1}{4}(\wt(M_i), \wt(M_i))}[M_i]\}_{i\in\K}, -\La,\widetilde B,D)$.
We assume further  that
\eq
&&\hs{1ex}\parbox[t]{79ex}{
 The $\Q(q^{1/2})$-algebra $\Q(q^{1/2})\tens\limits_{\Z[q^{\pm1}]}K(\shc)$
is isomorphic to $\Q(q^{1/2})\tens\limits_{\Z[q^{\pm1}]}\mathscr A_{q^{1/2}}([\seed])$.
}\label{eq:Cluster}
\eneq
Then, for each $x\in\Kex$, the pair
$\bl\{\mu_x(M)_i\}_{i\in\K},\mu_x(\widetilde B)\br$
 is admissible in $\shc$.
\enth
\Proof In Proposition ~\ref{prop:condition simplified} (\ref{item:2}), we have already  shown that
the condition (i) in Definition~\ref{def:admissible} holds for
$(\{\mu_x(M)_i\}_{i\in\K},\mu_x(\widetilde B))$.
The condition (ii) is clear from the definition.
Let us show (iii).
Set $N_i\seteq\mu_x(M)_i$ and
$b_{ij}' \seteq\mu_x(\widetilde B)_{ij}$ for $i\in \K$ and $j\in \Kex$.
It is enough to show that, for any $y\in\Kex$,
there exists a self-dual simple module $M_y'' \in \shc$ such that
there is a short exact sequence
\eq
\xymatrix{
0 \ar[r] &  q \sodot_{b'_{iy} > 0} N_i^{\snconv b'_{iy}} \ar[r]
&q^{\tLa(N_y,M_y'')}  N_y \conv  M_y'' \ar[r]
& \sodot_{b'_{iy} <0} N_i^{\snconv (-b'_{iy})}   \ar[r] & 0
}\label{eq:seqdes}
\eneq
and
$$\de(N_i,M_y'')=0 \quad \text{for $i \neq y$.}$$

If $x=y$, then $b'_{iy}=-b_{ix}$ and hence $M''_y=M_x$ satisfies the desired condition.

Assume that $x\not=y$ and $b_{xy}=0$. Then $b'_{iy} = b_{iy}$ for any $i$
and $N_i=M_i$ for any $i\not=x$.
Hence  $M_y''=\mu_y(M)_y$ satisfies the desired condition.

\medskip
We will show the assertion in the case $b_{xy} > 0$.
We omit the proof of the case $b_{xy} <0$ because it
can be shown in a similar way.

Recall that we have
\eq \label{eq:bxy>0 mutation}
b_{iy}'=\begin{cases}
b_{iy}+b_{ix}b_{xy} & \text{if } b_{ix} >0, \\
b_{iy} & \text{if } b_{ix} \le 0
\end{cases}
\eneq
for $i  \in \K$ different from $x$ and $y$.

Set
\eqn
&&M_x'\seteq\mu_x(M)_x, \hs{3ex} M_y'\seteq\mu_y(M)_y, \\
&&C\seteq\sodot_{b_{ix} >0} M_i^{\snconv b_{ix}}, \quad  S\seteq\sodot_{b_{ix} <0, \ i\neq y} M_i^{\snconv -b_{ix}}, \\
&&P\seteq\sodot_{b_{iy} >0, i\neq x} M_i^{\snconv b_{iy}}, \quad  Q:=\sodot_{b_{iy}' <0, \ i\neq x} M_i^{\snconv-b_{iy}'}, \\
&&A\seteq\sodot_{b_{iy}' \le 0, \ b_{ix} >0} M_i^{\snconv b_{ix}b_{xy}} \nconv \sodot_{\substack{b_{iy} <0,\; b_{iy}' >0,\;b_{ix} >0}} M_i^{\snconv -b_{iy}}
 \\
&&\hs{3ex}\simeq
\sodot_{\substack{b_{iy} <0,\; b_{ix} >0}}
M_i^{\snconv \min(b_{ix}b_{xy},-b_{iy})},\\
&&B\seteq\sodot_{ b_{iy} \ge 0,  \ b_{ix} >0} M_i^{\snconv b_{ix}b_{xy}} \nconv \sodot_{\substack{b_{iy}' >0, \; b_{iy} <0,\;b_{ix} >0}} M_i^{\snconv b_{iy}'}.
 \eneqn

Then    using \eqref{eq:bxy>0 mutation} repeatedly,
we have
\begin{align*}
Q \nconv A \simeq \sodot_{b_{iy} < 0} M_i^{\snconv -b_{iy}}, \quad
A \nconv B \simeq C^{\snconv b_{xy}}, \quad \text{and}  \quad
 B\nconv P \simeq \nconv \sodot_{ b_{iy}' > 0} M_i^{\snconv b_{iy}'}.
\end{align*}
Set
 \eqn
&&L\seteq(M_x')^{\snconv b_{xy}},   \quad %\text{and} \quad
V\seteq M_x^{\snconv b_{xy}}
\eneqn
 and set
\eqn
X \seteq \sodot_{b_{iy} >0} M_i^{\snconv b_{iy}}\simeq
M_x^{\snconv b_{xy}} \nconv P =V \nconv P,
 \quad Y\seteq \sodot_{b_{iy} <0} M_i^{\snconv-b_{iy}}\simeq Q \nconv A.
\eneqn

Then \eqref{eq:seqdes}  is read as%reads as
\eq
\xymatrix{
0 \ar[r] &  q (B \nconv P)\ar[r]  &q^{\tL(M_y,M''_y)}M_y \conv M''_y \ar[r] &
L  \nconv Q\ar[r] & 0.
}\label{eq:seqdes'}
\eneq

Note that we have
\eq
&&0 \to q C  \to q^{\tLa(M_x,M_x')}M_x \conv M_x' \to  M_y^{\snconv b_{xy}} \nconv S \to 0,  \\
&&0 \to q X \to  q^{\tLa(M_y,M_y')}M_y \conv M_y' \to  Y \to 0. \label{eq:sesMy}
\eneq

Taking the convolution products of $L=(M_x')^{\snconv b_{xy}}$  and \eqref{eq:sesMy}, we obtain
\eqn
&&
\xymatrix@R=.5ex{0 \ar[r]& qL \conv X  \ar[r]&
 q^{\tLa(M_y,M_y')} L \conv (M_y \conv M_y')\ar[r]& L \conv Y \ar[r]& 0, \\
0\ar[r]& qX \conv L \ar[r]&
q^{\tLa(M_y,M_y')} (M_y \conv M_y') \conv L \ar[r]&  Y\conv L \ar[r]& 0.}
\eneqn

Since $L$ commutes with  $M_y$,
we have
\eqn
&&\Lambda(L, Y)=\Lambda(L, M_y \hconv M_y') \\
 &&=\Lambda(L, M_y) +\Lambda(L, M_y') = \Lambda(L,M_y \conv M_y' ).
\eneqn

On the other hand, we have
\eqn
&&\Lambda(M_x',X) - \Lambda(M_x',Y)  \\
&&=\Lambda(M_x', \sodot_{b_{iy} >0} M_i^{\snconv b_{iy}})-\Lambda(M_x',\sodot_{b_{iy} <0} M_i^{\snconv -b_{iy}})  \\
&&=\sum_{b_{iy} >0} \Lambda(M_x',M_i)b_{iy} -\sum_{b_{iy} <0} \Lambda(M_x',M_i)(-b_{iy})  \\
&&=\sum_{i\in\K} \Lambda(M_x', M_i)b_{iy}
=\sum_{i \neq x} \Lambda(M_x',M_i)b_{iy} + \Lambda(M_x',M_x)b_{xy} \\
&&=\sum_{i \neq x} \Lambda(M_x', M_i)(b_{iy}' -\delta(b_{ix} >0)b_{ix} b_{xy}) + \Lambda(M_x',M_x)b_{xy} \\
&&=\sum_{i \neq x} \Lambda(M_x', M_i) b_{iy}'
-\sum_{b_{ix} >0} \Lambda(M_x', M_i)b_{ix}b_{xy} + \Lambda(M_x',M_x)b_{xy} \\
&&\mathop=\limits_{\mathrm(a)}0-\Lambda(M_x', \sodot_{b_{ix} >0} M_i^{\snconv b_{ix}})b_{xy} + \Lambda(M_x',M_x)b_{xy} \\
&&=\bl-\Lambda(M_x',\sodot_{b_{ix} >0} M_i^{\snconv b_{ix}} ) + \Lambda(M_x',M_x) \br b_{xy} \\
&&=(-\Lambda(M_x', M_x' \hconv M_x) + \Lambda(M_x',M_x) )b_{xy} \\
&&=(-\Lambda(M_x', M_x')-\Lambda(M_x', M_x) + \Lambda(M_x',M_x) )b_{xy}=0.
\eneqn
Note that we used the compatibility of the pair
$\bl\bl-\Lambda(\mu_x(M_i),\mu_x(M_j))\br_{i,j\in\K}, \, \mu_x(\widetilde B)\br$
when we derive the equality (a).

Since $L=(M'_x)^{\snconv b_{xy}}$, the equality $\La(M'_x,X)=\La(M'_x,Y)$ implies
\eqn
&& \Lambda(L,X)
=\Lambda(L,Y)=\La(L, M_y \conv M_y'). \label{eq:LaLX=LaLY}
\eneqn
Hence  the following  diagram is commutative by Proposition \ref{prop:inequalities} (i):
\eqn \label{eq:commutative ses}
&&\ba{c}\xymatrix{
0 \ar[r] & q L \conv X \ar[r] \ar[d]^{\rmat{L,X}} &  q^{\tLa(M_y,M_y')} L \conv (M_y \conv M_y')
 \ar[d]^{\rmat{L,M_y \circ M_y'}} \ar[r] &  L \conv Y \ar[r] \ar[d]_{\rmat{L,Y}} ^{\bwr}&0\\
0 \ar[r] & q^{d+1} X \conv L \ar[r]  & q^{d+\tLa(M_y,M_y')} (M_y \conv M_y') \conv L \ar[r] & q^{d}\, Y \conv L \ar[r] & 0,
}\ea
\eneqn
where $d=-\Lambda(L,X)=-\Lambda(L,M_y \conv M_y')=-\Lambda(L,Y)$.
Note that  since $L=(M_x')^{\snconv b_{xy}}$ commutes with $Q$ and $A$, $\rmat{L,Y}$ is an isomorphism.  Hence
we have
$$\Im(\rmat{L,Y}) \simeq L \conv Y.$$
 Therefore we obtain  an exact sequence
\eq
\xymatrix{0\ar[r]&\Im(\rmat{L,X})\ar[r]&\Im(\rmat{L, M_y\circ M'_y})\ar[r]
&L\circ Y\ar[r]&0.}\label{eq:exIm}
\eneq

On the other hand, $\rmat{L, M_y \circ M_y'}$ decomposes
(up to a grading shift) by Lemma \ref{lem:composition} as follows:
\eqn
\xymatrix@C=6em{
L \conv M_y \conv M_y'  \ar[r]^\sim_{\rmat{L,M_y} \circ M_y'}
\ar@/^2pc/ [rr]^{\rmat{L, M_y \circ M_y'}}&
 M_y \conv L \conv M_y'  \ar[r]_{M_y \circ \rmat{L,M_y'}} &
 M_y  \conv M_y' \conv L.
 }
 \eneqn
Since $L=(M_x')^{\snconv b_{xy}}$ commutes with $M_y$,
the homomorphisms $\rmat{L,M_y} \conv M_y'$ is an isomorphism and hence we have
\eqn
\Im
 (\rmat{L, M_y \circ M_y'}) \simeq  M_y \conv (L\hconv M'_y)
\quad\text{up to a grading shift.}
 \eneqn

Similarly, $\rmat{L, X}$ decomposes
(up to a grading shift) as follows:
\eqn
\xymatrix@C=6em{
L \conv V \conv P  \ar[r]_{\rmat{L,V} \circ P}
\ar@/^1.5pc/ [rr]^{\rmat{L,X}}&
 V \conv L \conv P  \ar[r]_{ V\circ \rmat{L,P}}^\sim &
V \conv P  \conv L.
 }
 \eneqn
Since $L$ commutes with $P$,
the  homomorphism  $V \conv \rmat{L,P} $ is an isomorphism and hence we have
  \eqn
\Im
 (\rmat{L,X})
  \simeq (L \hconv V) \conv P
\simeq \bl(M_x')^{\conv b_{xy}} \hconv M_x^{\conv b_{xy}}\br \conv P
\quad\text{up to a grading shift.}
 \eneqn
On the other hand,
Lemma~\ref{lem:MN} implies that
\eqn
(M_x')^{\conv b_{xy}} \hconv M_x^{\conv b_{xy}}
\simeq (M_x' \hconv M_x)^{\conv b_{xy}}
\simeq  C^{\conv b_{xy}}
\simeq B \nconv A
\eneqn
and hence we obtain
  \eqn
\Im
 (\rmat{L,X})
\simeq (B \nconv P)  \nconv A\quad\text{up to a grading shift.}
\eneqn
Thus the exact sequence \eqref{eq:exIm} becomes the exact sequence in $\shc$:
\eq
&&\xymatrix{
0 \ar[r] &  q^m (B \nconv P) \nconv A \ar[r]  &q^nM_y \conv (L\hconv M_y') \ar[r] & (L  \nconv Q) \nconv A \ar[r] & 0
}\label{eq:exML}
\eneq
for some $m,n\in\Z$.
Since $(L  \nconv Q) \nconv A$ is self-dual, $n=\tL(M_y,L\hconv M_y')$.
On the other hand, by Proposition \ref{prop:Lambda decomp} (i) and  Proposition \ref{prop:condition simplified} (\ref{item:sl}), we have
\eqn
\de(M_y,L \hconv M_y') \le \de(M_y, L) + \de(M_y, M_y') =1.
\eneqn
By the exact sequence \eqref{eq:exML},
$M_y \conv (L\hconv M_y')$ is not simple and we conclude
$$\de(M_y,L \hconv M_y')=1.$$
Then Lemma~\ref{lem:crde} implies that $m=1$.
Thus we obtain an exact sequence in $\shc$:
\eq
&&\xymatrix@C=3ex{
0 \ar[r] &  q (B \nconv P) \nconv A \ar[r]  &q^{\tL(M_y,L\shconv M_y')}M_y \conv (L\hconv M_y') \ar[r] & (L  \nconv Q) \nconv A \ar[r] & 0.
}\label{eq:exML1}
\eneq

Now we shall rewrite \eqref{eq:exML1} by using $\scbul\conv A$ instead of
$\scbul\nconv A$.
We have \eqn
&&\tLa(B,A)+\tLa(A,A)
=b_{xy}\tLa(C,A)
=b_{xy}\tLa(M_x' \hconv M_x,A) \\
&&\hs{10ex}=b_{xy}\tLa(M_x',A)+b_{xy}\tLa(M_x,A)=\tL(L,A)+b_{xy}\tLa(M_x,A).
\eneqn
On the other hand, the exact sequence \eqref{eq:sesMy} gives
\eqn
&&b_{xy}\tLa(M_x,A)+\tLa(P,A)=\tL(X,A)=\tL(M'_y\hconv M_y,A)\\
&&\hs{5ex}=\tL(M'_y,A)+\tL(M_y,A)=\tL(M_y\hconv M'_y,A)=
\tL(Y,A)=\tL(Q,A)+\tL(A,A).
\eneqn
It follows that
\eqn
&&\tLa(B \conv P,A)=\tLa(B,A) + \tLa(P,A) \\
&&\hs{5ex}=  \bl \tL(L,A)+b_{xy}\tLa(M_x,A)-\tLa(A,A)\br
+\bl \tL(Q,A)+\tL(A,A)-b_{xy} \tLa(M_x,A)\br\\
&&\hs{5ex}=\tLa(L,A) + \tLa(Q,A)=\tLa(L \conv Q, A). \label{eq:qpower3}
\eneqn

 Hence we have
\eqn
&&\xymatrix{
0 \ar[r] &  q (B \nconv P) \conv A \ar[r]  &q^{c}  M_y \conv (L\hconv M_y') \ar[r] & (L  \nconv Q) \conv A \ar[r] & 0,
}\label{eq:LQA}
\eneqn
where $c=\tLa(M_y, L \hconv M_y')-\tLa(B \nconv P,A)$ by Lemma \ref{lem:selfdual}.

Thus we obtain the identity in $K(R \gmod)$:
\eqn
q^c  [M_y] [L \hconv M_y'] = \bl q [B \nconv P] + [L \nconv Q]\br [A].
\eneqn

On the other hand,  the hypothesis \eqref{eq:Cluster} implies that
there exists $\phi \in \Q(q^{1/2}) \otimes_{\Z[q^{\pm1}]} K(\shc)$
corresponding to $\mu_y\mu_x([M])$ so that
it satisfies
\eq
[M_y] \phi = q [B \nconv P] + [L \nconv Q]
\label{eq:Mphi}
\eneq
and
\eq
\phi [\mu_x(M)_i] = q^{\la'_{yi}} [\mu_x(M)_i] \phi
\quad
\text{for} \ i \neq y,
\label{eq:phiA}
\eneq
where $\mu_y\mu_x(-\La)=(\la'_{ij})_{i,j\in\K}$.

Hence, in $\Q(q^{1/2}) \otimes_{\Z[q^{\pm1}]} K(\shc)$, we have
\eqn
[M_y] \phi [A] = \bl q [B \nconv P] + [L \nconv Q]\br [A ]=  q^c  [M_y][L \hconv M_y'].
\eneqn
Since $\Q(q^{1/2}) \otimes_{\Z[q^{\pm1}]} K(\shc)$ is a domain, we conclude that
\eqn
\phi [A] = q^c [L \hconv M_y'].
\eneqn

On the other hand,  \eqref{eq:phiA}  implies
\eqn
\phi [A] = q^{l} [A] \phi\quad\text{for some $l\in\Z$.}
\eneqn
Hence, Theorem \ref{thm:divisible} implies that, when we write
\eqn
\phi =\sum_{b \in B(\infty)} a_b [L_b]\quad\text{for some $a_b \in \Q(q^{1/2})$,}
\eneqn
we have $$L_b\conv A \simeq q^{l} A\conv L_b \quad \text{whenever } \ a_b\not=0.$$
In particular,  each module  $L_b\conv A $ with $a_b\not=0 $ is simple because $A$ is a real simple module.
Thus we obtain
\eqn
q^c [ L \hconv M_y' ] =\phi [A] = \sum_{b \in B(\infty)} a_b [L_b \conv A].
\eneqn
 Since $L \hconv M_y'$ is simple,
 there exists $b_0$ such that $L_{b_0} \conv A$ is isomorphic to $L \hconv M_y'$ up to a grading shift, and
 $a_b=0$ for $b \neq b_0$.
Set $M_y'':=L_{b_0}$.
Then we conclude that $\phi[A] = q^m [M_y'' \conv A]  =q^m[M_y''][A] $ so that
$$ \phi=q^m[M''_y] \quad \text{for some } \ m \in\Z.$$
We emphasize that $M_y''$ is a self-dual simple module in $R \gmod$ which satisfies that $M_y'' \conv A
\simeq L \hconv M_y $ up to a grading shift.

% we conclude that
%there exists  a self-dual simple module  $M_y''$ in $R \gmod$  such that
%$M''_y$ commutes with $A$ and
%$$ \phi=q^m[M''_y]$$
%for some $m\in\Z$.
Now \eqref{eq:Mphi}  implies
$$q^m[M_y\conv M''_y] = q [B \nconv P] + [L \nconv Q].$$
Hence there exists an exact sequence
$$0\To W\To q^m\;M_y\conv M''_y\To Z\To 0,$$
where $W=q B \nconv P$ and $Z=L \nconv Q$ or
$W=L \nconv Q$ and $Z=q B \nconv P$.
By Lemma~\ref{lem:crde}, the  second case does not occur and we have an exact sequence
$$0\To q B \nconv P\To q^m\;M_y\conv M''_y\To L \nconv Q\To 0.$$

Since $M_y$, $M_y''$ and $L \nconv Q$ are self-dual, we have
$m=\tLa(M_y,M_y'')$,
and we obtain the desired short exact sequence \eqref{eq:seqdes'}.

\vskip 1em
Since $\phi$ commutes with $[\mu_x (M)_i]$  up to a power of $q$ in
$K(\shc)$,
and
$\mu_x(M)_i$ is real simple,
$M_y''$ commutes with $\mu_x (M)_i$ for $i\neq y$, by Corollary \ref{cor:qcomm_module}.
\QED

\Cor\label{cor:main}
Let  $(\{M_i\}_{i\in\K},\widetilde B)$ be an admissible pair in $\shc$.
Under the assumption \eqref{eq:Cluster},
$\shc$ is a monoidal categorification of the quantum cluster algebra
$\mathscr A_{q^{1/2}}([\seed])$.
Furthermore,  the following statements hold{\rm:}
\bnum
  \item The quantum monoidal seed $\seed=(\{M_i\}_{i\in\K},  -\La,\widetilde B,D)$
  admits successive mutations in all directions.
  \item Any cluster monomial in $\Z[q^{\pm1/2}] \otimes_{\Z[q^\pm1]} K(\shc)$
   is the isomorphism class of a  real simple object in $\shc$ up to a power of $q^{1/2}$.
  \item Any cluster monomial in $\Z[q^{\pm1/2}]\otimes_{\Z[q^\pm1]} K(\shc)$
is a Laurent polynomial of the initial cluster variables with coefficient
in $\Z_{\ge0}[q^{\pm1/2}]$.
\end{enumerate}
\encor

\Proof
(i) and (ii) are straightforward.

Let us show (iii). Let $x$ be a cluster monomial.
By the Laurent phenomenon (\cite{BZ05}),
we can
write
$$xX^{\mathbf{c}}=\sum_{\bfa\in\Z_{\ge0}^\K}c_\bfa X^\bfa,$$
where
$X=(X_i)_{i\in \K}$ is the initial cluster,  $\mathbf{c}\in\Z_{\ge0}^\K$
and $c_\bfa\in \Q(q^{\pm1/2})$.
Since $x$ and $X^{\mathbf{c}}$ are the isomorphism classes of simple modules
up to a power of $q^{1/2}$,
their product $xX^{\mathbf{c}}$ can be written as
a linear combination of
the isomorphism classes of simple modules
with coefficients in $\Z_{\ge0}[q^{\pm1/2}]$.
Since
every $X^\bfa$ is the isomorphism class of a simple module
up to a power of $q^{1/2}$,
we have
$c_\bfa\in \Z_{\ge0}[q^{\pm1/2}]$.
\QED

%%%%%%%%%%%%%%%%%%%%%%%%%%%%%%%%%%%%%%%%%%%%%%%%%%%%%%%
%Part 2
%\part{Admissible seeds for the unipotent quantum coordinate rings}

\section{Quantum coordinate rings and modified quantized enveloping algebras}

\subsection{Quantum coordinate ring} Let $U_q(\g)^*$ be $\Hom_{\Q(q)}(U_q(\g),\Q(q))$. Then the comultiplication $\comp$  (see \eqref{eq: comp})   induces the
 multiplication $\mu$ on $U_q(\g)^*$ as follows:
\begin{align*}
 \mu \cl  & U_q(\g)^* \tens U_q(\g)^* \to (\U\tens\U)^*\To[\;(\comp)^*\;]
U_q(\g)^* .
\end{align*}
Later on, it will be convenient to use Sweedler's notation
$\comp(x)=x_{(1)} \tens x_{(2)}$.
With this notation,
$$\text{$\bl fg\br(x)=f(x_{(1)})\,g(x_{(2)})$\quad
for $f,g \in U_q(\g)^*$ and  $x\in\U$.}$$

The $U_q(\g)$-bimodule structure on $U_q(\g)$ induces
a $U_q(\g)$-bimodule structure on $U_q(\g)^*$. Namely,
\begin{align*}
(x \cdot f) (v) =f(vx) \quad \text{ and } \quad (f \cdot x) (v)
=f(xv) \quad\text{ for } f \in U_q(\g)^* \text{ and } x,v \in U_q(\g).
\end{align*}

Then the multiplication $\mu$ is a morphism of a
$U_q(\g)$-bimodule, where $U_q(\g)^* \tens U_q(\g)^*$ has the
structure of a $U_q(\g)$-bimodule via $\comp$. That is, for $f,g \in U_q(\g)^*$ and  $x,y \in U_q(\g)$,  we have
$$ x(fg)y=(x_{(1)}fy_{(1)})(x_{(2)}gy_{(2)}),$$
where $\comp(x)=x_{(1)}\tens x_{(2)}$ and $\comp(y)=y_{(1)}\tens y_{(2)}$.

\begin{definition}
We define the {\em quantum coordinate ring} $A_q(\g)$ as follows:
$$A_q(\g) = \{ u \in U_q(\g)^* \ | \ \text{ $U_q(\g)u$ belongs to $\Oint(\g)$ and $u U_q(\g)$ belongs to $\Oint^\ri(\g)$}\}. $$
\end{definition}

Then, $A_q(\g)$ is a
subring of $U_q(\g)^*$
because {\rm (i)} $\mu$ is $U_q(\g)$-bilinear, {\rm (ii)}  $\Oint(\g)$ and
$\Oint^\ri(\g)$ are closed under the tensor product.

We have the weight decomposition:
$A_q(\g) = \soplus_{\eta,\zeta \in \wl} A_q(\g)_{\eta,\zeta}$, where
$$A_q(\g)_{\eta,\zeta} \seteq \{  \psi \in A_q(\g) \ | \  q^{h_\li} \cdot \psi \cdot q^{h_\ri} = q^{\langle h_\li,\eta \rangle +\langle h_\ri,\zeta \rangle  } \psi
\text{ for } h_\li,h_\ri \in \wl^\vee \},$$
For $\psi \in A_q(\g)_{\eta,\zeta}$, we write
$$ \wt_\li(\psi)=\eta \quad \text{ and } \quad \wt_\ri(\psi)=\zeta.$$

\medskip
For any $V\in\Oint(\g)$, we have the $U_q(\g)$-bilinear homomorphism
$$\Phi_V\colon V\tens (\Dv V)^\ri\to A_q(\g)$$
given by
$$\Phi_V(v\tens \psi^\ri)(a)=\lan \psi^\ri, av\ra=\lan \psi^\ri a, v\ra \quad\text{for $v\in V$,
$\psi\in\Dv V$ and $a\in U_q(\g)$.}$$

\begin{prop} [{\cite[Proposition 7.2.2]{Kas93}}] We have an isomorphism $\Phi$ of $\Uqg$-bimodules
\begin{equation} \label{eq: Phi}
\Phi\cl  \soplus_{\lambda \in \wl^+} V(\lambda) \tens_{\Q(q)}
V(\lambda)^\ri \overset{\sim}{\longrightarrow} A_q(\g)
\end{equation}
given by $\Phi|_{V(\lambda) \tens_{\Q(q)} V(\lambda)^\ri} = \Phi_\lambda \seteq \Phi_{V(\lambda)}$. Namely,
$$ \Phi(u \tens v^\ri)(x)= \lan v^\ri,xu \ra = \lan v^\ri x,u \ra =(v,xu) \text{ for any } v,u \in V(\lambda) \text{ and } x \in U_q(\g).$$
\end{prop}

We introduce the
 crystal basis $\big(L^\up(A_q(\g)), B(A_q(\g)) \big)$ of $\Ag$ as the images by $\Phi$ of
$$ \soplus_{\lambda \in \wl^+} L^\up(\lambda) \otimes L^\up(\lambda)^\ri \text{ and }
\bigsqcup_{\lambda \in \wl^+} B(\lambda) \otimes B(\lambda)^\ri.$$
Hence it is a crystal base with respect to the left action of
$\U$ and also the right action of $\U$.
We sometimes write by
$e_i^*$ and $f_i^*$ the operators  of $\Ag$ obtained by the right actions of
$e_i$ and $f_i$.

\medskip
We define the $\A$-form of $\Ag$ by
$$\Ag_\A\seteq\set{\psi\in\Ag}{\lan \psi,\, \U_\A\ran\subset\A}.$$
We define the bar-involution $-$ of $A_q(\g)$ by
$$ \overline{\psi}(x) = \overline{\psi(\overline{x})} \quad \text{ for } \psi \in A_q(\g), \  x \in U_q(\g).$$
Note that the bar-involution is not a ring homomorphism but it satisfies
\eqn
&&\ol{\psi\;\theta\;}=q^{(\wt_\lt(\psi),\wt_\lt(\theta))
-(\wt_\rt(\psi),\wt_\rt(\theta))}\;\ol{\theta}\;\ol{\psi}\quad
\text{for any $\psi$, $\theta\in\Ag$.}
\eneqn
Since we do not use this formula and
it is proved similarly to Proposition~\ref{prop:phimul} below,
we omit its proof.

The triple $\big(\Q\tens \Ag_\A,\,L^\up(A_q(\g)),\,\overline{L^\up(A_q(\g))}\big)$ is balanced (\cite[Theorem 1]{Kas93}), and hence there exists an upper global basis of $\Ag$
$$ \B^\up(A_q(\g))\seteq \{ G^\up(b) \ | \ b \in B^\up(A_q(\g)) \}.$$

For $\la\in\Pd$ and $\mu\in W\la$, we denote by $u_{\mu}$
 the unique member of the upper global basis
of $V(\la)$ with weight $\mu$. It is also a member of the lower global basis.

\begin{prop}\label{prop:Agglobal}
Let $\lambda\in\wl^+$, $w\in W$ and $b\in B(\lambda)$.
Then, $\Phi(G^\up(b)\tens u_{w\lambda}^\ri)$ is a member of
the upper global basis of $A_q(\g)$.
\end{prop}
\begin{proof}
The element $\psi\seteq\Phi(G^\up(b)\tens u_{w\lambda}^\ri)$ is
bar-invariant and a member of crystal basis modulo $q L^\up(A_q(\g))$.
For any $P\in U_q(\g)_\A$,
$$\lan\psi, P\ra=\big( u_{w\lambda}, PG^\up(b)\big)$$
belongs to $\A$ because $PG^\up(b)\in V^\up(\lambda)_\A$ and
$u_{w\lambda}\in V^\low(\lambda)_\A$.
Hence $\psi$ belongs to $\Ag_\A$.
\end{proof}

\medskip

The $\Q(q)$-algebra anti-automorphism $\vphi$ of $\U$ induces
a $\Q(q)$-linear automorphism $\vphi^*$
of $\Ag$
by
$$\bl\vphi^* \psi \br(x)=\psi\bl\vphi(x)\br\quad\text{for any $x\in \U$.}$$
We have
$$\vphi^*\bl\Phi(u\tens v^\rt)\br=\Phi(v\tens u^\rt),$$
and
$$\wtl(\vphi^*\psi)=\wtr(\psi)\qtext \wtr(\vphi^*\psi)=\wtl(\psi).$$
It is obvious that
$\vphi^*$ preserves
$\Ag_\A$, $L^\up(\Ag)$ and $ \B^\up(A_q(\g))$.
\Prop\label{prop:phimul}
$$\vphi^*(\psi\theta)=
q^{(\wtr(\psi),\wtr(\theta))-(\wtl(\psi),\wtl(\theta))}
(\vphi^*\psi)(\vphi^*\theta).$$
\enprop
In order to prove this proposition, we prepare
a sublemma.

Let $\xi$ be the $\Q(q)$-algebra automorphism of $\Uq$ given by
$$\xi(e_i)=q_i^{-1}t_ie_i, \quad\xi(f_i)=q_if_it_i^{-1},\quad \xi(q^h)=q^h.$$
We can easily see
\eqn
&&(\xi\tens\xi)\circ\Cmp=\Cmm\circ\xi.
\label{eq:Dexi}
\eneqn

Let $\xi^*$ be the automorphism of $\Ag$ given by
$$(\xi^*\psi)(x)=\psi(\xi(x))\quad\text{for $\psi\in\Ag$ and $x\in\Uq$.}$$

\Sub\label{sub:xi}
We have
\eqn
&&\xi^*(\psi)=q^{A(\wt_\lt(\psi),\wt_\rt(\psi))}\psi,\label{eq:xistar}
\eneqn
 where $A(\la,\mu)=\dfrac{1}{2}\bl(\mu,\mu)-(\la,\la)\br$.
\ensub
\Proof
Let us show that, for each $x$, the following equality
\eq
&&\psi(\xi(x))=q^{A(\wt_\lt(\psi),\wt_\rt(\psi))}\psi(x)\label{eq:xi}
\eneq
holds for any $\psi$.

The equality \eqref{eq:xi} is obviously true for $x=q^h$.
If \eqref{eq:xi} is true for $x$, then
\eqn
&&\xi^*(\psi)(xe_i)=\psi\bl\xi(xe_i)\br=\psi\bl\xi(x)e_it_i)q_i\\
&&=q^{(\al_i,\wt_\lt(\psi))+(\al_i ,\al_i)/2}\psi(\xi(x)e_i)\\
&&=q^{(\al_i,\wt_\lt(\psi))+(\al_i , \al_i)/2}\bl\xi^*(e_i\psi)\br(x)\\
&&=q^{(\al_i,\wt_\lt(\psi))+(\al_i , \al_i)/2+A(\wt_\lt(\psi)+\al_i,\wt_\rt(\psi))}
(e_i\psi)(x).
\eneqn
Since $\Vert\la+\al_i\Vert^2
=\Vert\la\Vert^2+2(\al_i,\la)+ \Vert \al_i\Vert^2 $,
\eqref{eq:xi} holds for $xe_i$.
Similarly if
\eqref{eq:xi} holds for $x$,
then it holds for $xf_i$.
\QED

\Proof[{Proof of Proposition~\ref{prop:phimul}}]
We have
\eqn
&&(\vphi\circ\vphi)\circ \Cmm=\Cmp\circ \vphi.
\label{eq:phiC}
\eneqn
Hence, we have
\begin{align*}
\lan \vphi^*(\psi\theta), x\ran
&=\lan\psi\theta,\vphi(x)\ran\\
&=\lan\psi\tens\theta,\Cmp(\vphi(x))\ran\\
&=\lan\psi\tens\theta,(\vphi\tens\vphi)\circ\Cmm(x)\ran\\
&=\lan \vphi^*(\psi)\tens\vphi^*(\theta),\;\Cmm(x)\ran.
\end{align*}
 It follows that
\begin{align*}
\lan \xi^*(\vphi^*(\psi\theta)),x\ran
&=\lan\vphi^*(\psi\theta), \xi (x)\ran
=\lan \vphi^*(\psi)\tens\vphi^*(\theta),\Cmm(\xi(x))\ran\\
&=\lan\vphi^*(\psi)\tens\vphi^*(\theta),(\xi\tens\xi)\circ\Cmp x \ran\\
&=\lan \xi^*\vphi^*(\psi)\tens\xi^*\vphi^*(\theta),\Cmp x \ran\\
&=\lan\bl\xi^*\vphi^*(\psi)\br\bl\xi^*\vphi^*(\theta)\br,x\ran\\
&=q^{A(\wtr(\psi),\wtl(\psi))+A(\wtr(\theta),\wtl(\theta))}
\lan  (\vphi^*\psi)\,(\vphi^*\theta) ,\,x\ran.
\end{align*}
 Therefore we obtain
$$\vphi^*(\psi\theta)=q^c(\vphi^*\psi)\,(\vphi^*\theta)$$
with
\eqn
&&c=A(\wtr(\psi),\wtl(\psi))+A(\wtr(\theta),\wtl(\theta))-
A(\wtr(\psi)+\wtr(\theta),\wtl(\psi)+\wtl(\theta))\\
&&\hs{2ex}=(\wtr(\psi),\wtr(\theta))-(\wtl(\psi),\wtl(\theta)).
\eneqn
\QED

\subsection{Unipotent quantum coordinate ring} \label{subsec:unipotent}
Let us endow $U_q^+(\g) \tens U_q^+(\g)$ with the algebra structure defined by
$$(x_1 \tens x_2) \cdot (y_1 \tens y_2) = q^{-(\wt(x_2),\wt(y_1))}(x_1y_1 \tens x_2y_2).$$
Let $\Delta_\n$ be the algebra homomorphism $U_q^+(\g) \to U_q^+(\g)
\tens U_q^+(\g)$ given by
$$ \Delta_\n(e_i) = e_i \tens 1 + 1 \tens e_i.$$

Set
$$ A_q(\n) = \soplus_{\beta \in  \rl^-} A_q(\n)_\beta \quad \text{ where } A_q(\n)_\beta \seteq (U^+_q(\g)_{-\beta})^*.$$

Defining the bilinear form $\langle \ \cdot \ , \ \cdot \ \rangle\cl
(A_q(\n) \tens A_q(\n)) \times (U^+_q(\g) \tens U^+_q(\g))  \to \Q(q)$ by
$$ \langle \psi \tens \theta, x \tens y \rangle =\theta(x) \psi(y),$$
we  get an algebra structure on $A_q(\n)$  given by
$$(\psi \cdot \theta)(x) = \langle \psi \tens \theta, \Delta_\n(x) \rangle
= \theta(x_{(1)})\psi(x_{(2)})$$
where $\Delta_\n(x)=x_{(1)} \tens x_{(2)}$.

Since $U_q^+(\g)$ has a $U_q^+(\g)$-bimodule structure, so does
$A_q(\n)$.

We define the $\A$-form of $\An$ by
\eqn
&&\An_\A=\set{\psi\in\An}{\psi\bl\Up_\A\br\subset\A},
\eneqn
 and define the bar-involution $-$ on $\An$ by

$$\ol{\psi}(x)=\ol{\psi(\ol{x})} \quad\text{for $\psi\in\An$ and $x\in\Up$.}
$$

Note that the bar-involution is not a ring homomorphism but it satisfies
\eqn
&&\ol{\psi\;\theta\;}=q^{(\wt(\psi),\wt(\theta))}\;\ol{\theta}\;\ol{\psi}\quad
\text{for any $\psi$, $\theta\in\An$.}
\eneqn

For $i\in I$, we denote by $e_i^*$ the right action of
$e_i$ on $\An$.

\begin{lemma} For $u,v \in A_q(\n)$, we have  $q$-boson relations
\begin{align*}
e_i(uv) = (e_iu)v + q^{(\alpha_i,\wt(u))} u (e_iv) \  \text{ and }
 \ e_i^*(uv) = u(e_i^*v) + q^{(\alpha_i,\wt(v))} (e_i^*u)v.
\end{align*}
\end{lemma}

\begin{proof}
$$
\lan e_i(uv), x\ra=\lan uv, xe_i\ra= \lan u\tens v,\; \Delta_\n(xe_i)\ra.$$
If we set $\Delta_\n x=x_{(1)}\tens x_{(2)}$,
then we have
$$\Delta_\n(xe_i)=(x_{(1)}\tens x_{(2)})(e_i\tens 1+1\tens e_i)
=q^{-(\alpha_i,\wt(x_{(2)}))}(x_{(1)}e_i)\tens x_{(2)}+x_{(1)}\tens (x_{(2)}e_i).$$
Hence, we have
\begin{align*}
&\lan u\tens v,\; \Delta_\n(xe_i)\ra
=q^{-(\alpha_i,\wt(x_{(2)}))}u(x_{(2)})v(x_{(1)}e_i)+u(x_{(2)}e_i)v(x_{(1)})\\
&\hspace{5ex}=q^{(\alpha_i,\wt(u))}u(x_{(2)})\cdot (e_iv)(x_{(1)})+(e_iu)(x_{(2)})
\cdot v(x_{(1)})\\
&\hspace{5ex}=\lan q^{(\alpha_i,\wt(u))}u\tens(e_iv)+(e_iu)\tens v,\,\Delta_\n x\ra.
\end{align*}
The second  identity follows in a similar way.
\end{proof}

We define the map $\iota \cl  U_q^-(\g) \to A_q(\n)$ by
$$  \langle \iota(u),x \rangle = (u,\vph(x)) \quad \text{ for any } u \in U_q^-(\g) \text{ and } x \in U_q^+(\g).$$
 Since $(\ ,\ )$ is a non-degenerate bilinear form on $\Um$,
$\iota$ is injective.
The relation $$\langle \iota(e_i'u), x \rangle = (e_i'u,\vph(x)) =
(u,f_i \vph(x)) =(u, \vph(xe_i))=\lan \iota(u),x e_i \ran = \langle e_i \iota(u),x
\rangle,$$ implies that
\eqn \iota(e_i'u) = e_i \iota(u).
\eneqn

\begin{lemma} $\iota$ is an algebra isomorphism.
\end{lemma}

\begin{proof}
The map $\iota$ is an algebra homomorphism because $e_i'$ and $e_i$
both satisfy the same $q$-boson relation.
\end{proof}

Hence, the algebra $A_q(\n)$ has an upper crystal basis $(L^\up(A_q(\n)),B(A_q(\n)))$ such that
$B(A_q(\n)) \simeq B(\infty)$. Furthermore, $A_q(\n)$ has an upper global basis $$\B^\up(A_q(\n))=\{ G^\up(b) \}_{ \ b \in B(A_q(\n))}$$
induced by the balanced triple
$\bl \Q\tens\An_\A,\,L^\up(A_q(\n)),\,\overline{L^\up(A_q(\n))}\br$
 (see \eqref{eq:Bup}).

There exists an injective map
\eqn &&\oi_\lambda\cl  B(\lambda)\to B(\infty)
\eneqn
induced by the $\Up$-linear homomorphism
$\iota_\la\cl V(\lambda)\to A_q(\n)$ given by
$$v\longmapsto \big ( \Up \ni a\mapsto (av,u_\lambda)\big).$$
 The map $\oi_\lambda$   commutes with $\te_i$.
We have
$$
G^\low_\lambda(b)=G^\low(\oi_\lambda(b))u_\lambda
 \quad\text{and}\quad
\iota_\la G_\la^\up(b)=G^\up(\oi_\la(b))\quad\text{for any $b\in B(\lambda)$.}$$

\begin{remark}
Note that the multiplication on $\An$ given in \cite{GLS} is different from ours.
Indeed, by denoting
the product of $\psi$ and $\phi$  in \cite[\S4.2]{GLS} by $\psi \cdot \phi$,
 for $x \in U_q^+(\g)$,  we have
\eqn
(\psi \cdot \phi) (x) = \psi(x^{(1)}) \phi(x^{(2)}),
\eneqn
where $\Delta_+(x)=x^{(1)}q^{h_{(1)}} \otimes x^{(2)}q^{h_{(2)}}$  for $x^{(1)}, x^{(2)} \in U_q^{+}(\g),  \ h_{(1)}, h_{(2)} \in \wl^\vee$.
By Lemma \ref{lem:Deltan} below, we have
\eqn
(\psi \cdot \phi) (x) &&= \psi\bl q^{(\wt(x_{(1)}),\wt(x_{(2)}))} (x_{(2)}) \br \phi \bl x_{(1)} \br \\
&&=q^{(\wt(x_{(1)},\wt(x_{(2)}))} \psi(x_{(2)}) \phi(x_{(1)})
= q^{(\wt(\psi),\wt(\phi))} (\psi \phi)(x)
\eneqn
for $x \in U_q^+(\g)$, where $\Delta_\n(x)=x_{(1)} \tens x_{(2)}$.
In particular, we have a $\Q(q)$-algebra isomorphism
from $(\An, \cdot)$ to $\An$ given by
\eq \label{eq:isoAqn}
 x \mapsto q^{-\frac{1}{2}(\beta, \beta)} x \quad \ \text{for} \ x \in \An_{\beta}.
\eneq
Note also that
the bar-involution $-$ is a ring anti-isomorphism
between $\An$ and $(\An, \cdot)$.
\end{remark}

\subsection{Modified quantum enveloping algebra}
For the materials in this subsection we refer the reader to
\cite{Lusz92, Kash94}.
We denote by ${\rm Mod}(\g,\wl)$
the category of left $U_q(\g)$-modules with the weight space decomposition. Let $({\rm forget})$ be the functor from
${\rm Mod}(\g,\wl)$ to the category of vector spaces over $\Q(q)$,
forgetting the $U_q(\g)$-module structure.

Let us denote by $\mathscr{R}$ the endomorphism ring of $({\rm forget})$. Note that $\mathscr{R}$ contains $U_q(\g)$. For $\eta \in \wl$, let
$a_\eta \in \mathscr{R}$ denotes the projector $M \to M_\eta$ to the
weight space of weight $\eta$. Then the defining relation of $a_\eta$ (as a left
$\Uqg$-module) is
$$q^h  a_\eta = q^{\langle h,\eta \rangle}a_\eta.$$
We have
\begin{align*}
a_\eta a_\zeta =\delta_{\eta,\zeta} a_\eta, \quad a_\eta P =
P a_{\eta-\xi} \quad \text{for $\xi \in \rl$ and $P \in \Uqg_\xi$.}
\end{align*}
Then $\mathscr{R}$ is isomorphic to $\displaystyle\prod_{\eta \in \wl} \Uqg
a_\eta$. We set
$$ \tUqg \seteq \soplus_{\eta \in \wl} \Uqg a_\eta\subset \mathscr{R}.$$
Then $\tUqg$ is a subalgebra of $\mathscr{R}$.
We call it the {\em modified quantum enveloping algebra}.
Note that any $\U$-module in ${\rm Mod}(\g,\wl)$ has a natural
$\tU$-module structure.

The (anti-)automorphisms $*$, $\vph$ and $\bar{ \ \ }$ of $\Uqg$ extend to the ones of $\tUqg$
by
$$ a_\eta^* = a_{-\eta}, \quad \varphi(a_\eta)=a_\eta, \quad  \overline{a}_\eta=a_\eta.$$

For a dominant integral weight $\lambda \in \wl^+$, let us denote by
$V(\lambda)$
(resp.\ $V(-\lambda)$) the irreducible module with highest (resp.\
lowest) weight $\lambda$ (resp.\ $-\lambda$). Let $u_\lambda$ (resp.\
$u_{-\lambda}$) be the highest (resp.\ lowest) weight vector.

For $\lambda \in \wl^+$, $\mu \in \wl^- \seteq - \wl^+$, we set
$$V(\la,\mu) \seteq V(\lambda) \tensm V(\mu).$$
Then $V(\la,\mu)$ is generated by
 $u_\lambda  \tensm u_\mu$ as a $\Uqg$-module, and the defining relation of
$u_\lambda \tensm u_\mu$ is
\begin{align*}
& q^h(u_\lambda \tensm u_\mu)=q^{\langle h,\lambda+\mu \rangle}(u_\lambda \tensm u_\mu), \\
& e_i^{1-\langle h_i,\mu \rangle}(u_\lambda \tensm
u_\mu)=0, \quad f_i^{1+\langle h_i,\lambda \rangle}(u_\lambda
\tensm u_\mu)=0.
\end{align*}
Let us define the $\Q$-linear automorphism $\bar{ \ \ }$ of $V(\lambda,\mu)$
by
$$ \overline{P(u_\lambda \tensm u_\mu)} = \overline{P}(u_\lambda \tensm u_\mu).$$

We set
\begin{itemize}
\item[{\rm (i)}] $L^\low(\lambda,\mu) \seteq L^\low(\lambda) \tens_{\QA_0} L^\low(\mu)$,
\item[{\rm (ii)}]  $V(\lambda,\mu)_\A \seteq V(\lambda)_\A \tens_\A V(\mu)_\A$,
\item[{\rm (iii)}] $B(\lambda,\mu) \seteq B(\lambda) \tens B(\mu)$.
\end{itemize}

\begin{prop}[\cite{Lusz92}] $(L^\low(\lambda,\mu),B(\lambda,\mu))$ is a lower crystal basis of $V(\la,\mu)$. Furthermore,
$\bl\Q\tens V(\lambda,\mu)_\A,\,L^\low(\lambda,\mu),\,\ol{L^\low(\lambda,\mu)}\br$ is balanced, and
there exists a lower global basis $\B^\low(V(\la,\mu))$ obtained from the lower crystal basis  $(L^\low(\lambda,\mu),B(\lambda,\mu))$.
\end{prop}

\begin{theorem} [\cite{Lusz92}]\label{th:tU}
The algebra $\tUqg$ has a lower crystal basis $(L^\low(\tUqg),B(\tUqg))$ satisfying the following properties:
\bnum
\item $L^\low(\tUqg) = \soplus_{\lambda \in \wl} L^\low(\tUqg a_\lambda)$ and $B(\tUqg) = \bigsqcup_{\lambda \in \wl} B(\tUqg a_\lambda)$, where
\begin{itemize}
\item $L^\low(\tUqg a_\lambda) = L^\low(\tUqg) \cap \Uqg a_\lambda$ and
\item $B(\tUqg a_\lambda) = B(\tUqg) \cap \big( L^\low(\tUqg a_\lambda)/qL^\low(\tUqg a_\lambda) \big)$.
\end{itemize}
\item
Set $\tUqg_\A:=\soplus_{\eta \in \wl} \Uqg_\A  a_\eta$. Then
$\bl\Q\tens\tUqg_\A,\, L^\low(\tUqg),\,\ol{ L^\low(\tUqg)}\br$ is balanced,
and
$\tUqg$ has the lower global basis
$\B^\low(\tUqg)\seteq\{ G^\low(b)\mid b \in B(\tUqg)\}$.
\item For any $\lambda \in \wl^+$ and $\mu \in \wl^-$, let
$$\Psi_{\lambda,\mu}\cl\Uqg a_{\lambda+\mu} \to V(\lambda,\mu)$$
 be the $U_q(\g)$-linear map
$a_{\lambda+\mu} \longmapsto u_\lambda \tens u_\mu$.
Then we have $\Psi_{\lambda,\mu}\bl L(\tUqg a_{\lambda+\mu})\br =
L^\low(\lambda,\mu)$.
\item Let $\overline{\Psi}_{\lambda,\mu}$ be the induced homomorphism
$$L^\low(\tUqg a_{\lambda+\mu})/qL^\low(\tUqg a_{\lambda+\mu})   \longrightarrow  L^\low(\lambda,\mu)/qL^\low(\lambda,\mu) .$$
Then we have
\bnam
\item $\{ b \in B(\tUqg a_{\lambda+\mu}) \ | \ \overline{\Psi}_{\lambda,\mu}b \ne 0\} \overset{\sim}{\longrightarrow} B(\lambda,\mu)$,
\item $\Psi_{\lambda,\mu}\bl G^\low(b)\br =  G^\low(\overline{\Psi}_{\lambda,\mu}(b))$ for any $b \in B(\tUqg a_{\lambda+\mu})$.
\ee
\item $B(\tUqg)$ has a structure of crystal such that
 the injective map induced by {\rm (iv) (a)}
$$ B(\lambda,\mu) \to B(\tUqg a_{\lambda+\mu}) \subset B(\tUqg)$$
is a strict embedding of crystals  for any $\la\in\Pd$ and $\mu\in\Po^-$.
\end{enumerate}
\end{theorem}

For $\lambda \in \wl$, take any $\zeta \in \wl^+$ and $\eta \in \wl^-$ such that
$\lambda = \zeta +\eta$. Then $B(\zeta) \tens B(\eta)$ is
embedded into $B(\tUqg a_\lambda)$.

For $\mu \in \wl$, let $T_\mu=\{ t_\mu \}$ be the crystal with
$$\wt(t_\mu) = \mu, \quad \ve_i(t_\mu) = \vph_i(t_\mu)=-\infty, \quad
\te_i (t_\mu) =\tf_i (t_\mu) =0.$$
Since we have
$$ B(\zeta) \hookrightarrow B(\infty) \tens T_\zeta, \  B(\eta) \hookrightarrow  T_\eta \tens B(-\infty) \text{ and }  T_\zeta \tens T_\eta \simeq T_{\zeta+\eta},$$
$B(\zeta) \tens B(\eta)$ is embedded into the crystal $B(\infty)
\tens T_\lambda \tens B(-\infty)$. Taking $\zeta \to \infty$ and $\eta \to -\infty$, we have
\Lemma [\cite{Kash94}]  For any $\la\in\Po$, we have a canonical crystal isomorphism
$$ B(\tUqg a_{\la}) \simeq B(\infty) \tens T_\lambda \tens B(-\infty).$$
\enlemma
Hence we identify
$$ B(\tUqg)=\bigsqcup_{\la\in\Po}B(\infty) \tens T_\lambda \tens B(-\infty).$$

For $\xi\in\nrtl$ and $\eta\in\prtl$, we shall denote by
\begin{align*}
\Um_{>\xi}\seteq\soplus_{\xi'\in\nrtl\cap(\xi+\prtl)\setminus\{\xi\}}\Um_{\xi'}, \qquad
\Up_{<\eta}\seteq\soplus_{\eta'\in\prtl\cap(\eta+\nrtl)\setminus\{\eta\}}\Up_{\eta'}.
\end{align*}
Then for any $\la\in \Po$, $b_-\in B(\infty)_\xi$ and
$b_+\in B(-\infty)_\eta$, we have
\eq
G^\low(b_-\tens t_\la\tens b_+)
-G^\low(b_-) G^\low(b_+)a_\la\in
\Um_{>\xi}\Up_{<\eta}a_\la
\label{eq:tUc}
\eneq
(\cite[(3.1.1)]{Kash94}). In particular, we have
\eqn
G^\low(b_\infty \tens t_\la \tens b_+ )  = G^\low(b_+)a_\la \quad \text{and} \ G^\low(b_- \tens t_\la \tens b_{-\infty}) =G^\low(b_-)a_\la.
\eneqn

\begin{theorem}  [\cite{Kash94}] \label{Thm: Kas94} \hfill
\bnum
\item $L^\low(\tUqg)$ is invariant under the anti-automorphisms $*$ and $\vph$.
\item $B(\tUqg)^* = \vph(B(\tUqg)) = B(\tUqg)$.
\item $\bl G^\low(b)\br^*=G^\low(b^*)$ and $\vph(G^\low(b))=G^\low(\vph(b))$ for $b \in B(\tUqg)$.
\end{enumerate}
\end{theorem}

\begin{corollary} [\cite{Kash94}] \label{cor: inv cry}
For $b_1 \in B(\infty)$, $b_2 \in B(-\infty)$, we have
\begin{enumerate}
\item[{\rm (1)}] $(b_1 \tens t_\mu \tens b_2)^* = b_1^* \tens t_{-\mu-\wt(b_1)-\wt(b_2)} \tens b_2^*$.
\item[{\rm (2)}] $\vph(b_1 \tens t_\mu \tens b_2) = \vph(b_2) \tens t_{\mu+\wt(b_1)+\wt(b_2)} \tens \vph(b_1)$.
\end{enumerate}
\end{corollary}

We define, for $b \in B$ with $B=B(\tUqg), B(\infty)$ or $B(-\infty)$,
\begin{align*}
\ve_i^*(b) = \ve_i(b^*),  \ \vph_i^*(b) = \vph_i(b^*), \  \wt^*(b)=\wt(b^*), \  \te_i^*(b)
= \te_i(b^*)^* \text{ and } \tf_i^*(b) = \tf_i(b^*)^*.
\end{align*}

This defines another crystal structure on $\tUqg$: For $b_1 \in
B(\infty)$ and $b_2 \in B(-\infty)$ and $\eta \in \wl$, we have
\begin{align*}
 \ve^*_i(b_1 \tens t_\eta \tens b_2) & = \max( \ve^*_i(b_1), \vph^*_i(b_2)+\langle h_i, \eta \rangle ), \\
 \vph_i^*(b_1 \tens t_\eta \tens b_2) & = \max( \ve^*_i(b_1)-\langle h_i, \eta \rangle , \vph^*_i(b_2)), \\
&= \ve_i^*(b_1 \tens t_\eta \tens b_2) + \langle  h_i,\wt^*(b_1 \tens t_\eta \tens b_2)\rangle, \\
 \wt^*(b_1 \tens t_\eta \tens b_2) & = -\eta, \\
 \te^*_i(b_1 \tens t_\eta \tens b_2) & =
 \begin{cases} (\te^*_ib_1) \tens t_{\eta-\alpha_i} \tens b_2 & \text{ if } \ve^*_i(b_1) \ge \vph^*_i(b_2)+\langle h_i,\eta \rangle, \\
               b_1 \tens t_{\eta-\alpha_i} \tens( \te^*_ib_2) & \text{ if } \ve^*_i(b_1) < \vph^*_i(b_2)+\langle h_i,\eta \rangle, \end{cases} \\
 \tf^*_i(b_1 \tens t_\eta \tens b_2) & =
 \begin{cases} (\tf^*_ib_1) \tens t_{\eta+\alpha_i} \tens b_2 & \text{ if } \ve^*_i(b_1) > \vph^*_i(b_2)+\langle h_i,\eta \rangle, \\
 b_1 \tens t_{\eta+\alpha_i} \tens ( \tf^*_ib_2) & \text{ if } \ve^*_i(b_1) \le \vph^*_i(b_2)+\langle h_i,\eta \rangle. \end{cases} \\
\end{align*}

In particular, we have
$$\te_i\circ \vphi=\vphi\circ\tf^*_i\quad\text{and}\quad
\tf_i\circ \vphi=\vphi\circ\te^*_i\quad\text{for every $i\in I$.}$$

\subsection{Relationship of $\Ag$ and $\tU$}

There exists a canonical pairing $A_q(\g) \times \tUqg \to \Q(q)$ by
\begin{equation*} \label{eq: coupling}
\laa \psi, x a_\mu \raa = \delta_{\mu,\wt_\li(\psi)} \psi(x)
\quad\text{ for any  $\psi \in A_q(\g)$, $x\in\U$ and $\mu\in\Po$.}
\end{equation*}

\begin{theorem}[{\cite{Kash94}}] \label{thm: bicrystal embed}
There exists a bi-crystal embedding
\begin{equation*} \label{eq: iota}
\oi_\g\cl  B(A_q(\g)) \To B(\tUqg)
\end{equation*}
which satisfies:
$$\laa G^\up(b),\vph(G^\low(b')) \raa =  \delta_{\,\oi_\g(b),b'}$$
for any $b \in B(A_q(\g))$ and $b' \in B(\tUqg)$.
\end{theorem}

\subsection{Relationship of $\Ag$ and $\An$}

\begin{definition} \label{def: p_n}
 Let $p_\n\cl A_q(\g) \to A_q(\n)$ be the homomorphism induced by $U_q^+(\g) \to U_q(\g)$:
$$\langle p_\n(\psi),x \rangle = \psi(x) \quad \text{for \ any
$x \in U_q^+(\g)$.}$$
\end{definition}

Then we have
$$ \wt(p_\n(\psi))=\wt_\li(\psi)-\wt_\ri(\psi).$$

It is obvious that $p_\n$ sends all $\Phi(u_{w\lambda}\tens u_{w\lambda}^\ri)$
($\lambda\in\wl^+$ and $w\in W$) to $1$. Note that
$\oi_\g(u_{w\lambda}\tens u_{w\lambda}^\ri)=b_\infty\tens t_{w\lambda}\tens b_{-\infty}\in B(\tUqg)$.

\begin{prop} \label{prop: nonzero elt}
For $b \in B(A_q(\g))$,
set $\bg(b)=b_1 \tens t_\zeta \tens b_2
\in B(\infty) \tens T_\zeta \tens B(-\infty)\subset B(\tU)$ $(\zeta \in \wl)$.
Then we have
$$ p_\n(G^\up(b)) = \delta_{b_2,\,b_{-\infty}}G^\up(b_1).$$
\end{prop}

\begin{proof}
Set $\eta \seteq \wt(b_1)+\zeta+\wt(b_2)=\wt_\lt(b)$. Then for any $\tilde{b} \in B(\infty)$, we
have
\begin{align*}
\big\lan p_\n(G^\up(b)), \vph(G^\low(\tilde{b})) \big\ra &= \big\laa G^\up(b), G^\low(\vph(\tilde{b}))a_{\eta} \big\raa
\\
& \hspace{-20ex} = \big\laa G^\up(b), G^\low(b_\infty \tens t_\eta \tens \vph(\tilde{b})) \big\raa
 = \big\laa G^\up(b), \vph(G^\low(\tilde{b} \tens t_{\eta-\wt(\tilde{b})} \tens b_{-\infty})) \big\raa \\
& \hspace{-20ex} = \delta(\bg(b) = \tilde{b} \tens t_{\eta-\wt(\tilde{b}) } \tens b_{-\infty} ) =\delta(b_2=b_{-\infty},b_1 = \tilde{b}).
\qedhere
\end{align*}
\end{proof}

Hence the map $p_\n$ sends
the upper global basis of $A_q(\g)$ to the upper global basis of
$A_q(\n)$ or zero. Thus we have a map
$$ \overline{p}_\n \cl  B(A_q(\g)) \to B(A_q(\n)) \bigsqcup \{ 0 \}.$$

Although the map $\pn$ is not an algebra homomorphism,
it preserves the multiplications up to a power of $q$, as we will see below.

\begin{lemma}  \label{lem:Deltan}
 For $x \in U_q^+(\g)$, if $\Delta_\n(x)=x_{(1)} \tens x_{(2)}$, then
\begin{align} \label{eq: Delta_n comp}
\comp(x)=q^{\wt(x_{(1)})} x_{(2)} \tens x_{(1)}.
\end{align}
\end{lemma}

\begin{proof}
Assume that \eqref{eq: Delta_n comp} holds for $x \in U_q^+(\g)$.
Note that
$$ \Delta_\n(e_ix) = (e_i \tens 1+1 \tens e_i)(x_{(1)} \tens x_{(2)}) =
e_ix_{(1)} \tens x_{(2)} + q^{-(\alpha_i,\wt(x_{(1)}))}x_{(1)} \tens (e_ix_{(2)}).$$
On the other hand, we have
\begin{align*}
\comp(e_ix)&=(e_i  \tens 1+ q^{\alpha_i} \tens e_i)(q^{\wt(x_{(1)})} x_{(2)} \tens x_{(1)}) \allowdisplaybreaks \\
&=(e_i q^{\wt(x_1)}) x_{(2)} \tens x_{(1)} +
( q^{\alpha_i+\wt(x_{(1)})} x_{(2)}) \tens (e_ix_{(1)}) \allowdisplaybreaks \\
& = q^{-(\alpha_i,\wt(x_{(1)}))}(q^{\wt(x_{(1)})}e_ix_{(2)}) \tens
x_{(1)}+ (q^{\wt(e_ix_{(1)})}x_{(2)}) \tens (e_ix_{(1)}).
\end{align*}
Hence \eqref{eq: Delta_n comp} holds for $e_ix$.
\end{proof}

\begin{prop}\label{prop: p_n}
For $\psi,\theta \in A_q(\g)$, we have
$$ p_\n(\psi\theta)= q^{(\wt_\ri(\psi),\wt_\ri(\theta)-\wt_\li(\theta))}p_\n(\psi)p_\n(\theta).$$
\end{prop}

\begin{proof}
For $x \in U_q^+(\g)$, set $\Delta_\n(x)=x_{(1)} \tens x_{(2)}$. Then, we
have
\begin{align*}
\lan p_\n(\psi\theta),x \ra & = \lan \psi\theta,x  \ra
= \lan \psi \tens \theta, q^{\wt(x_{(1)})} x_{(2)} \tens x_{(1)} \ra
= \lan \psi , q^{\wt(x_{(1)})} x_{(2)} \ra \lan \theta , x_{(1)} \ra \\
&= q^{(\wt_\ri(\psi),\wt(x_{(1)}))} \lan \psi ,  x_{(2)} \ra \lan \theta , x_{(1)} \ra
= q^{(\wt_\ri(\psi),\wt(x_{(1)}))} \lan p_\n(\psi) ,  x_{(2)} \ra
\lan p_\n(\theta) , x_{(1)} \ra  \\
&\underset{\mathrm{(a)}}{=} q^{(\wt_\ri(\psi),\wt_\ri(\theta)-\wt_\li(\theta))} \lan  p_\n(\psi)\tens p_\n(\theta),\Delta_\n(x) \ra \\
&= q^{(\wt_\ri(\psi),\wt_\ri(\theta)-\wt_\li(\theta))} \lan
p_\n(\psi)p_\n(\theta),x \ra .
\end{align*}
Here, we used $\wt(x_{(1)})=-\wt\bl p_\n (\theta )\br$  in (a).
\end{proof}

\subsection{Global basis of $\tUqg$ and tensor products of $\Uqg$-modules in $\Oint(\g)$}

Let $V$ be an integrable $\Uqg$-module with a bar-involution $-$, that is, there is a $\Q$-linear automorphism $-$ satisfying
$\ol{P v} = \ol{P} \ol{v}$ for all $P \in \Uqg$ and for all $v \in V$.
Then, for any $\la\in\Pd$, there exists a unique bar-involution  $-$ on $V(\la)\tensm V$
satisfying
$$\text{$\ol{(u_\la\tensm v)}=u_\la\tensm \ol{v}$ for any $v\in V$.}$$
Indeed, there exists
$\Xi\in \one+\prod_{\beta\in\prtl\setminus\{0\}}\Up_\beta \tens\Um_{-\beta}$, which defines a  bar-involution by setting
 $$\ol{u\tensm v}:=\Xi\bl\ol{u}\tensm\ol{v}\br$$
(see, \cite[Chapter 4]{Lus93}).
Assume that $V$ has a lower crystal basis $\bl L(V),B(V)\br$ and an $\QA$-form
$V_\QA$  such that
$\bl  V_\QA, L(V),\ol{L(V)}\br$ is balanced.
Then we have
\Prop The triple
$\Bigl( V(\la)_\QA\tens_\QA V_\QA, L(\la)\tens_{\QA_0} L(V), \ol{ L(\la)\tens_{\QA_0} L(V)}
\Bigr)$
in $V(\la)\tensm V$
is balanced.
\enprop
Note that $u_\la\tensm G^\low(b)$ is a lower global basis for
any $b\in B(V)$, i.e., $G^\low(u_\la\tens b)=u_\la\tensm G^\low(b)$.

In particular, it applies to $V(\la)\tensm V(\mu)$.
Moreover, we have

\begin{prop} \label{pro:gltens}
Let $\lambda,\mu \in \wl^+$ and $w \in W$. Then for any $b \in
B(\tUqg a_{\lambda+w\mu})$, $G^\low(b)(u_\lambda \tensm
u_{w\mu})$ vanishes or is a member of the lower global basis of
$V(\lambda) \tensm V(\mu)$.
\end{prop}

Hence we have a crystal morphism
\begin{align} \label{eq: pi la wmu}
\pi_{\lambda,w\mu}\cl  B(\tUqg a_{\lambda+w\mu}) \to B(\lambda)
\tens B(\mu)
\end{align}
by $G^\low(b)(u_\lambda \tensm
u_{w\mu})=G^\low(\pi_{\lambda,w\mu}(b))$.

Similarly, we have a bar-involution $-$ on
$V\tensp V(\la)$ such that
$$\text{$\ol{(v\tensp u_\la)}=\ol{v}\tensp u_\la$ for any $v\in V$.}$$
Hence if
$V$ has an upper crystal basis $(L^\up(V),B(V))$ and an $\QA$-form $V_\QA$
such that $\bl  V_\QA, L^\up(V), \ol{L^\up(V)} \br$ is balanced, then
$V\tensp V(\la)$ has an upper global basis. Note that
$G^\up(b)\tensp u_\la$ is a member of the upper global basis for
$b\in B(V)$.

In particular
for $\la,\mu\in \Po$,
$V(\la)\tensm V(\mu)$ has a lower global basis and $V(\la)\tensp V(\mu)$
has an upper global basis.

The bilinear form
\eqn (\scbul,\scbul)\cl \Bigl(V(\la)\tensm V(\mu)\Bigr)
\times  \Bigl(V(\la)\tensp V(\mu)\Bigr)\to \cor  \label{eq:bilinear}  \eneqn
defined by
$(u\tensm v,u'\tensp v')=(u,u')(v,v')$, $u,u'\in V(\la)$,
 $v,v'\in V(\mu)$ satisfies
\eqn\text{$(ax,y)=(x,\vphi(a)y)$ for any $x\in V(\la)\tensm V(\mu)$,
$y\in  V(\la)\tensp V(\mu)$, $a\in\U$.}\eneqn
With respect to this bilinear form,
 the lower global basis of $V(\la)\tensm V(\mu)$
and the upper global basis of $V(\la)\tensp V(\mu)$
are the dual  bases of each other.

\section{Quantum minors and T-systems}\label{sec:Quantum minors and T-systems}

\subsection{Quantum minors}
Using the isomorphism $\Phi$ in \eqref{eq: Phi}, for each
$\lambda \in \wl^+$ and $\mu, \zeta \in W \lambda$, we define the elements
$$ \Delta(\mu,\zeta)\seteq \Phi(u_\mu\tens u_\zeta^\ri) \in A_q(\g)$$
and
$$\D(\mu,\zeta)\seteq p_\n(\Delta(\mu,\zeta)) \in A_q(\n).$$
The element $\Delta(\mu,\zeta)$ is called a {\em {\rm (}generalized{\rm )}  quantum minor}
and $\D(\mu,\zeta)$ is called {\em a unipotent quantum minor}.

\begin{lemma} \label{lem: global element}
$\Delta(\mu,\zeta)$ is a member of the upper global basis of $A_q(\g)$. Moreover,
$\D(\mu,\zeta)$
is either a member of the upper global basis of $ A_q(\n)$ or zero.
\end{lemma}
\begin{proof}   Our assertions follow from Proposition~\ref{prop:Agglobal} and
Proposition~\ref{prop: nonzero elt}.
\end{proof}
\Lemma[{\cite[{(9.13)}]{BZ05}}]
 For $u,v \in W$ and $\lambda,\mu \in \wl^+$, we have
$$ \Delta(u\lambda,v\lambda)\Delta(u\mu,v\mu) =
\Delta\bl u (\lambda+\mu),v(\lambda+\mu)\br.$$
\enlemma

By Proposition \ref{prop: p_n}, we have the
following corollary:

\begin{corollary}\label{cor:Duv}
For $u,v \in W$ and $\lambda,\mu \in \wl^+$, we have
$$ \D(u\lambda,v\lambda)\D(u\mu,v\mu) =
q^{-(v\lambda,v\mu-u\mu)}\D\bl u (\lambda+\mu),v(\lambda+\mu)\br.$$
\end{corollary}

Note that
$$ \D(\mu,\mu)=1 \quad \text{ for $\mu \in W\lambda$ }.$$

Then $\D(\mu,\zeta) \ne 0$ if and only if $\mu \preceq \zeta$. Recall
that for $\mu,\zeta$ in the same $W$-orbit, we say that $\mu \preceq
\zeta$ if there exists a sequence $\{ \beta_k \}_{1 \le k \le l}$ of positive real roots such that, defining $\lambda_0=\zeta$,
$\lambda_k=s_{\beta_k}\lambda_{k-1}$ $(1 \le k \le l)$, we have
$(\beta_k,\lambda_{k-1}) \ge 0$ and $\lambda_l=\mu$.

More precisely, we have

\begin{lemma} \label{lem: D-property}
Let $\lambda \in \wl^+$ and $\mu,\zeta \in W \lambda$. Then the
following conditions are equivalent:
\begin{enumerate}[(i)]
\item $\D(\mu,\zeta)$ is an element of upper global basis of $A_q(\n)$,
\item $\D(\mu,\zeta) \ne 0$,
\item $u_\mu \in U_q^-(\g)u_\zeta$,
\item $u_\zeta \in U_q^+(\g)u_\mu$,
\item $\mu \preceq \zeta$,
\item for any $w \in W$ such that $\mu = w \lambda$, there exists $u \le w$ \ro in the Bruhat order\rf\ such that $\zeta=u \lambda$,
\item there exist $u,v \in W$ such that $\mu = w \lambda$, $\zeta=u\lambda$ and $ u \le w$.
\end{enumerate}
\end{lemma}
\Proof
(i) and (ii) are equivalent by  Lemma \ref{lem: global element}.
The equivalence of (ii), (iii) and  (iv) is obvious.
The equivalence of (v), (vi), (vii) is well-known.
The equivalence of (iv) and (vi) is proved in \cite{Kas93a}.
\QED

For any $u\in\An\setminus\{0\}$ and $i\in I$,
we set
\begin{align*}
\eps_i(u)&\seteq\max\set{n\in\Z_{\ge0}}{e_i^nu\not=0},\\
\eps^*_i(u)&\seteq\max\set{n\in\Z_{\ge0}}{e_i^{*\,n}u\not=0}.
\end{align*}
Then for any $b\in B(\An)$, we have
$$\text{$\eps_i(G^\up(b))=\eps_i(b)$ and $\eps^*_i(G^\up(b))=\eps^*_i(b)$.}$$

\Lemma\label{lem: weyl left right}
Let $\lambda \in \wl^+$, $\mu,\zeta \in W \lambda$
such that $\mu \preceq \zeta$ and $i \in I$.
\bnum
\item If $n \seteq \lan h_i,\mu \ra \ge 0$, then
$$ \ve_i(\D(\mu,\zeta))=0 \ \ \text{ and } \ \ e_i^{(n)}\D(s_i\mu,\zeta)=\D(\mu,\zeta).$$
\item If $\lan h_i,\mu \ra \le 0$ and $s_i\mu \preceq \zeta$, then $\ve_i(\D(\mu,\zeta))=-\lan h_i,\mu \ra$.
\item If $m \seteq -\lan h_i,\zeta \ra  \ge  0$, then
$$ \ve^*_i(\D(\mu,\zeta))=0 \ \ \text{ and } \ \ {e_i^*}^{(m)} \D(\mu,s_i\zeta)=
\D(\mu,\zeta).$$
\item If $\lan h_i,\zeta \ra \ge 0$ and $\mu \preceq s_i\zeta$, then $\ve^*_i(\D(\mu,\zeta))=\lan h_i,\zeta \ra$.
\end{enumerate}
\enlemma
\Proof
We have
$\eps_i\bl\Delta(\mu,\zeta)\br=\max(-\lan h_i,\mu\ran, 0)$
and $\eps_i^*\bl\Delta(\mu,\zeta)\br=\max(\lan h_i,\zeta\ran, 0)$.
Moreover, $\pn$ commutes with $e_i^{(n)}$ and ${e_i^*}^{(n)}$.

Let us show (ii). Set $\ell=-\lan h_i,\mu \ra$.
Then we have
$e_i^{\ell+1}\Delta(\mu,\zeta)=0$, which implies
$e_i^{\ell+1}\D(\mu,\zeta)=0$. Hence $\eps_i(\D(\mu,\zeta))\le \ell$.
We have
$$e_i^{(\ell)}\Delta(\mu,\zeta)=\Delta(s_i\mu,\zeta).$$
Hence we have $e_i^{(\ell)}\D(\mu,\zeta)=\D(s_i\mu,\zeta)$.
By the assumption  $s_i\mu \preceq \zeta$, $\D(s_i\mu,\zeta)$ does not vanish.
Hence we have $\eps_i(\D(\mu,\zeta))\ge \ell$.

\smallskip
The other statements can be proved similarly.
\QED

\begin{prop}[{\cite[(10.2)]{BZ05}}] \label{prop: BZ form} Let $\lambda,\mu \in \wl^+$ and $s,t,s',t' \in W$ such that $\ell(s's)=\ell(s')+\ell(s)$ and
$\ell(t't)=\ell(t')+\ell(t)$. Then we have
\bnum
\item
$ \Delta(s's\lambda,t'\lambda)\Delta(s'\mu,t't\mu)=q^{(s\lambda,\mu)-(\lambda,t\mu)}\Delta(s'\mu,t't\mu)\Delta(s's\lambda,t'\lambda)$.
\item
If we assume further that $s's\lambda \preceq t'\lambda$ and
$s'\mu \preceq t't\mu$, then we have
\begin{align}
& \D(s's\lambda,t'\lambda)\D(s'\mu,t't\mu) = q^{(s's\lambda+t'\lambda,\;s'\mu-t't\mu)} \D(s'\mu,t't\mu)\D(s's\lambda,t'\lambda),
\label{eq: quantum commuting factor} \end{align}
or equivalently
\begin{align}
& q^{(t'\lambda,\;t't\mu-s'\mu)}\D(s's\lambda,t'\lambda)\D(s'\mu,t't\mu)
=q^{(s'\mu-t't\mu,\;s's\lambda)}\D(s'\mu,t't\mu)\D(s's\lambda,t'\lambda)
\label{eq: quantum commuting factor 2}
\end{align}
\ee
\end{prop}
Note that (ii) follows from by Proposition \ref{prop: p_n} and (i).
Note also that the both sides of \eqref{eq: quantum commuting factor 2}
are bar-invariant, and hence they are members of the upper global basis
as seen by Corollary \ref{cor:qcomm}.
%\cite[Corollary 3.5]{KKKO15}.

\Prop\label{prop:Deprod}
For $\la,\mu\in\Pd$ and $s,t\in W$,
set
$\bg\bl u_{s\la}\tens (u_\la)^\rt\br=b_-\tens t_\la\tens b_{-\infty}$
and $\bg\bl u_{\mu}\tens (u_{t\mu})^\rt\br=b_{\infty}\tens t_{t\mu}\tens b_+$
with $b_{\mp}\in B(\pm\infty)$.
Then we have
$$\De{s\la,\la}\De{\mu,t\mu}=G^\up\bl\bg^{-1} (b_-\tens t_{\la+t\mu}\tens b_+)\br.$$
\enprop
\Proof

Recall that there is a
  pairing $(\scbul,\scbul):
\bl V(\la)\tensm V(\mu)\br\times \bl V(\la)\tensp V(\mu)\br\to \Q(q)$
 defined by
$(u\tensm  v,u'\tensp v')=(u,u')(v,v')$.
It satisfies
$$\bl P(u\tensm v), u'\tensp v'\br=\bl u\tensm v, \vphi(P)(u'\tensp v')\br
 \quad \text{for any} \ P \in \Uqg.
$$
For $u,u'\in V(\la)$ and $v,v'\in V(\mu)$,
we have
\begin{align*}
\lan \Phi(u\tens u'^\rt)\Phi(v\tens v'^\rt), P\ran&=
\bl u'\tensm v', P(u\tensp v)\br\\
&=\bl\vphi(P)(u'\tensm v'),u\tensp v\br.
\end{align*}
Hence for $P\in\U$, we have
\begin{align*}\lan \De{s\la,\la}\De{\mu,t\mu}, Pa_{\zeta}\ran
&=\delta(\zeta=s\la+\mu)
\bl\vphi(P)(u_{\la}\tensm u_{t\mu}),u_{s\la}\tensp u_{\mu}\br.
\end{align*}
If $Pa_\zeta=G^\low(\vphi(b))$ for $b\in B(\tU)$,
then
we have
$$\lan \De{s\la,\la}\De{\mu,t\mu},\vphi(G^\low(b))\ran
=\delta(\zeta=s\la+\mu)
\bl G^\low(b)(u_{\la}\tensm u_{t\mu}),u_{s\la}\tensp u_{\mu}\br.$$
The element
$G^\low(b)(u_{\la}\tensm u_{t\mu})$
vanishes or is a global basis of $V(\la)\tensm V(\mu)$
by Proposition~\ref{pro:gltens}.
Since $u_{s\la}\tensp u_{\mu}$ is a member of the upper global basis of
$V(\la)\tensp V(\mu)$,
 we have
$$\lan \De{s\la,\la}\De{\mu,t\mu},\vphi(G^\low(b))\ran
=\delta(\zeta=s\la+\mu)
\delta\bl \pi_{\la,t\mu}(b)=u_{s\la}\tens u_{\mu}\br.$$
Here $\pi_{\la,t\mu}\cl B(\tU a_{\la+t\mu})\to B(\la)\tens B(\mu)$
is the crystal morphism given in \eqref{eq: pi la wmu}.

Hence we obtain
$$\De{s\la,\la}\De{\mu,t\mu}=G^\up( \bg^{-1}( b)),$$
where $b\in B(\tU)$ is a unique element such that
$\bl G^\low(b)(u_{\la}\tensm u_{s\mu}), u_{s\la}\tensp u_{\mu}\br=1$.

On the other hand, we have
$G^\low(b_+)u_{t\mu}=u_{\mu}$ and
$G^{\low}(b_-)u_\la=u_{s\la}$.
The last equality implies $\vphi(G^{\low}(b_-))u_{s\la}=u_\la$ because
$(\vphi(G^{\low}(b_-))u_{s\la},u_\la)=(u_{s\la},G^{\low}(b_-)u_\la)=(u_{s\la},u_{s\la})=1$.
As seen in \eqref{eq:tUc}, we have
$$G^{\low}(b_-)G^\low(b_+)a_{\la+t\mu}-G^\low(b_-\tens t_{\la+t\mu}\tens b_+)
\in \Um_{>s\la-\la}\Up_{<\mu-t\mu}a_{\la+t\mu}.$$
Hence we obtain
\eqn
&&\bl G^\low(b_-\tens t_{\la+t\mu}\tens b_+)(u_{\la}\tensm u_{t\mu}),\;
u_{s\la}\tensp u_{\mu}\br\allowdisplaybreaks[1]\\
&&\hs{10ex}
=\bl G^{\low}(b_-)G^\low(b_+)(u_{\la}\tensm u_{t\mu}),\;
u_{s\la}\tensp u_{\mu}\br\allowdisplaybreaks[1]\\
&&\hs{20ex}
=\bl G^\low(b_+)(u_{\la}\tensm u_{t\mu}),\;
\vphi(G^\low(b_-))(u_{s\la}\tensp u_{\mu})\br
=1.
\eneqn
In the last equality, we used
$G^\low(b_+)(u_{\la}\tensm u_{t\mu})
=u_{\la}\tensm (G^\low(b_+)u_{t\mu})=u_{\la}\tensm u_\mu$
and
$\vphi(G^\low(b_-))(u_{s\la}\tensp u_{\mu})
=\bl\vphi(G^\low(b_-))u_{s\la})\tensp u_{\mu}=u_{\la}\tensp u_\mu$.

Hence we conclude that $b=b_-\tens t_{\la+t\mu}\tens b_+$.
\QED

Let
\eqn
&&\iota_{\la,\mu}\cl V(\la+\mu)\monoto V(\la)\tens V(\mu)\eneqn
be the canonical embedding and
\eqn
&&\oi_{\la,\mu}\cl B(\la+\mu)\monoto B(\la)\tens B(\mu)\eneqn
the induced crystal embedding.

\begin{lemma} \label{lem:oi la mu}
For $\lambda,\mu \in \wl^+$ and $x,y \in W$ such that $x \ge y$, we have
$$ u_{x\lambda} \otimes u_{y \mu} \in \overline{\iota}_{\lambda,\mu} (B(\lambda+\mu)) \subset B(\lambda) \otimes B(\mu).$$
\end{lemma}

\begin{proof}
Let us show by induction on $\ell(x)$, the length of $x$ in $W$. We may assume that $x \ne 1$. Then there exists $i \in I$ such that
$s_i x < x$. If $s_iy<y$, then $s_ix \ge s_iy$ and $\tilde{e}_i^{\max}(u_{x\lambda} \otimes u_{y\mu})= u_{s_ix\lambda} \otimes u_{s_iy\mu}$.
If $s_iy>y$, then $s_ix\ge y$ and
$\tilde{e}_i^{\max}(u_{x\lambda} \otimes u_{y\mu})= u_{s_ix\lambda} \otimes u_{y\mu}$. In
both cases,
$u_{x\lambda} \otimes u_{y \mu}$  is connected with an element of $\overline{\iota}_{\lambda,\mu}(B(\lambda+\mu))$.
\end{proof}

\Lemma \label{lem:Deprod}
For $\la, \mu\in\Pd$ and $w\in W$, we have
$$\De{w\la,\la}\De{\mu,\mu}=
G^\up\bl\oi_{\la,\mu}^{-1}(u_{w\la}\tens u_\mu)\tens u_{\la+\mu}{}^\rt\br.$$
\enlemma
\Proof We have
\eqn
&&\bg(u_{w\la}\tens u_\la^\rt\br=b_{w\la}\tens t_{\la}\tens b_{-\infty},\\
&&\bg(u_{\mu}\tens u_\mu^\rt\br=b_{\infty}\tens t_{\mu}\tens b_{-\infty},
\eneqn
where $b_{w\la}\seteq\oi_\la(u_{w\la})$.
Hence Proposition~\ref{prop:Deprod}
implies that
$$\De{w\la,\la}\De{\mu,\mu}=G^\up\bl\bg^{-1}(b_{w\la}\tens t_{\la+\mu}\tens b_{-\infty})\br.$$
Then, $\bg\bl\oi_{\la,\mu}^{-1}(u_{w\la}\tens u_\mu)\tens u_{\la+\mu}{}^\rt\br
=b_{w\la}\tens t_{\la+\mu}\tens b_{-\infty}$
 gives the desired result.
\QED

\subsection{T-system} \label{subsubsec: T-system} In this subsection, we  recall the {\em $T$-system} among the (unipotent) quantum minors for later use
(see \cite{KNS11} for $T$-system).

\begin{prop} [{\cite[Proposition 3.2]{GLS}}] \label{prop: the ses}
Assume that the Kac-Moody algebra $\g$ is of symmetric type. Assume that $u,v \in W$ and $i \in I$ satisfy $u < us_i$ and $v <
vs_i$. Then
\begin{align*}
& \Delta(us_i\varpi_i,vs_i\varpi_i)\Delta(u\varpi_i,v\varpi_i) =
q^{-1}\Delta(u  s_i \varpi_i,v\varpi_i)\Delta( u\varpi_i, v s_i\varpi_i)
 +\Delta(u\lambda,v\lambda),   \\
%+ \prod_{j \ne i}\Delta(u\varpi_j,v\varpi_j)^{-a_{i,j}},
& \Delta(u\varpi_i,v\varpi_i)\Delta(us_i\varpi_i,vs_i\varpi_i) =
q\Delta( u\varpi_i\ ,vs_i\varpi_i)\Delta(u  s_i  \varpi_i,v\varpi_i) +
 \Delta(u\lambda,v\lambda)
%\prod_{j \ne i}\Delta(u\varpi_j,v\varpi_j)^{-a_{i,j}},
\end{align*}
and
\begin{align*}
& q^{(vs_i\varpi_i,v\varpi_i-u\varpi_i)}\D(us_i\varpi_i,vs_i\varpi_i)\D(u\varpi_i,v\varpi_i) \allowdisplaybreaks\\
& \hspace{5ex} = q^{-1+(v\varpi_i,vs_i\varpi_i-u\varpi_i)}\D(us_i\varpi_i,v\varpi_i)\D(u\varpi_i,vs_i\varpi_i)  +\D(u\lambda,v\lambda)
\allowdisplaybreaks\\
& \hspace{5ex} = q^{-1+(vs_i\varpi_i,v\varpi_i-us_i\varpi_i)}\D(u\varpi_i,vs_i\varpi_i)\D(us_i\varpi_i,v\varpi_i) + \D(u\lambda,v\lambda),
\allowdisplaybreaks\\
& q^{(v\varpi_i,vs_i\varpi_i-us_i\varpi_i)}\D(u\varpi_i,v\varpi_i)\D(us_i\varpi_i,vs_i\varpi_i) \allowdisplaybreaks\\
& \hspace{5ex} = q^{1+(vs_i\varpi_i,v\varpi_i-us_i\varpi_i)}\D(u\varpi_i,vs_i\varpi_i)\D(us_i\varpi_i,v\varpi_i) + \D(u\lambda,v\lambda)
\allowdisplaybreaks\\
& \hspace{5ex} = q^{1+(v\varpi_i,vs_i\varpi_i-u\varpi_i)}\D(us_i\varpi_i,v\varpi_i)\D(u\varpi_i,vs_i\varpi_i) + \D(u\lambda,v\lambda),
\end{align*}
where  $\lambda=s_i\varpi_i+\varpi_i$. % = -\sum_{j \ne i}a_{j,i}\varpi_j$.
\end{prop}
 Note that the difference of $\la$ and $\displaystyle-\sum_{j \ne i}a_{j,i}\varpi_j$ are $W$-invariant.
Hence we have
$\displaystyle \D(u\lambda,v\lambda) = \prod_{j \ne i}\D(u\varpi_j,v\varpi_j)^{-a_{j,i}}$ from Corollary \ref{cor:Duv},
 by disregarding  a power of $q$.

\subsection{Revisit of crystal bases and global bases}

In order to prove Theorem~\ref{thm: DDD} below, we first investigate
the upper crystal lattice of
$\Dv V$ induced by an upper  crystal lattice of $V\in\Oint(\g)$.

Let $V$ be a $U_q(\g)$-module in $\Oint (\g)$. Let $L^\up$ be an upper crystal lattice of $V$. Then we have (see Lemma~\ref{lem:lowup})
$$ \soplus_{\xi \in \wl} q^{(\xi,\xi)/2}(L^\up)_\xi \text{ is a lower crystal lattice of $V$.}$$

Recall that, for $\lambda\in \wl^+$, the upper crystal lattice $L^\up(\lambda)$ and the lower crystal lattice $L^\low(\lambda)$ of $V(\lambda)$ are related by
\begin{align}
L^\up(\lambda) = \soplus_{\xi \in \wl} q^{((\lambda,\lambda)-(\xi,\xi))/2}L^\low(\lambda)_\xi \subset L^\low(\lambda).
\end{align}

Write
$$V \simeq \soplus_{\lambda \in \wl^+} E_\lambda \otimes V(\lambda)$$
with finite-dimensional $\Q(q)$-vector spaces $E_\lambda$ .
Accordingly, we have a canonical decomposition
$$L^\up \simeq \soplus_{\lambda\in \wl^+} C_\lambda \otimes_{\QA_0} L^\up(\lambda),$$
where $C_\lambda \subset E_\lambda$ is an $\QA_0$-lattice of $E_\lambda$.

On the other hand, we have
$$\Dv V \simeq\soplus_{\lambda\in \wl^+}  E^*_\lambda \otimes V(\lambda).$$

Note that we have
$$ \Phi_V \left( (a \otimes u) \otimes (b \otimes v)^\ri \right) = \lan a,b \ra \Phi_\lambda(u \otimes v^\ri)
\quad \text{for $u,v \in V(\lambda)$ and $a \in E_\lambda$, $b \in E^*_\lambda$.}$$

We define the induced upper crystal lattice $\Dv L^\up$ of $\Dv V$ by
$$\Dv L^\up \seteq \soplus_{\lambda \in \wl^+} C_\lambda^\vee \otimes_{\QA_0} L^\up(\lambda)
\subset \Dv V, $$
where $C_\lambda^\vee \seteq  \set{ u \in E_\lambda^*}{\lan u,C_\lambda \ra \subset \QA_0 }$. Then we have
$$ \Phi_V \left(L^\up \otimes (\Dv L^\up)^\ri \right) \subset L^\up(A_q(\g)).$$
Indeed, we have
$$ \Dv L^\up =\set{ u \in \Dv V }{\Phi_V(L^\up \otimes u^\ri) \subset L^\up(A_q(\g)) }. $$

Since $(L^\up(\lambda))^\vee = L^\low(\lambda)$, we have
$$ (L^\up)^\vee = \soplus_{\lambda \in \wl^+} C_\lambda^\vee \otimes_{\QA_0} L^\low(\lambda).$$

The properties $L^\up(\lambda) \subset L^\low(\lambda)$ and $L^\up(\lambda)_\lambda = L^\low(\lambda)_\lambda$ imply the following lemma.

\begin{lemma} \label{lem: largest upper crystal}
$\Dv L^\up$ is the largest upper crystal lattice of $\Dv V$ contained in the lower crystal lattice $(L^\up)^\vee$.
\end{lemma}

Let $\la, \mu\in\Pd$.
Then
$\bl L^\up(\la)\tensp L^\up(\mu)\br^\vee
=L^\low(\la)\tensm L^\low(\mu)$ is a lower crystal lattice of
$\Dv\bl V(\la)\tensp V(\mu)\br\simeq V(\la)\tensm V(\mu)$.
Let
$\Xi_{\la,\mu}\cl V(\la)\tensp V(\mu)\isoto V(\la)\tensm V(\mu)
\simeq \Dv\bl V(\la)\tensp V(\mu)\br$
be the $\U$-module isomorphism defined by
$$\Xi_{\la,\mu}(u\tensp v)=q^{(\la,\mu)-(\xi,\eta)}\bl u\tensm v\br
\quad\text{for $u\in V(\la)_\xi$ and $v\in V(\mu)_\eta$.}
$$

Then
\eqn
\tLt&\seteq&\Xi_{\la,\mu}\bl L^\up(\la)\tensp L^\up(\mu)\br\\
&\;=&\soplus_{\xi,\eta\in\Po}
q^{(\la,\mu)-(\xi,\eta)}L^\up(\la)_\xi\tensm L^\up(\mu)_\eta
\eneqn
is an upper crystal lattice of $V(\la)\tensm V(\mu)$.
Since we have $(\la,\mu)-(\xi,\eta)\ge0$
for any $\xi\in\wt\bl V(\la)\br$ and $\eta\in\wt\bl V(\mu)\br$,
Lemma~\ref{lem: largest upper crystal} implies that
\eq \label{eq:tildeL}
\tLt\subset \Dv\bl L^\up(\la)\tensp L^\up(\mu)\br.
\eneq

\begin{lemma}
Let $\lambda, \mu \in \wl^+$ and $x_1,x_2,y_1,y_2 \in W$ such that $x_k \ge y_k$
 $(k=1,2)$. Then we have
\begin{equation} \label{eq: multi}
\begin{aligned}
& q^{(\lambda,\mu)-(x_2\lambda,y_2\mu)} \Delta(x_1\lambda,x_2\lambda)\Delta(y_1\mu,y_2\mu) \\
& \hspace{10ex} \equiv G^\up( \oi^{-1}_{\lambda,\mu}(u_{x_1\lambda} \otimes u_{y_1 \mu}) \otimes  \oi^{-1}_{\lambda,\mu}(u_{x_2\lambda} \otimes u_{y_2 \mu})^\ri )
\quad {\rm mod} \ q L^\up(A_q(\g)).
\end{aligned}
\end{equation}
\end{lemma}
\Proof
By the definition, we have
$$\De{x_1\la,x_2\la}\De{y_1\mu,y_2\mu}
=\Phi_{V(\la)\tensp V(\mu)}\bl
(u_{x_1\la}\tensp u_{y_1\mu})\tens (u_{x_2\la}\tensm u_{y_2\mu})^\rt\br.$$
Hence we have
\eqn
&&q^{(\la,\mu)-(x_2\la,y_2\mu)}
\De{x_1\la,x_2\la}\De{y_1\mu,y_2\mu}\\
&&\hs{3ex}=\Phi_{V(\la)\tensp V(\mu)}\bl
(u_{x_1\la}\tensp u_{y_1\mu})\tens q^{(\la,\mu)-(x_2\la,y_2\mu)}(u_{x_2\la}\tensm u_{y_2\mu})^\rt\br\\
&&\hs{6ex}=\Phi_{V(\la)\tensp V(\mu)}\bl
(u_{x_1\la}\tensp u_{y_1\mu})\tens
\bl\Xi_{\la,\mu}(u_{x_2\la}\tensp u_{y_2\mu})\br^\rt\br.
\eneqn
The right-hand side of \eqref{eq: multi} can be calculated as follows.
Let us take $v_k\in L^\up(\la+\mu)$ such that $\iota_{\la,\mu}(v_k)-u_{x_k\la}\tensp u_{y_k\mu}
\in qL^\up(\la)\tensp L^\up(\mu) $  for $k=1,2$.
Here $\iota_{\la,\mu}\cl V(\la+\mu)\to V(\la)\tensp V(\mu)$
denotes the canonical $\U$-module homomorphism and such a $v_k$ exists by Lemma \ref{lem:oi la mu}.

Then we have
\eqn
&&G^\up\Bigl(\oi_{\la,\mu}^{\ -1}(u_{x_1\la}\tens u_{y_1\mu})
\tens\bl  \oi_{\la,\mu}^{\ -1}(u_{x_2\la}\tens u_{y_2\mu})\br^\rt\Bigl)\\
&&\hs{5ex}\equiv \Phi_{\la+\mu}(v_1\tens v_2^\rt)\mod qL^\up(\Ag)\\
&&\hs{6ex}=\Phi_{V(\la)\tensp V(\mu)}\bl \iota_{\la,\mu}(v_1)\tens
\bl\Xi_{\la,\mu}\iota_{\la,\mu}(v_2))^\rt \br.
\eneqn
 The last equality follows from $(v_2, u) = (\Xi_{\la,\mu}\iota_{\la,\mu}(v_2), \iota_{\la,\mu} (u))$ for all $u \in V(\la+\mu)$.

On the other hand, we have
$$\iota_{\la,\mu}(v_1)\equiv
u_{x_1\la}\tensp u_{y_1\mu}\mod qL^\up(\la)\tensp L^\up(\mu)$$
and
$$\Xi_{\la,\mu}\bl\iota_{\la,\mu}(v_2)\br
\equiv \Xi_{\la,\mu}(u_{x_2\la}\tensp u_{y_2\mu})\mod q\tLt.
$$
Hence
\eqn
&&\Phi_{V(\la)\tensp V(\mu)}\bl
(u_{x_1\la}\tensp u_{y_1\mu})\tens \Xi_{\la,\mu}(u_{x_2\la}\tensp u_{y_2\mu})^\rt\br\\
&&\hs{10ex}\equiv
\Phi_{V(\la)\tensp V(\mu)}\bl (\iota_{\la,\mu}(v_1)\tens
(\Xi_{\la,\mu}\iota_{\la,\mu}(v_2))^\rt\br
\mod qL^\up(\Ag)
\eneqn
by \eqref{eq:tildeL}, as desired.
\QED

\begin{theorem} \label{thm: DDD}
Let $\lambda \in \wl^+$ and $x,y \in W$ such that $x \ge y$. Then we have
$$ \D(x\lambda,y\lambda)\D(y\lambda,\lambda) \equiv \D(x\lambda,\lambda) \quad {\rm mod} \ qL^\up(A_q(\n)).$$
\end{theorem}

\begin{proof}
Applying $p_\n$ to \eqref{eq: multi}, we have
\begin{equation*}  \label{eq: DDD}
\begin{aligned}
& \D(x\lambda,y\lambda)\D(y\lambda,\lambda) \\
& \hspace{10ex} \equiv p_\n
\left( G^\up( \oi^{-1}_{\lambda,\lambda}(u_{x\lambda} \otimes u_{y \lambda})  \otimes \oi^{-1}_{\lambda,\lambda}(u_{y\lambda} \otimes u_{\lambda})^\ri ) \right)
\quad {\rm mod} \ q L^\up(A_q(\n)).
\end{aligned}
\end{equation*}
Hence the desired result follows from
Proposition~\ref{prop: nonzero elt}, Proposition \ref{prop: p_n} and  Lemma \ref{lem: iotas} below.
\end{proof}

\Lemma\label{lem: iotas}
Let $\la\in \Pd$ and $x,y\in W$ such that $x\ge y$.
Then we have
$$
\bg\Bigl(
\oi_{\la,\la}^{\ -1}(u_{x\la}\tens u_{y\la})
\tens\bl  \oi_{\la,\la}^{\ -1}(u_{y\la}\tens u_{\la})\br^\rt\Bigr)
=\oi_\la(u_{x\la})\tens t_{y\la+\la}\tens b_{-\infty}.$$
\enlemma

\Proof
We shall argue by induction on $\ell(x)$.
We set $b_{x\la}=\oi_\la(u_{x\la})$.
Since the case $x=1$ is obvious, assume that $x\not=1$.
Take $i\in I$ such that $x'\seteq s_ix<x$

\medskip
\noi
(a) First assume that $s_iy>y$. Then we have
$y\le  x'$.
Hence by the induction hypothesis,
\eq
&&\oi_\g\Bigl(
\oi_{\la,\la}^{\ -1}(u_{x'\la}\tens u_{y\la})
\tens\bl  \oi_{\la,\la}^{\ -1}(u_{y\la}\tens u_{\la})\br^\rt\Bigr)
=b_{x'\la}\tens t_{y\la+\la}\tens b_{-\infty}.\label{eq:agtu}
\eneq
We have $\vphi_i(u_{x'\la})=\lan h_i,x'\la\ran$ and
$\vphi_i(b_{x'\la}\tens t_{y\la+\la}\tens b_{-\infty}) =\vphi_i(b_{x'\la}\tens t_{y\la+\la})=\lan h_i,x'\la\ran
+\lan h_i,y\la\ran\ge\lan h_i,x'\la\ran$.
Hence, applying $\tf_i^{\lan h_i,x'\la\ran}$ to \eqref{eq:agtu}, we obtain
\eqn
&&\oi_\g\Bigl(
\oi_{\la,\la}^{\ -1}(u_{x\la}\tens u_{y\la})
\tens\bl  \oi_{\la,\la}^{\ -1}(u_{y\la}\tens u_{\la})\br^\rt\Bigr)
=b_{x\la}\tens t_{y\la+\la}\tens b_{-\infty}.
\eneqn

\medskip
\noi
(b) Assume that $y'\seteq s_iy<y$. Then we have
$y'\le x'$, and the induction hypothesis implies that
$$
\oi_\g\Bigl(
\oi_{\la,\la}^{\ -1}(u_{x'\la}\tens u_{y'\la})
\tens\bl  \oi_{\la,\la}^{\ -1}(u_{y'\la}\tens u_{\la})\br^\rt\Bigr)
=b_{x'\la}\tens t_{y'\la+\la}\tens b_{-\infty}.$$
 Apply $\te_i^{*\;\lan h_iy'\la\ran}\tf_i^{\lan h_i,x'\la+y'\la\ran}$
to the both sides.
Then the left-hand side yields
$$\oi_\g\Bigl(
\oi_{\la,\la}^{\ -1}(u_{x\la}\tens u_{y\la})
\tens\bl  \oi_{\la,\la}^{\ -1}(u_{y\la}\tens u_{\la})\br^\rt\Bigr).$$
Since $\vphi_i(b_{x'\la}\tens t_{y'\la+\la})=\lan h_i,x'\la\ran
+\lan h_i,y'\la+\la\ran\ge\lan h_i,x'\la+y'\la\ran$,
the right-hand side yields
\begin{align*}
\te_i^{*\;\lan h_i,y'\la\ran}\tf_i^{\lan h_i,x'\la+y'\la\ran}
\bl b_{x'\la}\tens t_{y'\la+\la}\tens b_{-\infty}\br
&=\te_i^{*\;\lan h_i,y'\la\ran}\bl(\tf_i^{\lan h_i,x'\la+y'\la\ran}
 b_{x'\la})\tens t_{y'\la+\la}\tens b_{-\infty}\br\\
&=\te_i^{*\;\lan h_i,y'\la\ran}\bl(\tf_i^{\lan h_i,y'\la\ran}
 b_{x\la})\tens t_{y'\la+\la}\tens b_{-\infty}\br.
\end{align*}
Since
$\eps_i^*(b_{x\la})=-\vphi_i(b_{x\la}) = \lan h_i, \la\ran$
and
$\tf_i^{\lan h_i,y'\la\ran}b_{x\la}=\tf_i^{*\;\lan h_i,y'\la\ran}b_{x\la}$,
we have
\begin{align*}\te_i^{*\;\lan h_i,y'\la\ran}\bl(\tf_i^{\lan h_i,y'\la\ran}
 b_{x\la})\tens t_{y'\la+\la}\tens b_{-\infty}\br
&=b_{x\la}\tens t_{y\la+\la}\tens b_{-\infty}.
 \qedhere
\end{align*}
\QED

\subsection{Generalized T-system}

The $T$-system in \S \ref{subsubsec: T-system} can be interpreted as  a system of equations among the three products of elements
in $\B^\up(A_q(\g))$ or $\B^\up(A_q(\n))$.
In this subsection, we introduce  another among the three products
of elements in $\B^\up(A_q(\g))$, called {\em a generalized $T$-system}.

\begin{prop} \label{prop: g T-system D}
Let $\mu \in W \varpi_i$ and set $b = \oi_{\varpi_i}(u_\mu) \in B(\infty)$. Then we have
\begin{equation}\label{eq: G t-sys}
\begin{aligned}
\Delta(\mu,s_i\varpi_i)\Delta(\varpi_i,\varpi_i) & = q_i^{-1}
G^\up\left( \oi^{-1}_{\varpi_i,\varpi_i}(u_\mu \otimes u_{\varpi_i})
\otimes \big( \oi^{-1}_{\varpi_i,\varpi_i}(u_{s_i\varpi_i} \otimes u_{\varpi_i}) \big)^\ri \right) \\
& \hspace{10ex} +G^\up \left( \oi_{\varpi_i+s_i\varpi_i}^{-1}(\widetilde{e^*_i}b) \otimes u^\ri_{\varpi_i+s_i\varpi_i} \right).
\end{aligned}
\end{equation}
\end{prop}

Note that if $\mu = \varpi_i$, then $b=1$ and the last term in \eqref{eq: G t-sys} vanishes. If $\mu \ne \varpi_i$, then
$\ve_i^*(b)=1$ \ and $\oi_{\varpi_i+s_i\varpi_i}^{-1}(\te_i^* b) \in B(\varpi_i+s_i\varpi_i)$,  $u_\mu \otimes u_{\varpi_i} \in
\oi_{\varpi_i,\varpi_i}B(2\varpi_i)$.

\begin{proof}
In the sequel, we omit $\oi^{-1}_{\varpi_i,\varpi_i}$ for the sake of simplicity. Set
$$  u = \Delta(\mu,s_i\varpi_i)\Delta(\varpi_i,\varpi_i)- q_i^{-1}
G^\up\left( (u_\mu \otimes u_{\varpi_i}) \otimes
(u_{s_i\varpi_i} \otimes u_{\varpi_i})^\ri  \right).$$
 Then $\wt_\ri(u)=\lambda\seteq\varpi_i+s_i\varpi_i$.

It is obvious that we have $uf_j=0$ for $j \ne i$. Since $\te_i(u_{s_i\varpi_i} \otimes u_{\varpi_i})= u_{\varpi_i} \otimes u_{\varpi_i}$,
we have
\begin{align*}
G^\up\left( (u_\mu \otimes u_{\varpi_i}) \otimes (u_{s_i\varpi_i} \otimes u_{\varpi_i})^\ri  \right)f_i  & =
G^\up\left( (u_\mu \otimes u_{\varpi_i}) \otimes (u_{\varpi_i} \otimes u_{\varpi_i})^\ri  \right) \\
& = \Delta(\mu,\varpi_i)\Delta(\varpi_i,\varpi_i) \\
& = G^\up(u_\mu \otimes u^\ri_{\varpi_i})G^\up(u_{\varpi_i} \otimes u^\ri_{\varpi_i}).
\end{align*}
Here the second equality follows from Lemma~\ref{lem:Deprod} and the third follows from Proposition \ref{prop:Agglobal}.
On the other hand, we have
\begin{align*}
\left( \Delta(\mu,s_i\varpi_i)\Delta(\varpi_i,\varpi_i)   \right) f_i & =  \left( \Delta(\mu,s_i\varpi_i)f_i \right)
\left(\Delta(\varpi_i,\varpi_i)t_i^{-1}\right) \\ & = q_i^{-1}\Delta(\mu,\varpi_i)\Delta(\varpi_i,\varpi_i).
\end{align*}
Hence we have $uf_i=0$. Thus, $u$ is a lowest weight vector of wight $\lambda$ with respect to the right action of
$U_q(\g)$. Therefore there exists some $v \in V(\lambda)$ such that
$$u=\Phi(v \otimes u_\lambda^\ri).$$
Hence we have $p_\n(u)=\iota_\lambda(v) \in A_q(\n)$. On the other hand, we have
\begin{align*}
p_\n\left( \Delta(\mu,s_i\varpi_i)\Delta(\varpi_i,\varpi_i) \right) & =
p_\n\left(\Delta(\mu,s_i\varpi_i) \right) p_\n\left(\Delta(\varpi_i,\varpi_i)\right) \\
& = \D(\mu,s_i\varpi_i)=G^{\up}(\te_i^*b) \\
& = \iota_\lambda\left(G_\lambda^\up \big(\oi_\lambda^{-1}(\te_i^* b) \big) \right)
\end{align*}
Note that since $\ve_i^*(\te_i^*b) =0$ and $\ve_j^*(\te_i^*b) \le -\lan h_j,\alpha_i \ra$ for $j \ne i$, we have
$\te_i^*b \in \oi_\lambda(B(\lambda))$.

Hence in order to prove our assertion, it is enough to show that
$$
p_\n\left( G^\up\big( (u_\mu \otimes u_{\varpi_i}) \otimes (u_{s_i\varpi_i} \otimes u_{\varpi_i})^\ri  \big) \right) =0.
$$
This follows from Proposition \ref{prop: nonzero elt} and
\begin{align} \label{eq: kernel}
\oi_\g\left( (u_\mu \otimes u_{\varpi_i}) \otimes (u_{s_i\varpi_i} \otimes u_{\varpi_i})^\ri \right) = b \otimes t_\lambda \otimes \te_ib_{-\infty}.
\end{align}

Let us prove \eqref{eq: kernel}. Since
$$
(u_\mu \otimes u_{\varpi_i}) \otimes (u_{s_i\varpi_i} \otimes u_{\varpi_i})^\ri
= \te_i^*
\bl(u_\mu \otimes u_{\varpi_i}) \otimes (u_{\varpi_i} \otimes u_{\varpi_i})^\ri\br,
$$
the left-hand side of \eqref{eq: kernel} is equal to
$$
\te_i^*\left(\oi_\g\big( (u_\mu \otimes u_{\varpi_i}) \otimes (u_{\varpi_i} \otimes u_{\varpi_i})^\ri \big)\right)
= \te_i^*(b \otimes t_{2\varpi_i}\otimes b_{-\infty}).
$$

Since $\ve_i^*(b)=1 < \lan h_i,2\varpi_i \ra=2$, we obtain
\begin{align*}
&\te_i^*\left(b \otimes t_{2\varpi_i}\otimes b_{-\infty}\right)
= b \otimes t_{2\varpi_i-\al_i}\otimes \te_i^* b_{-\infty}
=b \otimes t_\lambda \otimes \te_ib_{-\infty}.  \qedhere
\end{align*}
\end{proof}

\section{KLR algebras and their modules}

\subsection{Chevalley and Kashiwara operators}
 Let us recall the definition of several functors on modules over \KLRs \ which  are used to categorify $\Um_{\Z[q^\pm 1]}^\vee$.
\begin{definition} Let $\beta\in\rtl^+$.
\bnum
\item For $i \in I$ and $1 \le a \le |\beta|$, set
$$ e_a(i)=\sum_{\nu \in I^\beta,\nu_a=i}e(\nu)\in R(\beta).$$
\item We take conventions:
\begin{align*}
&E_iM=e_1(i)M, \allowdisplaybreaks\\
&E^*_iM=e_{|\beta|}(i)M,
\end{align*}
which are functors from $R(\beta) \gmod$ to $R(\beta-\al_i) \gmod$.
\item For a simple module $M$, we set
\begin{align*}
&\eps_i(M)=\max\set{n\in\Z_{\ge0}}{E_i^nM\not=0},\\
&\eps^*_i(M)=\max\set{n\in\Z_{\ge0}}{E_i^{*\;n}M\not=0},\\
& \widetilde{F}_iM = q_i^{\ve_i(M)} L(i) \hconv M, \allowdisplaybreaks\\
& \widetilde{F}^*_iM = q_i^{\ve^*_i(M)} M \hconv L(i), \allowdisplaybreaks\\
& \widetilde{E}_iM = q_i^{1-\ve_i(M)} \soc(E_iM) \simeq q_i^{\ve_i(M)-1} \hd(E_iM), \allowdisplaybreaks\\
& \tE^*_iM = q_i^{1-\ve^*_i(M)} \soc(E^*_iM) \simeq q_i^{\ve^*_i(M)-1}
\hd(E_i^*M),\\
&\tE_i^{\,\max} M=\tE_i^{\eps_i(M)}M\quad\text{and}\quad
\tE_i^{*\;\max} M=\tE_i^{*\;\eps^*_i(M)}M.
\end{align*}
\item For $i \in I$ and $n \in \Z_{\ge 0}$, we set
$$L(i^n)=q_i^{n(n-1)/2} \underbrace{L(i)\conv \cdots \conv L(i)}_n.$$
Here $L(i)$ denotes the $R(\al_i)$-module $R(\al_i)/R(\al_i)x_1$.
Then  $L(i^n)$ is a self-dual real simple  $R(n\alpha_i)$-module.
\ee
\end{definition}
 Note that, under the isomorphism in Theorem \ref{thm:categorification 1}, the functors $E_i$ and $E_i^*$ correspond to the linear operators $e_i$ and $e_i^*$  on
 $ \An_{\Z[q^{\pm 1}]} =  \iota(\Um_{\A}^\vee) \subset \An$, respectively.
Note also that,  for a simple $R(\beta)$-module $S$, we have
$\tE_i \tF_i S \simeq S$, and $\tF_i \tE_i S \simeq S$ if $\eps_i(M) >0$.

In the course of  proving the following propositions, we use the following notations.

\begin{align} \label{eq: overQ}
\overline{Q}_{i,j}(x_a,x_{a+1},x_{a+2}) \seteq
\dfrac{Q_{i,j}(x_a,x_{a+1})-Q_{i,j}(x_{a+2},x_{a+1})}{x_a-x_{a+2}}.
\end{align}
Then we have
$$\tau_{a+1}\tau_a\tau_{a+1}-\tau_{a}\tau_{a+1}\tau_{a}
=\sum_{i,j\in I}\overline{Q}_{i,j}(x_a,x_{a+1},x_{a+2})
e_a(i)e_{a+1}(j)e_{a+2}(i).$$
\begin{prop} \label{prop: La i}
Let $\beta \in \rl^+$ with $n=|\beta|$. Assume that an $R(\beta)$-module $M$ satisfies $E_iM=0$. Then the left $R(\al_i)$-module homomorphism $R(\alpha_i) \tens M \Lto
q^{(\alpha_i,\beta)}M \conv R(\alpha_i)$ given by
\begin{align} \label{eq: pre-form}
e(i) \tens u \longmapsto \tau_1 \cdots \tau_n(u \tens e(i))
\end{align}
extends uniquely to an $(R(\alpha_i+\beta),R(\alpha_i))$-bilinear homomorphism
\begin{align} \label{eq: ext-form}
R(\alpha_i) \conv M \Lto q^{(\alpha_i,\beta)}M \conv R(\alpha_i)
\end{align}
\end{prop}

\begin{proof} {\rm (i)} First note that, for $1 \le a \le n$,
\begin{equation} \label{eq: e_a(i) vanish}
\tau_1 \cdots \tau_{a-1}e_a(i)\tau_{a+1}\cdots\tau_n\big( u \tens e(i) \big) =
\tau_{a+1}\cdots\tau_n\big(e_1(i)\tau_1 \cdots \tau_{a-1}(u \tens e(i))  \big)=0
\end{equation}
since $E_iM =0$. \\
\noindent
{\rm (ii)} In order to see that \eqref{eq: ext-form} is a well-defined $R(\alpha_i+\beta)$-linear homomorphism, it is enough
to show that \eqref{eq: pre-form} is $R(\beta)$-linear. \\
\noindent
{\rm (a)} Commutation with $x_a\in R(\beta)$ $(1\le a \le n)$: We have
\begin{align*}
x_{a+1}\tau_1 \cdots \tau_n\big(u \tens e(i)\big) &=  \tau_1 \cdots \tau_{a-1}x_{a+1}\tau_{a} \cdots \tau_n\big(u \tens e(i)\big) \allowdisplaybreaks \\
 &=  \tau_1 \cdots \tau_{a-1}\big( \tau_{a} x_{a}+ e_a(i) \big) \tau_{a+1}\cdots \tau_n\big(u \tens e(i) \big) \allowdisplaybreaks \\
 &=  \tau_1 \cdots  \tau_n  x_{a}\big(u \tens e(i)\big)
\end{align*}
by \eqref{eq: e_a(i) vanish}.  \\
\noindent
{\rm (b)} Commutation with $\tau_a\in R(\beta)$ $(1\le a < n)$:
  We have
\begin{align*}
&\tau_{a+1}\tau_1 \cdots \tau_n\big(u \tens e(i)\big) \\
& \hspace{5ex} =  \tau_1 \cdots \tau_{a-1}(\tau_{a+1}\tau_{a}\tau_{a+1})\tau_{a+2} \cdots \tau_n\big(u \tens e(i)\big) \allowdisplaybreaks \\
& \hspace{5ex} =  \tau_1 \cdots \tau_{a-1}\big(\tau_{a}\tau_{a+1}\tau_{a}+\sum_{j}\overline{Q}_{i,j}(x_{a},x_{a+1},x_{a+2})e_a(i)e_{a+1}(j) \big)\tau_{a+2}\cdots \tau_n\big(u \tens e(i) \big) \allowdisplaybreaks \\
& \hspace{5ex} =  \tau_1 \cdots  \tau_n  \tau_{a}\big(u \tens e(i)\big) \allowdisplaybreaks \\
& \hspace{10ex} + \sum_{j}\tau_1 \cdots \tau_{a-1}\overline{Q}_{i,j}(x_{a},x_{a+1},x_{a+2})e_a(i)e_{a+1}(j)\tau_{a+2}\cdots \tau_n\big(u \tens e(i) \big).
\end{align*}
The last term vanishes because $E_iM=0$ implies
\begin{align*}
& \tau_1 \cdots \tau_{a-1}f(x_{a},x_{a+1})g(x_{a+2})e_a(i)\tau_{a+2}\cdots \tau_n\big(u \tens e(i) \big) \\
& \hspace{18ex} = g(x_{a+2})\tau_{a+2}\cdots \tau_n e_1(i)\tau_1 \cdots \tau_{a-1}f(x_{a},x_{a+1})\big(u \tens e(i) \big)=0
\end{align*}
for any polynomial $f(x_{a},x_{a+1})$ and $g(x_{a+2})$. \\
\noindent
{\rm (iii)} Now let us show that \eqref{eq: ext-form} is right $R(\alpha_i)$-linear. By \eqref{eq: e_a(i) vanish}, we have
\begin{align*}
 \tau_1 \cdots \tau_{a-1} x_a \tau_a \cdots \tau_{n}\big(u \tens e(i) \big) & =
\tau_1 \cdots \tau_{a-1} \big(\tau_a x_{a+1}-e_{a}(i) \big)\tau_{a+1}\cdots \tau_{n} \big(u \tens e(i) \big) \\
& = \tau_1 \cdots \tau_{a} x_{a+1} \tau_{a+1} \cdots \tau_{n}\big(u \tens e(i) \big)
\end{align*}
for $1 \le a \le n$. Therefore we have
$$ x_{1}\tau_1 \cdots \tau_n\big(u \tens e(i)\big) = \tau_1 \cdots \tau_n x_{n+1}\big(u \tens e(i)\big)=\tau_1 \cdots \tau_n \big(u \tens e(i)x_{1}\big). \qedhere$$
\end{proof}

 Recall that  for $m,n\in\Z_{\ge0}$,
we denote by $w[{m,n}]$ the element of $\sym_{m+n}$  defined by
\eq
&&w[{m,n}](k)=\begin{cases}k+n&\text{if $1\le k\le m$,}\\
k-m&\text{if $m<k\le m+n$}.\end{cases}
\eneq
Set $\tau_{w[{m,n}]}:=\tau_{i_1} \cdots \tau_{i_r}$, where $s_{i_1} \cdots s_{i_r}$ is a reduced expression
of $w[{m,n}]$.
Note that $\tau_{w[{m,n}]}$ does not depend on the choice of reduced expression (\cite[Corollary 1.4.3]{K^3}).

\begin{prop} Let $M \in R(\beta) \gmod$  and $N \in R(\gamma) \gmod$, and set $m=|\beta|$ and
$n=|\gamma|$. If $E_iM=0$ for any $i \in \supp(\gamma)$, then
\begin{align*}
v \tens u \longmapsto \tau_{w[m,n]}(u \tens v)
\end{align*}
gives a well-defined $R(\beta+\gamma)$-linear homomorphism $N \conv M \Lto q^{(\beta,\gamma)} M \conv N$.
\end{prop}

\begin{proof} The proceeding proposition implies that
$$ v \tens u \longmapsto \tau_{w[m,n]}(u \tens v) \quad \text{ for } u \in M, v \in R(\gamma)$$
gives a well-defined $R(\beta+\gamma)$-linear homomorphism $R(\gamma) \conv M \to M \conv R(\gamma)$. Hence
it is enough to show that it is right $R(\gamma)$-linear. Since we know that it commutes with
the right multiplication of $x_k$, it is enough to show that it commutes with the right
multiplication of $\tau_k$. For this, we may assume that $n = 2$ and $k = 1$. Set $\gamma = \alpha_i +\alpha_j$.

Thus we have reduced the problem to  the equality
\begin{align*}
\tau_1(\tau_2\tau_1)\cdots(\tau_{m+1}\tau_m)\big( u \tens e(i) \tens e(j) \big)
= (\tau_2\tau_1)\cdots(\tau_{m+1}\tau_m) \tau_{m+1}\big( u \tens e(i) \tens e(j) \big)
\end{align*}
for $u \in M$, which is a consequence of
\begin{equation*}
\begin{aligned}
& (\tau_2\tau_1)\cdots(\tau_{a}\tau_{a-1})\tau_a(\tau_{a+1}\tau_{a})\cdots(\tau_{m+1}\tau_m)\big( u \tens e(i) \tens e(j) \big) \allowdisplaybreaks \\
& \hs{5ex} = (\tau_2\tau_1)\cdots(\tau_{a+1}\tau_{a})
\tau_{a+1}(\tau_{a+2}\tau_{a+1})\cdots(\tau_{m+1}\tau_m)\big( u \tens e(i) \tens e(j) \big)
\end{aligned}
\end{equation*}
for $1 \le a \le m$. Note that
\begin{align*}
& \tau_a(\tau_{a+1}\tau_{a})\cdots(\tau_{m+1}\tau_m)\big( u \tens e(i) \tens e(j) \big) \allowdisplaybreaks \\
& \hspace{10ex} = \tau_a(\tau_{a+1}\tau_{a})e_{a+1}(i)e_{a+2}(j)(\tau_{a+2}\tau_{a+1})\cdots(\tau_{m+1}\tau_m)\big( u \tens e(i) \tens e(j) \big)
\end{align*}
and
\begin{align*}
& \tau_a(\tau_{a+1}\tau_{a})e_{a+1}(i)e_{a+2}(j) \allowdisplaybreaks \\
& \hspace{10ex} = (\tau_{a+1}\tau_{a})\tau_{a+1} e_{a+1}(i)e_{a+2}(j)-\overline{Q}_{ji}(x_{a},x_{a+1},x_{a+2})e_{a}(j)e_{a+1}(i)e_{a+2}(j).
\end{align*}
Hence it is enough to show
\begin{align*}
& (\tau_2\tau_1)\cdots(\tau_{a}\tau_{a-1})\overline{Q}_{j,i}(x_{a},x_{a+1},x_{a+2})
e_a(j)
\allowdisplaybreaks\\
& \hspace{20ex}(\tau_{a+2}\tau_{a+1})\cdots (\tau_{m+1}\tau_{m})\big( u \tens e(i) \tens e(j) \big)=0.
\end{align*}
This follows from
\begin{align*}
&(\tau_2\tau_1)\cdots(\tau_{a}\tau_{a-1})f(x_a)g(x_{a+1},x_{a+2}) e_a(j)
(\tau_{a+2}\tau_{a+1})\cdots(\tau_{m+1}\tau_m)
\big( u\tens e(i)\tens e(j)\big)\allowdisplaybreaks\\
& \hspace{10ex}=(\tau_2\cdots\tau_a)
(\tau_1\cdots\tau_{a-1})f(x_a)g(x_{a+1},x_{a+2}) e_a(j) \allowdisplaybreaks\\
&\hspace{35ex}(\tau_{a+2}\tau_{a+1})\cdots(\tau_{m+1}\tau_m)
\big( u\tens e(i)\tens e(j)\big)\allowdisplaybreaks\\
&\hspace{10ex}=(\tau_2\cdots\tau_a)
g(x_{a+1},x_{a+2})(\tau_{a+2}\tau_{a+1})\cdots(\tau_{m+1}\tau_m)\allowdisplaybreaks\\
&\hspace{35ex}
e_1(j)(\tau_1\cdots\tau_{a-1})f(x_a)\big( u\tens e(i)\tens e(j)\big)\allowdisplaybreaks\\
&\hspace{10ex}=0
\end{align*}
for $1 \le a \le m$ and $f(x_a) \in \ko[x_{a}]$, $g(x_{a+1},x_{a+2}) \in \ko[x_{a+1},x_{a+2}]$.
\end{proof}

Let $P(i^n)$ be a projective cover of
$L(i^n)$.
 Define  the functor
$$E_i^{(n)} : R(\beta) \Mod \to R(\beta - n \alpha_i) \Mod$$
 by
$$E_i^{(n)}(M) := P(i^n)^\psi \tens_{R(n\al_i)}E_i^nM, $$
where
$P(i^n)^\psi $ denotes the right $R(n\al_i)$-module obtained from the left
$R(\beta)$-module $P(i^n)$  via the anti-automorphism $\psi$.
 We define the functor $E_i^{*\,(n)}$ in a similar way.
 Note that $$E_i^n\simeq  [n]_i ! E_i^{(n)}.$$

\Cor Let $R$ be a symmetric \KLR.
Let $i\in I$ and $M$ a simple module.
Then we have
\eqn
&&\tLa\bl L(i),M\br=\eps_i(M),\\
&&\La\bl L(i),M\br=2\eps_i(M)+\lan h_i,\wt(M)\ran=\eps_i(M)+\vphi_i(M).\\
\eneqn
\encor
\Proof
Set $n=\eps_i(M)$ and $M_0=E_i^{(n)}(M)$.
Then the preceding proposition implies
$\La(L(i),M_0)=\bl \al_i,\wt(M_0)\br$.
Hence we have $\tLa(L(i),M_0)=0$, which implies
$$\tLa(L(i),M)=\tLa(L(i), L(i^n)\circ M_0)
=\tLa\bl L(i), L(i^n)\br+\tLa(L(i),M_0)=n.$$
\QED

\begin{prop}\label{prop: Econv} Let $M$, $N$ be modules
and $m,n \in \Z_{\ge 0}$.
\begin{enumerate}
\item[{\rm (i)}] If $E_i^{m+1}M=0$ and $E_i^{n+1}N=0$, then we have
\begin{align*} %\label{eq: shuffle 1}
E_i^{(m+n)} (M \conv N) \simeq q^{mn+n \lan h_i, \wt(M) \ran}
 E_i^{(m)} M \conv  E_i^{(n)}N.
\end{align*}
\item[{\rm (ii)}] If $E_i^{*\,m+1}M=0$ and $E_i^{*\,n+1}N=0$, then we have
\begin{align*} %\label{eq: shuffle 2}
E_i^{*\,(m+n)} (M \conv N) \simeq
q^{mn+m \lan h_i, \wt(N) \ran} E_i^{*\,(m)} M \conv E_i^{*\,(n)} N.
\end{align*}
\end{enumerate}
\end{prop}

\begin{proof}
 Our assertions follow from the shuffle lemma (\cite[Lemma 2.20]{KL09}).
\end{proof}

The following corollaries are  immediate consequences of Proposition~\ref{prop: Econv}.
\Cor\label{cor:Ereal} Let $i\in I$ and let $M$ be a real simple module.
Then $\tE_i^{\,\max}M$ is also real simple.
\encor

\Cor
Let $i\in I$ and let $M$ be a  simple module with $\eps_i(M)=m$ .
Then we have $\tE_i^{m} M \simeq E_i^{(m)} M$.
\encor

\begin{prop} \label{prop: acted on left only}
Let $M$ and $N$ be simple modules.
 We assume that one of them is real.
If $\ve_i(M \hconv N)=\ve_i(M)$,
then we have an isomorphism in $R\gmod$
$$ \widetilde{E}^{\,\max}_i (M \hconv N) \simeq (\widetilde{E}^{\,\max}_i M) \hconv N.$$
Similarly, if  $\ve^*_i(N \hconv M)=\ve^*_i(M)$, then we have
$$ \widetilde{E}^{*\max}_i (N \hconv M) \simeq (N \hconv \widetilde{E}^{*\max}_i M).$$
\end{prop}

\begin{proof}
Set $n=\ve_i(M \hconv N) = \ve_i(M)$ and $M_0=\widetilde{E}^{\,\max}_i M$.
Then $M_0$ or $N$ is real.
Now we have
$$ L(i^n) \otimes M_0 \otimes N \monoto
E_i^n(M\hconv N)\simeq L(i^n) \otimes \widetilde{E}^{\,\max}_i (M \hconv N),$$
which induces a non-zero map $M_0 \otimes N \to
\widetilde{E}^{\,\max}_i (M \hconv N)$. Hence we have a surjective map
$$M_0 \conv N \twoheadrightarrow \widetilde{E}^{\,\max}_i (M \hconv N).$$
Since $M_0$ or $N$ is real by Corollary~\ref{cor:Ereal}, $M_0\conv N$ has a simple head
and we obtain the desired result.  A similar proof works for the second statement.
\end{proof}

\subsection{Determinantial modules and T-system}
We will use the materials in \S\,\ref{sec:Quantum minors and T-systems} to obtain properties on the determinantial modules.

In the rest of this paper,
we assume  that {\em  $R$ is symmetric and the base field $\ko$ is of characteristic $0$.}
Under this condition, the family of self-dual simple $R$-modules
corresponds to the upper global basis of $\An$
by Theorem \ref{thm:categorification 2}.

Let $\ch$ be the map from $K(R \gmod)$ to $\An$ obtained by composing $\iota$ and the isomorphism  \eqref{eq:KLRU} in Theorem \ref{thm:categorification 1}.

\begin{definition}
For $\lambda \in \wl^+$ and $\mu,\zeta \in W \lambda$ such that $\mu \preceq \zeta$, let $\M(\mu,\zeta)$ be a simple $R(\zeta-\mu)$-module
such that $\ch(\M(\mu,\zeta))=\D(\mu,\zeta)$.
\end{definition}
Since $\D(\mu,\zeta)$ is a member of the upper global basis,
such a module exists uniquely due to Theorem \ref{thm:categorification 2}.
The module $\M(\mu,\zeta)$ is self-dual and we call it {\em the determinantial module}.

\begin{lemma}
$\M(\mu,\zeta)$ is a real simple module.
\end{lemma}

\begin{proof}
It follows from $\ch(\M(\mu,\zeta) \conv \M(\mu,\zeta)) =\ch\big( \M(\mu,\zeta) \big)^2=q^{-(\zeta,\zeta-\mu)}\D(2\mu,2\zeta)$ which is a member of the upper global
basis up to a power of $q$.
Here the last equality follows from Corollary~\ref{cor:Duv}.
\end{proof}

\begin{prop} \label{prop: commute 1}
Let $\lambda,\mu \in \wl^+$, and $s,s',t,t' \in W$ such that $\ell(s's)=\ell(s')+\ell(s)$,
$\ell(t't)=\ell(t')+\ell(t)$, $s's \lambda \preceq t'\lambda $ and
$s'\mu \preceq t't\mu$. Then
\bnum
\item $\M(s's\lambda,t'\lambda)$ and $\M(s'\mu,t't\mu)$ commute,
\item $\La\bl\M(s's\lambda,t'\lambda),\M(s'\mu,t't\mu)\br
= (s's\lambda+t'\lambda,\;t't\mu-s'\mu)$,
\item $\tLa\bl\M(s's\lambda,t'\lambda),\M(s'\mu,t't\mu)\br
=(t'\lambda,\;t't\mu-s'\mu)$, \\
$\tLa\bl\M(s'\mu,t't\mu),\M(s's\lambda,t'\lambda)\br
=(s'\mu-t't\mu,\;s's\lambda)$.
\end{enumerate}
\end{prop}

\begin{proof}
It is a consequence of Proposition \ref{prop: BZ form}~(ii) and
Corollary \ref{cor:qcomm_module}.
\end{proof}

\begin{prop} \label{prop: weyl left right}
Let $\lambda \in \wl^+$, $\mu,\zeta \in W \lambda$
such that $\mu \preceq \zeta$ and $i \in I$.
\begin{enumerate}
\item[{\rm (i)}] If $n \seteq \lan h_i,\mu \ra \ge 0$, then
$$ \ve_i(\M(\mu,\zeta))=0 \ \ \text{ and } \ \ \M(s_i\mu,\zeta) \simeq \widetilde{F}_i^n\M(\mu,\zeta) \simeq L(i^n) \hconv \M(\mu,\zeta)
\ \text{in $R\gmod$.}$$
\item[{\rm (ii)}] If $\lan h_i,\mu \ra \le 0$ and $s_i\mu \preceq \zeta$, then $\ve_i(\M(\mu,\zeta))=-\lan h_i,\mu \ra$.
\item[{\rm (iii)}] If $m \seteq -\lan h_i,\zeta \ra  \ge  0$, then
$$ \ve^*_i(\M(\mu,\zeta))=0 \ \ \text{ and } \ \ \M(\mu,s_i\zeta) \simeq
\tF_i^{*\,m}\M(\mu,\zeta) \simeq \M(\mu,\zeta) \hconv L(i^m )
\ \text{ in $R\gmod$.}$$
\item[{\rm (iv)}] If $\lan h_i,\zeta \ra \ge 0$ and $\mu \preceq s_i\zeta$, then $\ve^*_i(\M(\mu,\zeta))=\lan h_i,\zeta \ra$.
\end{enumerate}
\end{prop}
\Proof
It is a consequence of Lemma~\ref{lem: weyl left right}.
\QED

\begin{prop}\label{prop:T}
Assume that $u,v \in W$ and $i \in I$ satisfy $u<us_i$ and $v<vs_i \le u$.
\bnum
\item
We have exact sequences
\begin{equation} \label{eq: deteminatial seq1}
\begin{aligned}
& 0 \Lto \M(u\lambda,v\lambda) \Lto q^{(vs_i\varpi_i,v\varpi_i-u\varpi_i)}\M(us_i\varpi_i,vs_i\varpi_i) \conv \M(u\varpi_i,v\varpi_i) \\
& \hspace{14ex} \Lto q^{-1+(v\varpi_i,vs_i\varpi_i-u\varpi_i)}\M(us_i\varpi_i,v\varpi_i) \conv \M(u\varpi_i,vs_i\varpi_i) \Lto 0,
\end{aligned}
\end{equation}
and
\begin{equation} \label{eq: deteminatial seq2}
\begin{aligned}
& 0 \Lto  q^{1+(v\varpi_i,vs_i\varpi_i-u\varpi_i)}  \M(us_i\varpi_i,v\varpi_i) \conv \M(u\varpi_i,vs_i\varpi_i)  \\
& \hspace{3ex} \Lto   q^{(v\varpi_i,vs_i\varpi_i-us_i\varpi_i)}  \M(u\varpi_i,v\varpi_i)\conv \M(us_i\varpi_i,vs_i\varpi_i) \Lto \M(u\lambda,v\lambda)  \Lto 0,
\end{aligned}
\end{equation}
where  $\lambda = s_i\varpi_i+\varpi_i$. %= -\sum_{j \ne i} a_{j,i}\varpi_j$.
\item
$\de\bl \M(u\varpi_i,v\varpi_i), \M(us_i\varpi_i,vs_i\varpi_i)\br=1$.
\ee
\end{prop}

\begin{proof}
Since the proof of \eqref{eq: deteminatial seq1} is similar,
let us only prove \eqref{eq: deteminatial seq2}.
(Indeed, they are dual to each other.)

Set
\eqn
&&X=q^{(v\varpi_i,vs_i\varpi_i-u\varpi_i)}  \M(us_i\varpi_i,v\varpi_i) \conv \M(u\varpi_i,vs_i\varpi_i) ,\\
&&Y=q^{(v\varpi_i,vs_i\varpi_i-us_i\varpi_i)}  \M(u\varpi_i,v\varpi_i)\conv \M(us_i\varpi_i,vs_i\varpi_i),\\
&&Z=\M(u\lambda,v\lambda).
\eneqn
Then Proposition \ref{prop: the ses} implies that
$$\ch(Y)=\ch(qX)+\ch(Z).$$
Since $X$ and $Z$ are simple and self-dual,
 our assertion follows from
Lemma~\ref{lem:short}.
\QED

\subsection{Generalized T-system on determinantial module}

\begin{theorem} \label{thm: canonical surjection}
Let $\lambda \in \wl^+$ and $\mu_1,\mu_2,\mu_3 \in W \lambda$ such that $\mu_1 \preceq \mu_2 \preceq \mu_3$. Then there exists a canonical epimorphism
$$ \M(\mu_1,\mu_2) \conv \M(\mu_2,\mu_3) \twoheadrightarrow \M(\mu_1,\mu_3),$$
which is equivalent to saying that $\M(\mu_1,\mu_2) \hconv \M(\mu_2,\mu_3) \simeq \M(\mu_1,\mu_3)$.

In particular, we have
$$\tLa(\M(\mu_1,\mu_2),\M(\mu_2,\mu_3))=0\quad\text{and}\quad
\La(\M(\mu_1,\mu_2),\M(\mu_2,\mu_3)) = -(\mu_1-\mu_2,\mu_2-\mu_3). $$
\end{theorem}

\begin{proof}
(a)\ Our assertion follows from Theorem \ref{thm: DDD} and Theorem \ref{th:head}
%\cite[Theorem 3.6]{KKKO15} From Mon I
 when $\mu_3=\lambda$.

\medskip\noi
(b) We shall prove the general case by induction on $|\la-\mu_3|$.
By (a), we may assume that $\mu_3 \ne \lambda$. Then there exists $i$ such that $\lan h_i,\mu_3\ra <0$. The induction hypothesis  yields that
$$\M(\mu_1,\mu_2) \hconv \M(\mu_2,s_i \mu_3) \simeq \M(\mu_1,s_i \mu_3).$$
Since $\mu_1\preceq\mu_2\preceq\mu_3\preceq s_i\mu_3$,
Proposition~\ref{prop: weyl left right} (iv)  gives
$$\eps_i^*\bl\M(\mu_2,s_i\mu_3)\br=\eps_i^*\bl\M(\mu_1,s_i\mu_3)\br
=-\lan h_i,\mu_3\ran.$$
Then Proposition~\ref{prop: acted on left only} implies that
$$\tEs_i{}^\max\bl \M(\mu_1,\mu_2)\hconv \M(\mu_2,s_i\mu_3)\br
\simeq
\M(\mu_1,\mu_2)\hconv\bl \tEs_i{}^\max \M(\mu_2,s_i\mu_3)\br,$$
from which we obtain
$$\M(\mu_1,\mu_3)\simeq \M(\mu_1,\mu_2)\hconv  \M(\mu_2,\mu_3)$$

\medskip\noi
By  Lemma \ref{lem:selfdual},  we have
$\tLa\bl\M(\mu_1,\mu_2),\M(\mu_2,\mu_3)\br=0$. Hence we  obtain
\begin{align*}
&\La\bl\M(\mu_1,\mu_2),\M(\mu_2,\mu_3)\br
=-\bl\wt(\M(\mu_1,\mu_2),\wt(\M(\mu_2,\mu_3))\br.
\qedhere
\end{align*}
\end{proof}

\begin{prop}\label{prop:dMM}
Let $i \in I$ and $x,y,z \in W$.
\begin{enumerate}
\item[{\rm (i)}] If $\ell(xy)=\ell(x)+\ell(y)$, $zs_i>z$,
$xy\ge zs_i$
and $x \ge z$, then we have
$$ \de(\M(xy\varpi_i,zs_i\varpi_i),\M(x\varpi_i,z\varpi_i)) \le 1.$$
\item[{\rm (ii)}] If $\ell(zy)=\ell(z)+\ell(y)$, $xs_i >  x$, $xs_i\ge zy$ and $x \ge z$, then we have
$$ \de(\M(xs_i\varpi_i,zy\varpi_i),\M(x\varpi_i,z\varpi_i)) \le 1.$$
\end{enumerate}
\end{prop}

\begin{proof}
In the course of proof, we omit $\oi_{\vpi,\vpi}^{-1}$ for the sake of simplicity.
If $y\vpi = \vpi$, then the assertion follows from Proposition \ref{prop: commute 1} (i). Hence we may assume that
 $y' \seteq ys_i < y$.

\medskip

Let us show {\rm (i)}. By Proposition \ref{prop: g T-system D}, we have
\begin{equation} \label{eq: ac le 1}
\begin{aligned}
\Delta(y\vpi,s_i\vpi)\Delta(\vpi,\vpi) & =q^{-1}G^\up\left( (u_{y\vpi} \otimes u_{\vpi}) \otimes (u_{s_i\vpi} \otimes u_{\vpi})^\ri \right) \\
& \hspace{10ex} +G^\up(\oi_\lambda^{-1}(\te_i^*b) \otimes u_\lambda^\ri),
\end{aligned}
\end{equation}
where $\lambda=\vpi+s_i\vpi$ and $b=\oi_\vpi(u_{y\vpi}) \in B(\infty)$.
Let  $S_{z,\la}^* $  be the operator on $\Ag$ given by
the application of $e_{j_1}^{(a_1)}\cdots e_{j_t}^{(a_t)}$ from the right,
where $z=s_{j_t}\cdots s_{j_1}$ is a reduced expression of $z$ and
$a_k=\lan h_{j_k},s_{j_{k-1}}\cdots s_{j_1}\la\ran$.
Then applying  $S_{z,\la}^*$   to \eqref{eq: ac le 1}, we obtain
\begin{align*}
\Delta(y\vpi,zs_i\vpi)\Delta(\vpi,z\vpi) &  =q^{-1}G^\up\left( (u_{y\vpi} \otimes u_{\vpi} ) \otimes (u_{zs_i\vpi} \otimes u_{z\vpi})^\ri \right) \\
& \hspace{10ex} +G^\up(\oi_\lambda^{-1}(\te_i^*b) \otimes u_{z\lambda}^\ri).
\end{align*}

Recall that $\mu\in\Po$ is called {\em $x$-dominant}
if $c_k\ge0$.
Here $x=s_{i_r} \cdots s_{i_1}$ is a reduced expression of $x$ and
 $c_k\seteq \lan h_{i_k}, s_{i_{k-1}} \cdots s_{i_1} \mu \ra$ ($1\le k\le r$).
Recall that an element $v\in\Ag$ with $\wtl(v)=\mu$ is called
{\em $x$-highest}
if $\mu$ is $x$-dominant and
$$\text{$f_{i_k}^{1+c_k}f_{i_{k-1}}^{(c_{k-1})}\cdots f_{i_1}^{(c_{1})}v=0$ for any $k$
($1\le k\le r$).}$$
If $v$ is $x$-highest, then
$v$ is a linear combination of
$x$-highest $G^\up(b)$'s.
Moreover, $S_{x,\mu}G^\up(b)\seteq f^{(c_r)}_{i_r}\cdots f^{(c_1)}_{i_1} G^\up(b)$ is either a member of the upper global basis or zero.
Since $\Delta(y\vpi,zs_i\vpi)\Delta(\vpi,z\vpi)$ is $x$-highest
of weight $\mu\seteq  y\vpi+\vpi$,
we obtain
\begin{align*}
\Delta(xy\vpi,zs_i\vpi)\Delta(x\vpi,z\vpi) & =q^{-1}G^\up\left( (u_{xy\vpi} \otimes u_{x\vpi} ) \otimes (u_{zs_i\vpi} \otimes u_{z\vpi})^\ri \right) \\
& \hspace{10ex} +S_{x,\mu}G^\up(\oi_\lambda^{-1}(\te_i^*b) \otimes u_{z\lambda}^\ri).
\end{align*}

Applying $p_\n$, we obtain
\begin{align*}
q^c \D(xy\vpi,zs_i\vpi)\D(x\vpi,z\vpi) & =q^{-1}p_\n G^\up\left( (u_{xy\vpi} \otimes u_{x\vpi} ) \otimes (u_{zs_i\vpi} \otimes u_{z\vpi})^\ri \right) \\
& \hspace{10ex} +p_\n S_{x,\mu}G^\up(\oi_\lambda^{-1}(\te_i^*b) \otimes u_{z\lambda}^\ri)
\end{align*}
for some integer $c$. Hence we obtain {\rm (i)} by
Lemma~\ref{lem:short} (i).

\medskip\noi
{\rm (ii)} is proved similarly. By applying $\varphi^*$  to \eqref{eq: ac le 1}, we obtain
\begin{align*}
q^{(s_i\vpi_i,\vpi_i)-(y\vpi_i,\vpi_i)}\Delta(s_i\vpi,y\vpi)\Delta(\vpi,\vpi) & =q^{-1}G^\up\left( (u_{s_i\vpi} \otimes u_{\vpi}) \otimes (u_{y\vpi} \otimes u_{\vpi})^\ri \right) \\
& \hspace{10ex} +G^\up(u_\lambda  \otimes (\oi_\lambda^{-1}\te_i^*b)^\ri).
\end{align*}
 Here we used Proposition~\ref{prop:phimul}.
Then the similar arguments as above show {\rm (ii)}.
\end{proof}

\begin{prop}\label{prop:deMM}
Let $x \in W$ such that $xs_i > x$ and $x \vpi \ne \vpi$. Then we have
$$\de(\M(xs_i\vpi,x\vpi),\M(x\vpi,\vpi)) =1.$$
\end{prop}

\begin{proof}
 By Proposition~\ref{prop:dMM} (ii),
 we have $\de(\M(xs_i\vpi,x\vpi),\M(x\vpi,\vpi)) \le1$.
Assuming $\de(\M(xs_i\vpi,x\vpi),\M(x\vpi,\vpi))=0$, let us derive a contradiction.

By Theorem~\ref{thm: canonical surjection} and the assumption, we have
$$\M(xs_i\vpi,x\vpi) \conv \M(x\vpi,\vpi) \simeq \M(xs_i\vpi,\vpi).$$
Hence we have
$$ \ve_j^*(\M(xs_i\vpi,\vpi))= \ve_j^*(\M(xs_i\vpi,x\vpi))+\ve_j^*(\M(x\vpi,\vpi))$$
for any $j \in I$. Since $xs_i\vpi \preceq x\vpi \preceq s_i\vpi$, Proposition \ref{prop: weyl left right} implies that
$$ \ve_j^*(\M(xs_i\vpi,\vpi))=\ve_j^*(\M(x\vpi,\vpi))=\lan h_j,\vpi \ra.$$
It implies that
$$\ve_j^*(\M(xs_i\vpi,x\vpi))=0 \quad \text{ for any } j \in I.$$
It is a contradiction  since $\wt\bl\M(xs_i\vpi,x\vpi)\br=xs_i\vpi-x\vpi$ does not vanish.
\end{proof}

\section{Monoidal categorification of $A_q(\n(w))$}

\subsection{Quantum cluster algebra structure on $A_q(\n(w))$} \label{subsubsec: Q cluster Anw}
In this subsection, we shall consider the Kac-Moody algebra
$\g$ associated with a symmetric Cartan matrix $\cmA=(a_{i,j})_{i,j\in I}$.
We shall recall briefly the definition of the subalgebra $A_q(\n(w))$ of $A_q(\g)$ and its quantum cluster algebra structure by using
the results of \cite{GLS} and \cite{Kimu12}.
Remark that we bring the results in \cite{GLS}
through the  isomorphism \eqref{eq:isoAqn}.
\medskip

For a given $w \in W$, fix a reduced expression $\widetilde{w}=s_{i_r}\cdots s_{i_1}$.

For $s \in \{ 1,\ldots,r \}$ and $j \in I$, we set
\begin{equation*} \label{eq: s+ s-}
\begin{aligned}
s_+&\seteq\min( \{k \ | \ s<k\le r,\; i_k=i_s \}\cup\{r+1\}),\\
s_-&\seteq\max(\{k \ | \  1\le k<s, \; i_k=i_s \}\cup\{0\}),\\
s^-(j)&\seteq\max(\{k \ | \ 1\le k<s,\; i_k=j \}\cup\{0\}).
\end{aligned}
\end{equation*}

We set
\eq \text{$u_k:=s_{i_1}\cdots s_{i_k}$ for $0\le k \le r$}, \label{eq:lau}
\eneq
and
\eqn
\text{$\la_k\seteq u_k\varpi_{i_k}$ for $1\le k \le r$}.
\eneqn

Note that $\la_{k_- }=u_{k-1}\varpi_{i_k}$,  if $k_- >0$ .
For $0\le t\le s\le r$, we set
\begin{align*} %\label{eq: D(s,t)}
\D(s,t)&=
\bc
\D(\lambda_s,\lambda_t) & \text{if $0<t$,}\\
\D(\lambda_s,\varpi_{i_s})& \text{if $0=t<s\le r$,}\\
\one&\text{if $t=s=0$.}
\ec
\end{align*}

 The $\Q(q)$-subalgebra of  $A_q(\n)$ generated by
$\D(i,i_-)$ ($1\le i\le r$) is independent of the choice of a reduced expression of $w$. We denote it by $A_q(\n(w))$.
Then every $\D(s,t)$ ($0\le t\le s\le r$) is contained in $A_q(\n(w))$
(\cite[Corollary 12.4]{GLS}).
The set $\B^\up\bl A_q(\n(w))\br\seteq\B^\up\bl A_q(\g)\br\cap A_q(\n(w))$
is a basis of $A_q(\n(w))$ as a $\Q(q)$-vector space
(\cite[Theorem 4.2.5]{Kimu12}).
We call it the \emph{upper global basis} of $A_q(\n(w))$.
 We denote by $A_q(\n(w))_\A$ the $\A$-module generated by $\B^\up\bl A_q(\n(w))$.
Then it is a $\A$-subalgebra of $A_q(\n(w))$ (\cite[\S 4.7.2]{Kimu12}).
We set $A_{q^{1/2}}(\n(w)) \seteq \Q(q^{1/2})  \otimes_{\Q(q)}  A_{q}(\n(w))$.

\medskip
 Let $\K=\{1,\ldots,r\}$,
$\Kfr\seteq\set{k\in\K}{k_+=r+1}$ and $\Kex\seteq\K\setminus\Kfr$.
\begin{definition} We define the quiver $Q$
with the set of vertices $Q_0$ and the set of
arrows $Q_1$ as follows:
\begin{enumerate}
\item[$(Q_0)$]
$Q_0=\K=\{1,\ldots,r\}$,
\item[$(Q_1)$] There are two types of arrows:
\eqn
&&\ba{lccl}\text{{\em ordinary arrows}}&:&\text{$s\To[\ |a_{i_s,i_t}|\ ] t$}
&\text{if $1\le s<t<s_+<t_+\le r+1$,}\\[1ex]
\text{{\em horizontal arrows}}&:&
\text{$s\longrightarrow s_-$}&
\text{if $1\le s_-<s\le r$.}
\ea
\eneqn
\end{enumerate}
Let $\widetilde B = (b_{i,j})$ be the integer-valued $\K\times\Kex$-matrix associated to the quiver $Q$ by \eqref{eq: bij}.
\end{definition}

\begin{lemma}
Assume that $0 \le d\le b\le a\le c \le r$ and
\begin{itemize}
\item $i_b=i_a$  when   $b \ne 0$,
\item $i_d=i_c$  when  $d \ne 0$.
\end{itemize}
Then $\D(a,b)$ and $\D(c,d)$ $q$-commute; that is, there exists $\lambda \in \Z$ such that
$$ \D(a,b)\D(c,d)= q^{\lambda}\D(c,d)\D(a,b).$$
\end{lemma}

\begin{proof}
We may assume $a>0$.
Let $u_k$ be as in \eqref{eq:lau}.
Take $s'=u_a$, $s=u_a^{-1}u_c$, $t'=u_d$ and $t=u_d^{-1}u_b$. Then we have
$$ \D(s'\varpi_{i_a},t't\varpi_{i_a})=\D(a,b)\quad \text{and}\quad
  \D(s's\varpi_{i_c},t'\varpi_{i_c})=\D(c,d).$$
From Proposition \ref{prop: BZ form}, our assertion follows.
\end{proof}

Hence we have an integer-valued skew-symmetric matrix $L=(\lambda_{i,j})_{1 \le i,j \le r}$ such that
$$ \D(i,0)\D(j,0)=q^{\lambda_{i,j}}\D(j,0)\D(i,0).$$

\begin{prop} [{\cite[Proposition 10.1]{GLS}}] The pair $(L,\wB)$ is compatible with $d=2$ in \eqref{eq:compatible}.
\end{prop}

\begin{theorem}[{\cite[Theorem 12.3]{GLS}}] Let
$\mathscr{A}_{q^{1/2}}([\mathscr{S}])$ be the quantum cluster algebra associated to the initial quantum seed
$[\mathscr{S}]\seteq( \{ q^{-(d_s,d_s)/4}\D(s,0) \}_{1 \le s \le r}, L, \wB )$.
Then we have an  isomorphism of  $\Q(q^{1/2})$-algebras
$$ \Q(q^{1/2})  \otimes_{\Z[q^{\pm 1/2}]} \mathscr{A}_{q^{1/2}}([\mathscr{S}]) \simeq   A_{q^{1/2}}(\n(w)),$$
where $d_s \seteq \lambda_s-\varpi_{i_s}=\wt(D(s,0))$ and
$A_{q^{1/2}}(\n(w)) \seteq \Q(q^{1/2})  \otimes_{\Q(q)}  A_{q}(\n(w))$.
\end{theorem}

\subsection{Admissible seeds in the monoidal category $\shc_w$}

 For $0\le t\le s\le r$, we set
$\M(s,t)=\M(\la_s,\la_t)$.
It is a real simple module with $\ch(\M(s,t))=D(s,t)$.

\begin{definition} \label{def:cw}
For $w \in W$, let $\shc_w$ be the smallest monoidal abelian full subcategory of $R \gmod$ satisfying the following properties:
\begin{enumerate}
\item[{\rm (i)}] $\shc_w$ is stable under the subquotients, extensions and grading shifts,
\item[{\rm (ii)}] $\shc_w$ contains $\M(s,s_-)$ for all $1 \le s \le \ell(w)$.
\end{enumerate}
\end{definition}

Then by \cite{GLS}, $M\in R\gmod$ belongs to $\shc_w$ if and only if $\ch(M)$
belongs to $A_q(\n(w))$. Hence we have  a $\A$-algebra isomorphism
$$K(\shc_w)\simeq A_q(\n(w))_\A.$$

We set
$$\Lambda \seteq (\La(\M(i,0),\M(j,0)))_{1 \le i,j\le r} \quad \text{ and }
\quad D= (d_i)_{1 \le i \le r} \seteq (\wt(\M(i,0)))_{1 \le i \le r}.$$

Then, by Proposition~\ref{prop: commute 1},
$\Seed \seteq ( \{ \M(k,0) \}_{1 \le k \le r}, -\Lambda, \wB, D )$
is a quantum monoidal seed in $\shc_w$.
We are now ready to state the main theorem in this section:

\begin{theorem} \label{thm: main}
The pair
$\bl \{ \M(k,0) \}_{1 \le k \le r}, \wB\br$ is admissible.
\end{theorem}
As we already explained,  combined with Theorem  \ref{th:main} and Corollary \ref{cor:main}, this theorem implies
the following  theorem.
\begin{theorem}
The category $\shc_w$ is a monoidal categorification of the quantum cluster algebra $A_{q^{1/2}}(\n(w))$.
\end{theorem}

In the course of  proving Theorem~\ref{thm: main},
{\em we omit grading shifts} if there is no  danger of confusion.

We shall start the proof of Theorem~\ref{thm: main}
by proving that, for each $s\in\Kex$,  there exists a simple module $X$ such that
\begin{eqnarray}&&\left\{
\parbox{70ex}{ \bnam
\item there exists a surjective homomorphism (up to a grading shift)
$$  X \conv \M(s,0) \twoheadrightarrow \conv_{t;\; b_{t,s}> 0 } \M(t,0)^{\circ b_{t,s} },$$
\item there exists a surjective homomorphism (up to a grading shift)
$$ \M(s,0) \conv X \twoheadrightarrow \conv_{t;\;\ b_{t,s}<0} \M(t,0)^{\circ -b_{t,s}},$$
\item $\de(X,\M(s,0))=1$.
\ee }\right. \label{eq: conditions}
\end{eqnarray}

We set
\eqn
&&x\seteq i_s\in I,\\
&&I_s\seteq\set{i_k}{s<k<s_+}\subset I\setminus\{x\},\\
&&A\seteq
  \displaystyle \mathop{\conv}\limits_{t<s < t_+ < s_+ }
\M(t,0)^{\circ |a_{i_s,i_t}| }
= \displaystyle\mathop{\conv} \limits_{y\in I_s}
\M(s^-(y),0)^{\circ |a_{x,y}| }.
\eneqn
Then $A$ is a real simple module.

Now we claim that the following simple module $X$ satisfies
the conditions in \eqref{eq: conditions}:
$$X \seteq \M(s_+,s) \hconv  A.$$

Let us show \eqref{eq: conditions} (a).
The incoming arrows to $s$ are
\begin{itemize}
\item $t \To[{|a_{x,i_t}|}]s $ for $1\le t<s<t_+<s_+$,
\item $s_+ \To s$.
\end{itemize}
Hence we have
$$\conv_{t;\;  b_{t,s} > 0 } \M(t,0)^{\circ b_{t,s}}
\simeq A\conv M(s_+,0).$$

\medskip

Then the morphism in (a)
is obtained  as the composition:
\begin{align} \label{eq: surjection 1}
X \conv \M(s,0)  \rightarrowtail A \conv \M(s_+,s) \conv \M(s,0) \twoheadrightarrow A \conv \M(s_+,0).
\end{align}
Here the second epimorphism
is given in Theorem \ref{thm: canonical surjection},  and
Lemma \ref{lem:composition} asserts that the composition
\eqref{eq: surjection 1} is non-zero
and hence an epimorphism.

\medskip

Let us show \eqref{eq: conditions} (b).
The outgoing arrows from $s$ are
\begin{itemize}
\item $s \To[{\;|a_{x,i_t}|\;}]t$\quad for $s<t<s_+<t_+\le r+1$.
\item $s \longrightarrow s_-$\quad if $s_- > 0$.
\end{itemize}
Hence we have
\eq
&&\mathop{\conv}_{t; b_{t,s}<0 } \M(t,0)^{\circ -b_{t,s}}
\simeq\M(s_-,0)\conv \left(\mathop{\conv}_{y\in I_s} \M((s_+)^-(y),0)^{\circ -a_{x,y}}\right).
\eneq

\begin{lemma} \label{lem: step1}
There exists an epimorphism  \ro up to a grading\rf
\begin{align*} %\label{eq: Omega}
\Omega: \M(s,0) \conv \M(s_+,s) \conv A \twoheadrightarrow \conv_{t;b_{t,s}<0} \M(t,0)^{\circ -b_{t,s}}.
\end{align*}
\end{lemma}

\begin{proof}
By the dual of Theorem \ref{thm: canonical surjection} and the $T$-system
\eqref{eq: deteminatial seq2}
with $i=i_s$, $u=u_{s_+-1}$ and $v=u_{s-1}$, we have morphisms
\eqn
&& \M(s,0) \rightarrowtail \M(s_-,0) \conv \M(s,s_-), \label{eq: 1} \\
&&\M(s,s_-)\conv \M(s_+,s) \twoheadrightarrow
\displaystyle \conv_{y \in I\setminus\{x\}} \M((s_+)^-(y),s^-(y))^{\circ -a_{x,y}}\\\label{eq: 2}
&&\hs{35ex}\simeq \displaystyle \conv_{y \in I_s}
\M((s_+)^-(y),s^-(y))^{\circ -a_{x,y }}.  \nn
\eneqn
Here the last isomorphism follows from the fact that
$(s_+)^-(y)=s^-(y)$ for any $y\not\in \{x\}\cup I_s=\set{i_k}{s\le k<s_+}$.

Thus we have a sequence of morphisms
\eqn
&&\M(s,0) \conv \M(s_+,s) \conv A\; \xymatrix{\ar@{>->}[r]^-{\varphi_1}&}
 \M(s_-,0) \conv \M(s,s_-) \conv \M(s_+,s) \conv A  \\
&&\hs{25ex}\xymatrix@C=6ex{\ar@{->>}[r]^-{\varphi_2} &}\M(s_-,0) \conv
\left(\conv_{y \in I_s} \M((s_+)^-(y),s^-(y))^{\circ -a_{x,y}}\right) \conv A.
\eneqn
By Lemma \ref{lem:composition} (i), the composition
$\varphi \seteq \varphi_2 \circ \varphi_1$ is non-zero.

Since
$A=  \displaystyle\conv_{ \substack{y\in I_s} }  \M(s^-(y),0)^{\circ -a_{x,y} }$,
Theorem \ref{thm: canonical surjection} gives the  morphisms
\begin{align*}
 \M(s,0) \conv \M(s_+,s) \conv A   & \xymatrix@C=4ex{\ar[r]^-{\varphi}&}
   \M(s_-,0) \conv
\left(\conv_{y\in I_s} \M((s_+)^-(y),s^-(y))^{\circ -a_{x,y}}\right) \conv A  \\
&\xymatrix@C=4ex{\ar@{->>}[r]^-{\phi} &}\M(s_-,0)\conv \left(\conv_{y\in I_s} \M((s_+)^-(y),0)^{\circ -a_{x,y}}\right)\simeq \conv_{t; b_{t,s}<0} \M(t,0)^{\circ -b_{t,s}}.
\end{align*}
Here we have used Lemma~\ref{lem:MN} to obtain the morphism $\phi$.
Note that  the module $\conv_{y \in I_s} \M((s_+)^-(y),s^-(y))^{\circ -a_{x,y}}$
is  simple.
By applying Lemma \ref{lem:composition}
 once again, $\phi \circ \varphi$ is non-zero, and hence it is an epimorphism.
\end{proof}

\Lemma\label{lem: simply linked}
We have $\de(X,\M(s,0))=1$.
\enlemma

\begin{proof}
Since $A$ and $\M(s,0)$ commute and
$\de\big(\M(s_+,s),\M(s,0)\big)=1$ by Proposition~\ref{prop:deMM}, we have
$$\de\big( X,\M(s,0) \big) \le\de\big(\M(s_+,s),\M(s,0)\big)+
\de\big(A,\M(s,0)\big)\le1$$
by Proposition \ref{pro:subquotient}
and Lemma \ref{lem:commute_equiv}.
If $X$ and $\M(s,0)$ commute, then \eqref{eq: conditions} (a)
 would imply
that ${\rm ch} \left(\conv_{t;\ b_{t,s}>0}  M(t,0)^{\circ b_{t,s}} \right)$
belongs to $K(R \gmod)\,{\rm ch}(\M(s,0))$.
It contradicts the result in \cite{GLS13} that all the ${\rm ch}(\M(k,0))$'s are
prime at $q=1$.
\end{proof}

\begin{prop} \label{prop: step2}
The map  $\Omega$
factors through $\M(s,0) \conv X$; that is,
\begin{align*} \label{eq: surjection 2}
\xymatrix{
 \M(s,0) \conv \M(s_+,s) \conv A  \ar@{->>}[drr]^{\tau}\ar@{->>}[rrrr]^{\Omega} &&&&  \displaystyle \conv_{t; b_{t,s}<0} \M(t,0)^{\circ  -b_{t,s}}. \\
 && \M(s,0) \conv X \ar@{->>}[urr]^{\overline{\Omega}}
}
\end{align*}
Here $\tau$ is the canonical surjection.
\end{prop}

\begin{proof}
We have $1=\de\big(\M(s,0), \M(s_+,s)\hconv A\big)$
by Lemma~\ref{lem: simply linked}, and
$$\de\big(\M(s,0),\M(s_+,s)\big)+
\de\big(\M(s,0),A\big)=1$$
by Proposition~\ref{prop:deMM} with $x=u_{s_+ -1}, \ i = i_s$.
Hence $ \M(s,0) \conv \M(s_+,s) \conv A
$ has a simple head by Proposition~\ref{prop:3simple} (iii).
\end{proof}

\begin{proof}[End of the proof of Theorem \ref{thm: main}]
By the   above arguments, we have proved the existence of $X$ which satisfies
\eqref{eq: conditions}.
By Proposition \ref{Prop: l2} and \eqref{eq: conditions} (c),
$\M(s,0) \conv X$ has composition
length $2$. Moreover, it has a simple socle and simple head.
On the other hand, taking the dual of \eqref{eq: conditions} (a),
we obtain a monomorphism
$$\sodot_{t; b_{t,s}>0} \M(t,0)^{\snconv  b_{t,s}} \monoto \M(s,0)\conv X$$
in $R\smod$.
Together with \eqref{eq: conditions} (b),
there exists a short exact sequence in $R\gmod$:
$$ 0 \to q^c\sodot_{t;b_{t,s}>0} \M(t,0)^{\snconv b_{t,s}} \to q^{\tLa(\M(s,0), X)}  \M(s,0) \conv X \to \sodot_{t;b_{t,s}<0} \M(t,0)^{\snconv(-b_{t,s})} \to 0$$
for some $c\in \Z$. By Lemma \ref{lem:crde}
$c$
must be equal to $1$.

It remains to prove that $X$ commutes with $\M(k,0)$ ($k\not=s$).
For any $k\in\K$, we have
\begin{align*}
\La(\M(k,0),X)&=\La(\M(k,0),\M(s,0)\hconv X)-\La(\M(k,0), \M(s,0))\\
&=\sum_{t;\;b_{t,s}<0}\La(\M(k,0),\M(t,0))(-b_{t,s})-\La(\M(k,0), \M(s,0))
\end{align*}
and
\begin{align*}
\La(X,\M(k,0))&=\La(X\hconv\M(s,0),\M(k,0))-\La(\M(s,0), \M(k,0))\\
&=\sum_{t;\;b_{t,s}>0}\La(\M(t,0),\M(k,0))b_{t,s}-\La(\M(s,0), \M(k,0)).
\end{align*}
Hence we have
\begin{align*}
2\de(\M(k,0),X)&  =-2\de(\M(k,0),\M(s,0)) -
\sum_{t;\;b_{t,s}<0}\La(\M(k,0),\M(t,0))b_{t,s} \\
& \hspace{23.5ex}-\sum_{t;\;b_{t,s}>0}\La(\M(k,0),\M( t  ,0))b_{t,s} \\
& =- \sum_{1 \le t \le r}\La(\M(k,0),\M(t,0))b_{t,s} \\ & = 2\delta_{k,s},
\end{align*}
We conclude that $X$  commutes with $M(k,0)$ if $k\not=s$.
Thus we complete the proof of Theorem~\ref{thm: main}.
\end{proof}

As a corollary, we  prove the following conjecture
%obtain the following answer to the conjecture
on the cluster monomials.
\Th  {\rm(}{\cite[Conjecture 12.9]{GLS}, \cite[Conjecture 1.1(2)]{Kimu12}} {\rm )}
%{\rm Conjecture~\ref{conj:intro}} in {\rm Introduction} is true, i.e.,
 Every cluster variable in  $A_{q}(\n(w))$ is a member of the upper global basis  up to a power of $q^{1/2}$.
\enth
Theorem \ref{thm: main} also implies \cite[Conjecture 12.7]{GLS} in the refined form as follows:

\begin{corollary}
$\Z[q^{\pm1/2}]\tens_{\Z[q^{\pm1}]} A_{q}(\n(w))_{\Z[q^{\pm1}]}$ has a quantum cluster algebra structure associated with the initial quantum seed
$[\mathscr{S}]=( \{ q^{-(d_i,d_i)/4}\D(i,0) \}_{1 \le i \le r}, L, \wB )$; i.e.,
$$\Z[q^{\pm1/2}]\tens_{\Z[q^{\pm1}]} A_{q}(\n(w))_{\Z[q^{\pm1}]} \simeq \mathscr{A}_{q^{1/2}}([\Seed]).$$
\end{corollary}

\end{document}